\documentclass[12 pt]{article}
\usepackage[pdftex]{graphicx}
\usepackage{amsmath}
\usepackage[hidelinks]{hyperref}
\usepackage{mathtools}
\usepackage{amsthm}
\usepackage{amsfonts}
\usepackage{amssymb}
\usepackage[margin=1.0in]{geometry}
\usepackage{tikz-cd}
\usepackage{enumitem}
\usepackage{esvect}
\usepackage{url}

\usepackage{imakeidx}

\newtheorem{theorem}{Theorem}[subsection]
\newtheorem{cor}[theorem]{Corollary}
\newtheorem{lemma}[theorem]{Lemma}
\newtheorem{prop}[theorem]{Proposition}
\newtheorem{prob}[theorem]{Problem}
\newtheorem{conj}[theorem]{Conjecture}

\newtheorem*{theorem*}{Theorem}

\theoremstyle{definition}
\newtheorem{defin}[theorem]{Definition}
\newtheorem{notation}[theorem]{Notation}
\newtheorem{fact}[theorem]{Fact}
\newtheorem{exa}[theorem]{Example}

\theoremstyle{remark}
\newtheorem*{rem}{Remark}

\renewcommand{\b}[1]{\mathbf{#1}}

\renewcommand{\rm}[1]{\mathrm{#1}}
\renewcommand{\sf}[1]{\mathsf{#1}}
\renewcommand{\frak}[1]{\mathfrak{#1}}
\renewcommand{\c}[1]{\mathcal{#1}}

\newcommand{\bA}{\mathbf{A}}

\newcommand{\cA}{\mathcal{A}}

\newcommand{\bB}{\mathbf{B}}

\newcommand{\cB}{\mathcal{B}}

\newcommand{\bC}{\mathbf{C}}

\newcommand{\cC}{\mathcal{C}}

\newcommand{\bD}{\mathbf{D}}

\newcommand{\cD}{\mathcal{D}}

\newcommand{\bE}{\mathbf{E}}

\newcommand{\cE}{\mathcal{E}}

\newcommand{\bF}{\mathbf{F}}

\newcommand{\cF}{\mathcal{F}}

\newcommand{\bG}{\mathbf{G}}

\newcommand{\cG}{\mathcal{G}}

\newcommand{\bH}{\mathbf{H}}

\newcommand{\cH}{\mathcal{H}}

\newcommand{\bK}{\mathbf{K}}

\newcommand{\cK}{\mathcal{K}}

\newcommand{\cL}{\mathcal{L}}

\newcommand{\bM}{\mathbf{M}}

\newcommand{\cM}{\mathcal{M}}

\newcommand{\sfP}{\mathsf{P}}

\newcommand{\sfQ}{\mathsf{Q}}

\newcommand{\sfR}{\mathsf{R}}

\newcommand{\bS}{\mathbf{S}}

\newcommand{\bT}{\mathbf{T}}

\newcommand{\sfU}{\mathsf{U}}

\newcommand{\bX}{\mathbf{X}}

\newcommand{\cX}{\mathcal{X}}

\newcommand{\bY}{\mathbf{Y}}

\newcommand{\bZ}{\mathbf{Z}}

\newcommand{\flim}{\mathrm{Flim}}
\newcommand{\age}{\mathrm{Age}}
\newcommand{\fin}{\mathrm{Fin}}
\newcommand{\fr}{Fra\"iss\'e }

\renewcommand{\phi}{\varphi}
\newcommand{\emb}{\mathrm{Emb}}
\newcommand{\aut}{\mathrm{Aut}}

\newcommand{\dom}{\mathrm{dom}}

\newcommand{\im}{\mathrm{Im}}

\newcommand{\cupdots}{\cup\cdots\cup}

\newcommand{\wt}{\widetilde}

\newcommand{\ol}{\overline}
\newcommand{\cd}{\mathrm{Cd}}
\newcommand{\crit}{\mathrm{Crit}}
\newcommand{\critnd}{\mathrm{CritNd}}
\newcommand{\Left}{\mathrm{Left}}
\renewcommand{\sp}{\mathrm{Sp}}
\newcommand{\ac}{\mathrm{AC}}

\newcommand{\con}{\mathrm{Con}}
\newcommand{\econ}{\mathrm{ECon}}
\newcommand{\lex}{\preceq_{lex}}

\renewcommand{\succ}{\mathrm{Succ}}
\newcommand{\is}{\mathrm{IS}}

\newcommand{\str}{\mathbf{Str}}

\newcommand{\ct}{\mathrm{CT}}

\newcommand{\full}{\mathsf{Full}}

\newcommand{\aemb}{\mathrm{AEmb}}
\newcommand{\oemb}{\mathrm{OEmb}}

\newcommand{\demb}{\mathrm{DEmb}}

\newcommand{\la}{\langle}
\newcommand{\ra}{\rangle}

\newcommand{\sort}{\mathrm{Sort}}

\newcommand{\pair}{\mathrm{Pair}}

\newcommand{\pcon}{\mathrm{PCon}}
\renewcommand{\frak}[1]{\mathfrak{#1}}
\renewcommand{\path}{\mathsf{Path}}
\newcommand{\argpath}{\mathrm{Path}}
\newcommand{\MP}{\mathsf{MP}}

\newcommand{\spnd}{\mathrm{SpNd}}
\newcommand{\cdnd}{\mathrm{CdNd}}
\newcommand{\shp}{\mathrm{Shp}}

\makeindex

\begin{document}
\bibliographystyle{amsplain}

	\title{Exact big Ramsey degrees for finitely constrained binary free amalgamation classes}
	\author{Martin Balko, David Chodounsk\'{y}, Natasha Dobrinen, Jan Hubi\v{c}ka, \\ Mat\v{e}j Kone\v{c}n\'{y}, Llu\'{i}s Vena, and Andy Zucker}
	
 
 
 
 
	\maketitle

\begin{abstract}
     We characterize the big Ramsey
degrees of free amalgamation classes in finite binary languages defined by
finitely many forbidden irreducible substructures, thus refining the recent upper bounds given by Zucker. Using this characterization, we show that the \fr limit of each such class admits a strong big Ramsey structure, implying that the automorphism group of the \fr limit has a metrizable universal completion flow. 
\end{abstract}

\section{Introduction}

Set-theoretic and model-theoretic notation throughout the paper is mostly standard. We identify $m< \omega$  with the set $\{0,\dots,m-1\}$, though sometimes we will write the latter for emphasis. If $m< n< \omega$, we write $[m, n] = \{r< \omega: m\leq r\leq n\}$, and similarly for $(m, n)$, $(m, n]$, and $[m, n)$. Given a  first-order language $\cL$, we typically denote $\cL$-structures in bold letters and use the un-bolded letter to denote the underlying set, i.e.\ $A, B, C$ are the underlying sets of $\bA$, $\bB$, $\bC$, etc. If $\bA$ is an $\cL$-structure and $B\subseteq A$, then $\bA|_B$ denotes the $\cL$-structure on underlying set $B$ induced from $\bA$. Given $\cL$-structures $\mathbf{A}$ and $\mathbf{B}$, an \emph{embedding} of $\bA$ into $\bB$ is an injection from $A$ to $B$ which preserves relations, non-relations, functions, and constants. 
We write $\emb(\b{A}, \b{B})$ for the set of embeddings of $\b{A}$ into $\b{B}$, and we write $\mathbf{A}\le \mathbf{B}$ when $\emb(\b{A}, \b{B})\neq \emptyset$. 

\begin{defin}
\label{Def:Ramsey}    
Given structures  $\mathbf{A},\mathbf{B},\mathbf{C}$ with $\mathbf{A}\le \mathbf{B}\le \mathbf{C}$
and integers $r> \ell\geq 1$, we write
\begin{equation*}
\mathbf{C}\to (\mathbf{B})^{\mathbf{A}}_{r, \ell}
\end{equation*}
if for every $\gamma\colon\emb(\bA, \bC)\to r$, there is some $g\in \emb(\bB, \bC)$ with $|\gamma[g\circ \emb(\bA, \bB)]|\leq \ell$. When $\ell = 1$, we omit it from the notation.  

Given a class $\mathcal{K}$ of finite $\cL$-structures and $\b{A}\in \c{K}$, the \emph{Ramsey degree} of $\bA$ in $\cK$, denoted by $\rm{RD}(\bA, \cK)$, is the least $\ell< \omega$ (or $\infty$ if no such $\ell$ exists) such that for any integer  $r> \ell$ (equivalently, $r = \ell+1$) and any $\b{B}\in \c{K}$ with $\b{A}\leq \b{B}$, there is $\b{C}\in \c{K}$ with $\b{B}\leq \b{C}$ and $\b{C}\to (\b{B})^{\b{A}}_{r, \ell}$. The class $\c{K}$ has \emph{finite Ramsey degrees} if $\rm{RD}(\bA, \cK)< \infty$ for every $\bA\in \cK$, and $\c{K}$ has the {\em Ramsey property} if $\rm{RD}(\bA, \cK) = 1$ for every $\b{A}\in \c{K}$. 

Given an infinite $\cL$-structure $\bK$ and a finite structure $\bA\leq \bK$, the \emph{big Ramsey degree of $\bA$ in $\bK$}, denoted $\rm{BRD}(\bA, \bK)$, is the least positive integer $\ell$ (or $\infty$ if no such $\ell$ exists) so that for every integer $r> \ell$ (equivalently, $r = \ell+1$), we have $\mathbf{K}\to (\mathbf{K})^{\mathbf{A}}_{r,\ell}$. \qed
\end{defin}

For instance, the classical finite Ramsey theorem \cite{Ramsey30} says exactly that the class of finite linear orders has the Ramsey property, and the classical infinite Ramsey theorem says exactly that every finite linear order $\bA$ satisfies $\rm{BRD}(\bA, \la \omega, <\ra) = 1$. Some authors call Ramsey degrees by the name \emph{small} Ramsey degrees to distinguish them from big Ramsey degrees. 

We briefly mention that many authors formulate Definition~\ref{Def:Ramsey} with respect to \emph{copies} rather than embeddings, where a \emph{copy} is simply the image of an embedding. Comparing (big) Ramsey degrees when using copies versus embeddings is straightforward. Letting $\rm{RD}^{copy}(\bA, \cK)$ and $\rm{BRD}^{copy}(\bA, \bK)$ be defined in the obvious way, one can show that $\rm{RD}(\b{A}, \cK) = \rm{RD}^{copy}(\bA, \cK)\cdot  |\aut(\b{A})|$ and $\rm{BRD}(\bA, \bK) = \rm{BRD}^{copy}(\bA, \bK)\cdot  |\aut(\b{A})|$; see for instance Proposition 4.4 of \cite{Zuc16}. In this paper, we always refer to the embedding version unless explicitly mentioned otherwise.

Structural Ramsey theory has a long and rich history. It was initiated in the 1970's when Ne\v set\v ril and R\"odl, and independently Abramson and Harrington proved that the class of finite ordered graphs has the Ramsey property~\cite{Abramson1978,Nesetril/Rodl77}. 
In fact, the Ne\v set\v ril--R\"odl theorem is much stronger, implying the Ramsey property for several classes of finite ordered relational structures. More recently, the Kechris--Pestov--Todor\v{c}evi\'c correspondence connecting structural Ramsey theory and topological dynamics as well as Ne\v set\v ril's classification program of Ramsey classes~\cite{Nevsetril2005} boosted the development of the area which led to discoveries of many more classes with the Ramsey property. The most general result on classes with the Ramsey property is due to  Hubi\v cka and Ne\v set\v ril~\cite{HN19}, who extended the earlier proof techniques (Ne\v set\v ril and R\"odl's partite construction \cite{Nesetril1981,Nesetril1989,Nesetril2007}) to structures in languages containing both relations and functions and
gave a sufficient structural condition which can be used to show that a given class is Ramsey, effectively providing a ``black box'' for the combinatorial arguments involved.

While most classes of finite unordered structures do not have
the Ramsey property, often enriching a given class by adding a linear order 
or finitely many new relations 
produces an expansion class with the Ramsey property. 
When this is possible, 
the original class will have finite small Ramsey degrees. For instance, the class of finite ordered graphs has the Ramsey property, while the class of  finite graphs has finite Ramsey degrees. One way to show that a class $\c{K}$ has finite Ramsey degrees, developed by Kechris--Pestov--Todor\v{c}evi\'c \cite{KPT} and Nguyen Van Th\'e \cite{NVT}, is to show that $\c{K}$ admits a \emph{precompact expansion class} $\c{K}^*$ which has the Ramsey property. For example, the class of finite sets has the class of finite linear orders as a precompact expansion, and the class of finite graphs has the class of finite ordered graphs as a precompact expansion. Zucker showed in \cite{Zuc16} that in fact, every \emph{\fr class} $\c{K}$ with finite Ramsey degrees admits a pre-compact expansion class with the Ramsey property.

Recall that a class $\c{K}$ of finite structures is \emph{\fr}if it is closed under isomorphism, contains countably many structures up to isomorphism, contains structures of arbitrarily large finite size, and satisfies the following three properties.
\begin{itemize}
    \item Hereditary property (HP): If $\b{A}\leq \b{B}$ and $\b{B}\in \c{K}$, then $\b{A}\in \c{K}$.
    \item 
    Joint embedding property (JEP): If $\b{A}, \b{B}\in \c{K}$, there is $\b{C}\in \c{K}$ with both $\b{A}\leq \b{C}$ and $\b{B}\leq \b{C}$.
    \item 
    Amalgamation property (AP): For any $\b{A}, \b{B}, \b{C}\in \c{K}$ and embeddings \newline $f\colon \b{A}\to \b{B}$ and $g\colon \b{A}\to \b{C}$, there are $\b{D}\in \c{K}$ and embeddings $r\colon \b{B}\to \b{D}$ and $s\colon \b{C}\to \b{D}$ with $r\circ f = s\circ g$. 
\end{itemize}
Given a \fr class $\c{K}$, we can use the members of $\c{K}$ to build a generic countably infinite structure, the \emph{\fr limit} of $\c{K}$, denoted $\flim(\c{K})$. Writing $\b{K} = \flim(\c{K})$, then $\b{K}$ is characterized up to isomorphism by the following two properties. 
\begin{itemize}
    \item 
    $\c{K} = \age(\bK):= \{\b{A} \text{ finite}: \b{A}\leq \b{K}\}$.
    \item 
    If $\b{A}\subseteq \bK$ is finite and $f\in \emb(\b{A}, \b{K})$, there is $g\in \aut(\b{K})$ with $g|_{\b{A}} = f$.
\end{itemize}
\textbf{Until Section~\ref{Sec:FutureWork}}, $\cK$ denotes a \fr class of $\cL$-structures and $\bK = \flim(\cK)$. We will add assumptions to $\cK$ and $\bK$ later, but at a minimum this will always hold. 

A routine compactness argument yields that if $\b{A}\in \c{K}$ and $\ell< \omega$, then $\rm{RD}(\b{A}, \cK) = \ell$ iff for any $\mathbf{A}\le\mathbf{B}\in \mathcal{K}$ and any positive integer $r$, we have $\mathbf{K}\to (\mathbf{B})^{\mathbf{A}}_{r, \ell}$. In particular, $\rm{RD}(\bA, \cK)\leq \rm{BRD}(\bA, \bK)$. As is commonly done, we refer to the big Ramsey degree of $\bA$ \emph{in $\cK$} and write $\rm{BRD}(\bA, \cK)$ for $\rm{BRD}(\bA, \bK)$. We say that $\c{K}$ has {\em finite big Ramsey degrees} if $\rm{BRD}(\bA, \cK)< \infty$ for every $\bA\in \cK$. If $\cK$ has finite big Ramsey degrees, then $\cK$ has finite small Ramsey degrees. We refer to Dobrinen's ICM paper  \cite{DobICM} for background on historical and recent results on big Ramsey degrees.

It is natural to ask, in analogy to how every \fr class with finite Ramsey degrees admits a precompact expansion with the Ramsey property, if something similar can be done for classes with finite big Ramsey degrees. In this spirit, Zucker in \cite{Zuc19} defines a \emph{big Ramsey structure}, whose definition we recall (and add to) in Definition~\ref{Def:BRS}.

\begin{defin}
    \label{Def:Expansion}
    Given first order languages $\cL^*\supseteq \cL$ and an $\cL^*$-structure $\bM^*$, the \emph{$\cL$-reduct} $\bM|_\cL$ is the $\cL$-structure on the same underlying set as $\bM^*$ and with the same interpretations of symbols from $\cL$ as in $\bM$. Conversely, if $\bM$ is an $\cL$-structure, an \emph{$\cL^*$-expansion} of $\bM$ is some $\cL^*$-structure $\bM^*$ on underlying set $M$ with $\bM^*|_\cL = \bM$. Given an $\cL$-structure $\bM$, $\bB\leq \bM$, and an $\cL^*$-expansion $\bM^*$ of $\bM$, we set
$$\bM^*(\bB):= \{\bB^*: \bB^* \text{ is an $\cL^*$-expansion of $\bB$ with } \bB^*\leq \bM^*\}.$$ 
If $f\in \emb(\bB, \bM)$, we write $\bM^*{\cdot}f$ for the unique $\bB^*\in \bM^*(\bB)$ with $f\in \emb(\bB^*, \bM^*)$. \qed 
\end{defin}

\begin{defin}
    \label{Def:BRS}
    If $\cK$ has finite big Ramsey degrees, a \emph{big Ramsey structure} for $\cK$ is an $\cL^*$-expansion $\bK^*$ of $\bK$ for some first-order language $\cL^*\supseteq \cL$ satisfying the following:
\begin{itemize}
    \item
    For every $\bA\in \age(\bK)$, we have $|\bK^*(\bA)| = \rm{BRD}(\bA, \cK)$.
    \item
    On $\emb(\bA, \bK)$, the coloring $f\to \bK^*{\cdot}f$ witnesses that $\rm{BRD}(\bA, \cK)\geq |\bK^*(\bA)|$.
\end{itemize}

We call $\bK^*$ a \emph{strong} big Ramsey structure for $\cK$ if additionally, $\bK^*$ satisfies the direct analogue of the infinite Ramsey theorem, i.e.\ if we have:
\begin{itemize}
    \item 
    For every $\bA^*\in \age(\bK^*)$, we have $\bK^*\to (\bK^*)^{\bA^*}_2$. \qed
\end{itemize}
\end{defin}
For example, $\la \omega, <\ra$ is a strong big Ramsey structure for the class of finite sets. In fact, all known examples of \fr classes with finite big Ramsey degrees admit strong big Ramsey structures.

While big Ramsey degrees 
(and even in a sense, the strong big Ramsey structures)  of the Rado graph and more generally, unrestricted structures with finitely many binary relations
were fully understood by 2006 via  work of 
Sauer \cite{Sauer06} and its sequel \cite{LSV} by Laflamme, Sauer, and Vuksanovic,
the big Ramsey degrees for  the  $k$-clique-free analogues of the Rado graph, called Henson graphs,   
were  not fully characterized until the present paper.
Vertices were shown to have big Ramsey degree one 
 for the triangle-free  \cite{KR}
and all  $k$-clique-free \cite{ES89} Henson graphs.
Moreover, 
Sauer had shown by 1998 that the big Ramsey degree for edges is two in  triangle-free Henson graph \cite{sauer1998}.
However, a general theorem proving finite big Ramsey degrees  in the Henson graphs  remained elusive, due to the fact that the standard techniques using Milliken's Theorem were known not to work in this setting.
Answering  questions of Sauer, Dobrinen~\cite{Dob0, Dob} developed the new  techniques  of  coding trees and  set-theoretic forcing  on these trees to show (in ZFC) that the classes of finite triangle-free graphs and more generally, finite $k$-clique-free graphs for each $k\ge 3$, have finite big Ramsey degrees. 
Generalizing these techniques to obtain a simpler proof of finite big Ramsey degrees, Zucker~\cite{Zuc22} considered \fr classes of the following form. Let $\mathcal{L}$ be a relational language. An $\mathcal{L}$-structure $\mathbf{F}$ is {\em irreducible} if every $a\neq b\in F$ is contained in some relation, meaning there are $R\in \cL$ and $\vec{c}$ from $F$ with $R^\bF(\vec{c})$ and $a, b\in \vec{c}$.
Given  a set  $\mathcal{F}$  of  finite irreducible 
$\mathcal{L}$-structures such that each $\bF\in \cF$ has size at least two,
let 
$$\rm{Forb}(\mathcal{F}) = \{\bA: \bA \text{ is a finite $\cL$-structure and }\forall\, \bF\in\cF\, (\bF\not\leq \bA)\}.$$
The class $\rm{Forb}(\cF)$ is always a \fr class which additionally satisfies a strengthening of the amalgamation property called \emph{free amalgamation}. This means that given an amalgamation problem $f\colon \bA\to \bB$ and $g\colon \bA\to \bC$, we can find a solution $r\colon \bB\to \bD$ and $s\colon \bC\to \bD$ satisfying the following.
\begin{itemize}
    \item 
    $D = \im(r)\cup \im(s)$.
    \item 
    $\im(r)\cap \im(s) = \im(r\circ f) = \im(s\circ g)$.
    \item 
    If $a\neq b\in D$ are contained in a relation, then $\{a, b\}\subseteq \im(r)$ or $\{a, b\}\subseteq \im(s)$.
\end{itemize}
Conversely, if $\cK$ is a \fr free amalgamation class, then $\cK = \rm{Forb}(\cF)$ for some set $\cF$ of finite irreducible $\cL$-structures. If $\cL$ is a finite relational language with symbols of arity at most two and $\cF$ is a \emph{finite} set of finite irreducible $\cL$-strucutres,  we call $\rm{Forb}(\cF)$ a \emph{finitely-constrained binary free amalgamation class}. In \cite{Zuc22}, Zucker proved that every finitely-constrained binary free amalgamation class has finite big Ramsey degrees. We mention that by examples of Sauer \cite{Sauer03}, the condition that $\cF$ be finite is necessary.

Dobrinen conjectured that the upper bounds obtained in \cite{Dob0} and  \cite{Dob}   for tri\-an\-gle-free and $k$-clique-free graphs were exact. 
This turns out  not to be the case in general and some slight modifications were needed to obtain the exact values. While these were being prepared in \cite{Dob2}, Balko, Chodounsk\'y, Hubi\v{c}ka, Kone\v{c}n\'y, and Vena independently obtained exact values derived from the upper bounds obtained in \cite{Hubicka2020CS}. Slightly before this, general upper bounds for binary finitely-constrained classes appeared in \cite{Zuc22}. Hence the seven authors decided to come together and produce this work. Here, the structural properties responsible for exact big Ramsey degrees will be developed in 
the general framework due to
Zucker in \cite{Zuc22}. This paper  characterizes exact big Ramsey degrees for finitely-constrained binary free amalgamation classes, culminating  and
 concluding work  in
\cite{Dob0}, \cite{Dob}, \cite{EHP}, \cite{ES89}, \cite{ES93}, \cite{Hubicka2020CS},  \cite{KR},  \cite{LSV}, \cite{PS},
\cite{Sauer06},   and  \cite{Zuc22}. The following is our main theorem.

\begin{theorem}\label{thm:main_intro}
    There is an algorithm which, when given a finitely-constrained binary free amalgamation class $\cK$ and $\bA\in \cK$, outputs $\rm{BRD}(\bA, \cK)$.
\end{theorem}

A key step in our characterization is finding the lower bounds, i.e.\ producing \emph{unavoidable} colorings of $\emb(\bA, \bK)$ for each $\bA\in \age(\bK)$. Our construction of these colorings yields the following.

\begin{theorem}
    \label{thm:main_intro_2}
    Let $\cK$ be a finitely-constrained binary free amalgamation class. Then $\cK$ admits a strong big Ramsey structure in a finite relational language. In particular, $\aut(\bK)$ has a metrizable universal completion flow.
\end{theorem}

\emph{Universal completion flows} were introduced by Zucker in \cite{Zuc19}. While it is unknown whether every topological group admits a universal completion flow, it is proven in \cite{Zuc19} that if $\cK$ admits a big Ramsey structure, then $\aut(\bK)$ has a metrizable universal completion flow. Indeed, in a suitable space of expansions of $\bK$, the universal completion flow is formed by simply taking the orbit closure of any big Ramsey structure.

\subsection{Big Ramsey degrees and trees}
\label{SubSec:BRD_Trees}

As we remarked above, $\la \omega, <\ra$ is a strong big Ramsey structure for the class of finite sets. One might hope that adding a well-order is enough to obtain a strong big Ramsey structure in general. However, this is not the case; as soon as the class has non-trivial binary relations, trees start appearing in the study of its big Ramsey behavior. The following paragraphs give some intuition about this, and we hope that they will give the reader a better ``big picture" understanding of what is happening in the main part of this paper.

Let $\cG$ denote the class of finite graphs, let $\mathbf{G} = \la G, E^\bG\ra = \flim(\cG)$ be the \emph{random graph}
(or {\em Rado graph}), and suppose that we want to construct a ``bad'' coloring of $\emb(\bA, \bG)$ for some $\b{A}\in \cG$ (in the sense that every copy of $\bG$ in $\bG$ attains many colors). We can assume that $G =\omega$. Considering vertex $0$, the remaining vertices of $\mathbf G$ split into two \emph{types} -- those which are connected to 0 and those which are not. 
In general, 
for each set of vertices $\{0,\dots,n-1\}\subseteq G$,
there are $2^n$ many types:
Each function $f:n\to 2$ is 
corresponds to the collection of 
all vertices $k\ge n$ in $G$ such that for each $i<n$, $k$ has an edge with $i$ if and only if $f(i)=1$. The set of all types over initial segments of $G$ can thus be identified with the set $T = 2^{<\omega}$ of all finite binary strings. There are three natural partial orders on $T$: the lexicographic order $\preceq$ (which is a dense linear order), the order $\sqsubseteq$ defined by $s\sqsubseteq t$ if $s$ is an initial segment of $t$, which defines a tree ordering on $T$, and the partial order $\leq_\ell$ of \emph{relative levels} given by $s\leq_\ell t$ iff $\dom(s)\leq \dom(t)$. Given $s,t\in T$, we define $s\wedge t$ to be the longest common prefix of $s$ and $t$ and call it their \emph{meet}. We also form a function $c\colon \omega\to T$ where given $v\in\mathbf \omega$, we have $\dom(c(v)) = v$ and $c(v)(i) = 1$ iff $E^\bG(v, i)$ holds.

Notice that just from fixing an enumeration, we obtained a tree with levels and a meet operation as well as a map $c$ which maps vertices of $\mathbf G$ into this tree; this is the \emph{coding tree} of $\bG$. In particular, we can use this tree to color $\emb(\bA, \bG)$, where given $f\in \emb(\bA, \bG)$, we record certain information about the function $c\circ f\colon \bA\to T$. Writing $S = \im(f)\subseteq \omega$ and $\crit(S) := S\cup \{\dom(c(i)\wedge c(j)): i, j\in S\} := \{i_0<\cdots < i_{n-1}\}$,  we obtain for each $a\in A$ a function $\pi_f(a)\in 2^{<n}$ with $\dom(\pi_f(a)) =m$ iff $f(a) = i_m$ and with $\pi_f(a)(j) = (c\circ f(a))(i_j)$, for each $j<m$. We color $f$ based on the map $a\to \pi_f(a)$.

It turns out that this coloring is not \emph{unavoidable}.     
In addition to proving 
that the random graph has finite big Ramsey degrees,
Sauer~\cite{Sauer06} considered the above coloring and constructed an $\eta\in \emb(\bG, \bG)$ which reduced the number of colors, and soon after, Laflamme, Sauer and Vuksanovic~\cite{LSV}  proved that Sauer's $\eta$  eliminates as many colors as possible. The remaining colors are such that, with notation as above, the $\sqsubseteq$-downwards closure of $\im(\pi_f)$ has exactly one ``interesting event" per level, i.e.\ exactly one node $v$ with either $\pi_f(a) = v$ for some $a\in A$ (a ``coding event") or $\pi_f(a)\wedge \pi_f(b)$ for some $a, b\in A$ (a ``splitting event"). In particular,  the coding nodes form an antichain in the tree.

Hence by enumerating the random graph, we create two kinds of ``interesting events" which we can use to construct unavoidable colorings, namely splitting and coding. For the random graph, these are the only interesting events. However, consider the class $\cG_3$ of finite triangle-free graphs, and set $\mathbf G_3 = \flim(\cG_3)$. We once again fix a well-order of vertices of $\mathbf G_3$, identify its vertex set with $\omega$. Recall that a type is a set of all vertices of $\mathbf G_3$ connected in the same way to a given initial segment of vertices of $\mathbf G_3$, and so we can ask which finite graphs embed into the substructure of $\mathbf G_3$ induced on this set (called the \emph{age} of the type). For $\mathbf G$ the answer is always $\cG$. In $\mathbf G_3$, however, the situation is more complex. If the type has no edges to the corresponding initial segment, then its age is $\cG_3$, but once the type attains a neighbor, its age shrinks to the class of all finite graphs with no edges; since all vertices of the type have a common neighbor, no two of them can be connected by an edge. Thus the appearance of the first neighbor of a vertex is an interesting event, and the definition of $\crit(S)$ in this setting must account for these. In general, one has to consider ages of not only single types, but of several types combined. For example, given two different types in $\bG_3$ over the same initial segment, an interesting event we must track is when these two types have an edge to a \emph{common} vertex of the initial segment. This corresponds to the notion of \emph{parallel ones} from \cite{sauer1998} and \cite{Dob0}. 
Here, we will broadly call events like this \emph{age changes}, and they form a crucial part of our characterization of exact big Ramsey degrees.

These ideas can be tracked back to work of Sauer and his co-authors who, in a series of papers~\cite{ES89}, \cite{ES93}, \cite{LNVTPS}, \cite{DLPS}, \cite{Sauer2012}, \cite{Sauer03}, \cite{SNVT}, studied these questions in the special case of vertex colorings. He introduced the poset of ages and used it to characterize the big Ramsey degrees of vertices of many homogeneous structures~\cite{Sauer03}, including  the examples 
which will be
discussed in Subsection~\ref{SubSec:Organization}. Characterizing big Ramsey degrees for larger finite substructures requires a generalization of this partial order which we develop in Section~\ref{Section:AgeClasses}.

\subsection{Organization of the paper}
\label{SubSec:Organization}

Section~\ref{Section:AgeClasses} introduces the key concept of gluings, which we use to define the higher-dimensional posets of ages. Section~\ref{Sec:Diaries} defines \emph{diaries}, the tree-like objects which are induced by the key ``unavoidable" aspects of the coding trees discussed in Subsection~\ref{SubSec:BRD_Trees} responsible for the big Ramsey degrees. In particular, every diary codes an $\cL$-structure embeddable into $\bK$, and any diary which codes an $\cL$-structure isomorphic to $\bK$ can be encoded as an $\cL^*$-expansion of $\bK$ for some suitable $\cL^*$. Using ideas from a recent preprint of Dobrinen and Zucker \cite{DZ2023}, Proposition~\ref{Prop:Exists_Diary} constructs a diary coding $\bK$.\footnote{This is a key difference between this version of the paper and the first version which appeared on arXiv. It allows us to eliminate the concept of a ``diagonal substructure," which dramatically simplifies the upper bound proof, now the content of Section~\ref{Sec:Upper_bounds}.}

The main goal of the rest of the paper is to show that diaries which encode $\bK$ are strong big Ramsey structures for $\cK$. The argument has two main components, namely lower and upper bounds for big Ramsey degrees. Section~\ref{Section:Proof_Of_Thm_EmbDiary} proves the lower bounds by proving the stronger Theorem~\ref{Thm:EmbDiary}, which says that any two diaries coding $\bK$ are bi-embeddable. Section~\ref{Sec:Upper_bounds} proves the upper bounds for big Ramsey degrees using the general upper bound theorem from \cite{Zuc22}. This along with the bi-embeddability from Theorem~\ref{Thm:EmbDiary} suffice to show that any diary coding $\bK$ is a strong big Ramsey structure for $\cK$. Lastly, Section~\ref{Sec:FutureWork} collects some open questions and previews some related ongoing work.

Throughout, we also develop a few key examples to help illustrate the definitions and the results. These examples are:
\begin{itemize}
    \item 
    Given $\ell \geq 3$, $\cG_\ell$ denotes the class of finite $K_\ell$-free graphs. We denote its \fr limit by $\bG_\ell$. We will especially focus on the case $\ell = 3$. 
    \item 
    $\cG_\bT$ denotes the class of finite oriented graphs (i.e.\ no loops and no $2$-cycles) which do not embed the oriented $3$-cycle, which we denote by $\bT$. Write $\bG_\bT$ for the \fr limit. 

    This example differs from $\cG_\ell$ since the single unary is \emph{free} (Definition~\ref{Def:Free}). Roughly speaking, this says that there is a non-trivial way of extending members of the class by one new point, with all old points having the same \emph{non-trivial} relation to the new point, and such that the result stays in the class. Another example of a class with a single unary which is free is $\cG$ (Subsection~\ref{SubSec:BRD_Trees}), but it is interesting to develop an example with both a free relation and a non-trivial forbidden substructure.
    \item 
    $\cH$ denotes the class of finite graphs with two types of edges, which we call red or blue (i.e.\ a given pair of vertices can be unrelated, red-related, or blue-related), and which forbid monochromatic triangles. We let $\bH$ denote the \fr limit.

    This class differs from both of the above classes in that the (one-dimensional) poset of ages has more than one \emph{maximal path} (Definition~\ref{Def:Paths}). As mentioned earlier in the introduction, Sauer in \cite{Sauer03} used the poset of ages, and in particular maximal paths through it, to compute the big Ramsey degrees of vertices. In particular, vertices in $\cH$ will have big Ramsey degree $2$. 
\end{itemize}
We mention that all of these examples only have a single unary. While we do give the proof in full generality for any number of unary predicates, we opt to keep the examples relatively simple while illustrating key differences between big Ramsey and small Ramsey degrees.

\section{\texorpdfstring{$\c{L}$}{L}-structures, gluings, and age classes}
\label{Section:AgeClasses}

As indicated in the introduction, understanding the poset of ages is one of the key ingredients in our characterization of big Ramsey degrees. Since members of this poset can be described by which fragments of forbidden structures one can glue to various structures without creating a forbidden structure, we spend this section developing the necessary formalism for being able to do this abstractly.

\textbf{Until Section~\ref{Sec:FutureWork}}, $\cL$ is a finite relational language with symbols of arity at most two, $\cF$ is a finite set of finite irreducible $\cL$-structures so that each $\bF\in \cF$ has size at least $2$, $\cK = \rm{Forb}(\cF)$, and $\bK = \flim(\cK)$. Let $\|\cF\| = \max\{|F|: \bF\in \cF\}$. We can arrange, by changing $\cL$ and $\cF$ as necessary, that all of the following hold for any $\bA\in \cK$:
\begin{itemize}
    \item 
    For any $a\in A$, there is exactly one unary predicate $U\in \cL$ so that $U^\bA(a)$ holds.
    \item
    For any $a\in A$ and any binary $R\in \cL$, we have $\neg R^\bA(a, a)$.
    \item
    For any $a\neq b\in A$, there is at most one $R\in \cL$ with $R^\bA(a, b)$.
    \item
    There is a map $\rm{Flip}: \cL\to \cL$ so that for $a\neq b\in A$, $R^\bA(a, b)$ iff $\rm{Flip}(R)^\bA(b, a)$.
\end{itemize}
Writing $\cL^\sf{u} = \{U_i: i< \sfU\}$ and $\cL^\sf{b} = \{R_i: i< k\}$ for the unary and binary predicates, respectively, we treat $U^\bA\colon A\to \sfU$ and $R^\bA\colon A^2\setminus \{(a, a): a\in A\}\to k$ as functions, i.e.\ $U^\bA(a) = i$ iff $U^\bA_i(a)$ holds, etc.

The relational symbol $R_0$ will play the role of ``no relation,'' and this is the sense in which we interpret notions such as ``irreducible,'' ``free amalgam,'' etc.

We let $\{V_n: n< \omega\}$ be new unary symbols not in $\cL^\sf{u}$, and we set $\cL_d := \cL^\sf{b}\cup \{V_i: i< d\}$. We always assume that the following items hold for any $\cL_d$-structure $\bB$ that we discuss: 
\begin{itemize}
    \item 
    $B\cap \omega = \emptyset$.
    \item
    For any $a\in B$, there is exactly one $i< d$ so that $V_i^\bB(a)$ holds. 
    \item
    Conventions regarding the binary symbols are exactly the same as those for $\cL$-structures. 
\end{itemize}
We similarly treat $V^\bB\colon B\to d$ as a function. Write $\fin(\cL_d)$ for the class of finite $\cL_d$-structures.

\subsection{Gluings}

Our reason for the convention that $B\cap \omega = \emptyset$ whenever $\bB$ is an $\cL_d$-structure is that we will attach members of $\fin(\cL_d)$ to $\cL$-structures to form new $\cL$-structures, and we will often consider $\cL$-structures whose underlying set is a subset of $\omega$.

\begin{defin}
    \label{Def:Gluings}
    A \emph{gluing} is a triple $\gamma:= (\bX, \rho, \eta)$ where:
\begin{itemize}
    \item
    $\bX$ is a finite $\cL$-structure with $X\subseteq \omega$. We call $X$ the \emph{underlying set} of $\gamma$ and $\bX$ the \emph{structure} of $\gamma$.
    \item 
    There is $d< \omega$ so that $\rho\colon d\to \sfU$ is a function. We call $d$ the \emph{rank} of $\gamma$ and $\rho$ the \emph{sort} of $\gamma$. More generally, we can call any function from a natural number to $\sfU$ a \emph{sort}.
    \item
    $\eta\colon d\times X \to  k$ is a function called the \emph{attachment map} of $\gamma$.
\end{itemize}

Given $\rho\colon d\to \sfU$, we let $\rm{Glue}(\rho)$ be the set of gluings with sort $\rho$, and we let $\tilde{\rho}\in \rm{Glue}(\rho)$ be the gluing $(\emptyset, \rho, \emptyset)$.

Given $\bB$ an $\cL_d$-structure and $\gamma = (\bX, \rho, \eta)$ a rank $d$ gluing, the $\cL$-structure $\gamma(\bB)$ is formed on underlying set $X\cup B$ so that:
\begin{itemize}
    \item 
    $\gamma(\bB)|_X= \bX$.
    \item
    The $\cL^\sf{b}$-part of $\gamma(\bB)$ on $B$ is induced from $\bB$.
    \item
    If $b\in B$, we have $U^{\gamma(\bB)}(b) = \rho\circ V^\bB(b)$.
    \item
    If $b\in B$ and $x\in X$, we have $R^{\gamma(\bB)}(b, x) = \eta(V^\bB(b), x)$. \qed
\end{itemize} 
\end{defin}

\begin{rem}
$\mathcal{L}_d$-structures and gluings enable a precise description of the different ways a given finite $\mathcal{L}$-structure can be extended to another $\mathcal{L}$-structure. The $\mathcal{L}_d$-structures are structures $\mathbf{B}$ with binary relations in $\cL^\sf{b}$ with an accompanying partition of $B$ into $d$-many (possibly empty) labeled pieces $\{B_i: i<d\}$. Given a finite $\mathcal{L}$-structure $\bX$, a gluing $\gamma=(\bX,\rho,\eta)$ of rank $d$  is a set of instructions for gluing any $\mathcal{L}_d$-structure $\mathbf{B}$ to $\bX$ to obtain an $\mathcal{L}$-structure extending $\bX$. Specifically, for each $i<d$, $\gamma$ assigns each element of $B_i$ the same unary relation $U_{\rho(i)}$ in $\cL^\sf{u}$, and for each $x$ in $X$, $\gamma$ assigns the same $\cL^\sf{b}$-relation between $x$ and each element of $B_i$ (see Figure~\ref{fig:Gluing}). 
    \begin{figure}[ht]
		\centering		\includegraphics{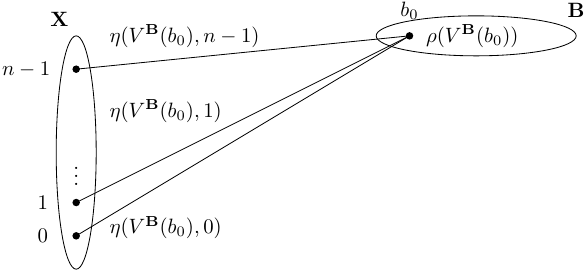}
		\caption{A gluing $(\bX, \rho, \eta)$ with $X = \{0,...,n-1\}$}
		\label{fig:Gluing}
    \end{figure}
\end{rem}

As the (abuse of) notation suggests, the gluing $\gamma$ gives rise to a map from the class of $\cL_d$-structures to the class of $\cL$-structures. This suggests the following convention.
\vspace{2 mm}

\noindent
\textbf{Convention}: Given a class $\cC$ of $\cL'$-structures ($\cL'\in \{\cL, \cL_d\}$), we identify $\cC$ with its characteristic function, i.e.\ $\cC(\bA) = 1$ iff $\bA\in \cC$. 

\begin{defin}
 \label{Def:Poset_Of_Classes}
Given a gluing $\gamma$ of rank $d$, we define the class of finite $\cL_d$-structures
$$\cK{\cdot}\gamma := \{\bB\in \fin(\cL_d): \gamma(\bB)\in \cK\}$$
Note that by treating $\cK$ and $\cK{\cdot}\gamma$ as characteristic functions, the notation becomes quite suggestive, i.e.\ $(\cK{\cdot}\gamma)(\bB) = \cK(\gamma(\bB))$. 
   
A class $\cA$ of finite $\cL_d$-structures \emph{contains all unaries} if $\cA$ contains every singleton $\cL_d$-structure. A rank $d$ gluing $\gamma$ \emph{admits all unaries} if $\cK{\cdot}\gamma$ contains all unaries.
In other words, writing $\gamma = (\bX, \rho, \eta)$, then $\bX\in \cK$ and for each $i<d$, we can extend $\bX$ to a new member of $\cK$ by adding one new vertex with unary $\rho(i)$ and connecting this new vertex to $\bX$ with binary relations determined by $\eta|_{\{i\}\times X}$. 
    
Given a sort $\rho\colon d\to \sfU$, we set
$$P(\rho):= \{\cK{\cdot}\gamma: \gamma\in \rm{Glue}(\rho)\text{ admits all unaries}\}$$
We treat $P(\rho)$ as a partial order under inclusion. In the case $\sfU = 1$, there is a unique sort with domain $d$, so we just write $P(d)$. \qed
\end{defin}

\begin{rem}
    When $\sfU = 1$ and we discuss $P(1)$, it is helpful to keep in mind that even though members of $P(1)$ are classes of $\cL_1$-structures, each such class is bi-interpretable with a subclass of $\cK$, where the interpretation interchanges the unary $U_0$ with the unary $V_0$. In particular, the maximal member of $P(1)$ is always bi-interpretable with $\cK$. Going forward, we will perform these bi-interpretations without explicit mention, and simply regard members of $P(1)$ as subclasses of $\cK$.  
\end{rem}

\begin{exa}
    \label{Exa:TFreeAges}
	Let $\c{K} = \cG_\bT$ be the class of \emph{oriented} graphs which don't embed the oriented $3$-cycle $\bT$ as in Subsection~\ref{SubSec:Organization}. Here $\sfU = 1$ and $k = 3$, with $\mathrm{Flip}(1) = 2$  and $\mathrm{Flip}(2)=1$. The poset $P(1)$ contains only $\cG_\bT$ (up to bi-interpretability), while the poset $P(2)$ contains $4$ members -- the class of finite $\b{T}$-free oriented graphs $\bB$ with $V_0^\bB, V_1^\bB$ arbitrary; the class of finite $\b{T}$-free oriented graphs $\bB$ with no edges from $V_0^\bB$ to $V_1^\bB$; the class of finite $\b{T}$-free oriented graphs $\bB$ with no edges from $V_1^\bB$ to $V_0^\bB$, and the class of finite $\b{T}$-free oriented graphs $\bB$ with no edges between $V_0^\bB$ and $V_1^\bB$. We can encode these $4$ classes via the $4$ \emph{directed} graphs (i.e.\ no loops, but $2$-cycles are ok) on vertex set $2$. Because $\bT$ has $3$ vertices, we have that for $d\geq 3$, $P(d)$ is determined by what is allowed on any pair $V_i^\bB, V_j^\bB$. Hence the members of $P(d)$ are in one-one correspondence with \emph{directed} graphs on vertex set $d$. 
\end{exa}

\begin{prop}
    \label{Prop:Finite_Posets}
    For any sort $\rho\colon d\to \sfU$, $P(\rho)$ is finite.
\end{prop}

\begin{proof}
Fix $\gamma = (\bX, \rho, \eta) \in \rm{Glue}(\rho)$. Observe that since $\cK$ is a free amalgamation class, so is $\cK{\cdot}\gamma$. Thus $\cK{\cdot}\gamma = \rm{Forb}(\cE)$ for some set $\cE$ of finite irreducible $\cL_d$-structures. We can assume that distinct $\bA, \bB\in \cE$ do not embed into one another. So fix $\bB\in \cE$. As $\gamma(\bB)\not\in \cK$, find $\bF\subseteq \gamma(\bB)$ isomorphic to a member of $\cF$. Towards showing that $B\subseteq F$, fix $b\in B$, and consider $\bC\subsetneq\bB$ induced on $B\setminus \{b\}$. By our assumption on $\cE$, we have $\gamma(\bC)\in \cK$. In particular, $F\not\subseteq C\cup X$. As $F\subseteq B\cup X$, it follows that we must have $b\in F$, so $B\subseteq F$ as desired. 
    Hence up to isomorphism, there are only finitely many possibilities for $\bB$, so also for $\cE$. 
\end{proof}

Since $\cK$ has free amalgamation, the following construction on gluings is quite natural.

\begin{defin}
    \label{Def:Union_Gluing}
    We call $\gamma_0 = (\bX_0, \rho, \eta_0)$ and $\gamma_1 = (\bX_1, \rho, \eta_1)\in \rm{Glue}(\rho)$ \emph{disjoint} if $X_0\cap X_1 = \emptyset$. Given disjoint $\gamma_0, \gamma_1\in \rm{Glue}(\rho)$ as above, the \emph{union gluing} $\gamma_0\cup \gamma_1$ is the gluing $(\bX_0\sqcup \bX_1, \rho, \eta_0\cup \eta_1)\in \rm{Glue}(\rho)$, where $\bX_0\sqcup \bX_1$ is the $\cL$-structuure on $X_0\cup X_1$ with $\bX_0$ and $\bX_1$ as induced substructures and with $R^{\bX_0\sqcup \bX_1}(x_0, x_1) = 0$ for every $x_0\in X_0$ and $x_1\in X_1$.
\end{defin}

\begin{lemma}
\label{Lem:Union_Gluing}
    For any sort $\rho$ and disjoint gluings $\gamma_0, \gamma_1\in \rm{Glue}(\rho)$, then writing $\gamma = \gamma_0\cup \gamma_1$, we have $\cK{\cdot}\gamma = \cK{\cdot}\gamma_0\cap \cK{\cdot}\gamma_1$.
\end{lemma}

\begin{proof}
    This is an immediate consequence of free amalgamation in $\cK$.
\end{proof}

\begin{prop}
    \label{Prop:Intersections}
    For any sort $\rho\colon d\to \sfU$, $P(\rho)$ is closed under intersections.
\end{prop}

\begin{proof}
As $P(\rho)$ is finite, it is enough to consider $\cA, \cB\in P(\rho)$ and show that $\cA\cap \cB \in P(\rho)$. Let $\gamma_\cA = (\bX_\cA, \rho, \eta_\cA)$ and $\gamma_\cB = (\bX_\cB, \rho, \eta_\cA)$ be disjoint gluings with $\cA = \cK{\cdot}\gamma_\cA$, and $\cB = \cK{\cdot}\gamma_\cB$. Form $\gamma_\cA\cup \gamma_\cB$ and apply Lemma~\ref{Lem:Union_Gluing}.
\end{proof}

\begin{exa}
    \label{Exa:LFreeGraphs}
	Fix $\ell\geq 3$ and consider the class $\c{G}_\ell$ of $K_\ell$-free graphs. So $\sfU = 1$ and $k = 2$. Given $d< \omega$, let $\binom{d}{{<}\ell}$ denote the collection of subsets of $\{0,\dots,d-1\}$ of size at most $\ell-1$. Each member of $P(d)$ is described by a function $z\colon \binom{d}{{<}\ell}\to \{1,2,\dots,\ell-1\}$ satisfying the following properties:
    \begin{enumerate}
        \item 
        $z$ is monotone: If $S, T\in \binom{d}{{<}\ell}$ and $S\subseteq T$, then $z(S)\leq z(T)$,
        \item
        $z$ is consistent: For each $T\in \binom{d}{{<}\ell}$, there is a function $\sigma\colon T\to \omega$ with $\sum_{i\in T}\sigma(i) = z(T)$ and so that for any $S\subseteq T$, we have $\sum_{i\in S}\sigma(i)\leq z(S)$.
    \end{enumerate}
    We call a function $z$ satisfying the above $2$ items a \emph{$(d, \ell)$-graph-age function}, and an \emph{$\ell$-graph-age function} is just a $(d, \ell)$-graph-age function for some $d< \omega$. Given a $(d, \ell)$-graph-age function $z$, we define the class $\c{G}_z$ via
    $$\c{G}_z := \left\{\b{B}\in \fin(\cL_d): \forall T\in \binom{d}{{<}\ell} \left[\bigcup_{i\in T} V_i^\b{B} \text{ contains no $(z(T)+1)$-clique}\right]\right\}.$$
    So the number $z(T)$ describes the largest clique which is allowed to appear in $\bigcup_{i\in T} V_i^\bB$.  
    
    We sketch the argument that $P(d)$ is indeed characterized by classes of this form. First, fix $z$ a $(d, \ell)$-graph-age function. To see that $\cG_z\in P(d)$, build a gluing $\gamma = (\bX, \rho, \eta)$ as follows.  For each $T\in \binom{d}{{<}\ell}$, add an $(\ell-1-z(T))$-clique $\bX_T$ to $\bX$, so that $X = \bigsqcup_{T\in \binom{d}{{<}\ell}} X_T$. Then define $\eta\colon d\times X\to 2$ by setting $\eta(i, x) = 1$ iff for some $T\in \binom{d}{{<}\ell}$ we have $i\in T$ and $x\in X_T$. It is straightforward to check that $\cG_\ell{\cdot}\gamma = \cG_z$. Conversely, if $\cA\in P(d)$, then letting $z_\cA\colon \binom{d}{{<}\ell}\to \{1,...,\ell-1\}$ be such that $z_\cA(T)$ describes the largest possible clique in $\cA$ with unaries from $T$, then one can check that $\cA = \cG_{z_\cA}$.

    Unfortunately, checking whether a given function $z$ satisfies item $(2)$ of the definition of $(d, \ell)$-graph-age function seems to be a challenging combinatorial problem.  When $\ell = 3$, however, item $(1)$ implies item $(2)$. We briefly mention that since the single unary in each $\cG_\ell$ is not \emph{free} (Definition~\ref{Def:Free}), the subset of $P(d)$ corresponding to those $(d, \ell)$-graph-age functions with $z(\{i\})\leq \ell-2$ for each $i< d$ will be of special importance. When $\ell = 3$, we can think of such $(d, \ell)$-graph-age functions as graphs; each singleton must get value $1$, and for pairs, we can think of value $2$ as an edge and value $1$ as a non-edge. Each graph on vertex set $d$ vertices corresponds to such a $(d, 3)$-graph-age function. \qed
\end{exa}

\subsection{Manipulating classes of structures}

When we introduce aged coding trees in Section~\ref{Sec:Diaries}, each level set of size $d$ from such a tree will be endowed with a class of $\cL_d$-structures. Passing to a subset of this level set will then correspond to a restriction operation on the class. Using a general form of this restriction operation, we discuss how to shrink a given class of $\cL_d$-structures along some $S\subseteq d$ to obtain a new, smaller class of $\cL_d$-structures. We then analyze what it means for such an operation to result in a \emph{consecutive} pair of classes.
\vspace{2 mm}

\noindent
\textbf{Convention}: When discussing partial functions $e\colon d_0\rightharpoonup d$ with $d_0, d<\omega$, we think of $e$ as equipped with the knowledge of what $d_0$ is, even if $\dom(e)\subsetneq d_0$, and of what $d$ is, even if $e$ is not surjective. 

\begin{defin}
    \label{Def:Maps_and_Gluings}
    Let $d_0, d< \omega$ and $e\colon d_0\rightharpoonup d$ be a partial function. If $\bB$ is an $\cL_{d_0}$-structure, then $e{\cdot}\bB$ is the $\cL_d$-structure on the underlying set $\{x\in \bB: V^\bB(x)\in \dom(e)\}$, with $\cL^\sf{b}$-part induced from $\bB$, and with $V^{e\cdot \bB}(x) = e(V^\bB(x))$. In other words, we take $\bB$, throw away the points whose unary is not in $\dom(e)$, then relabel the unaries according to the function $e$. If $\cA$ is a class of finite $\cL_d$-structures, we set
    $$\cA{\cdot}e := \{\bB\in \fin(\cL_{d_0}): e{\cdot}\bB\in \cA\}.$$
    Treating $\cA$ and $\cA{\cdot}e$ as characteristic functions, we have $(\cA{\cdot}e)(\bB) = \cA(e{\cdot}\bB)$.
\end{defin} 

\begin{figure}[ht]
		\centering
		\includegraphics{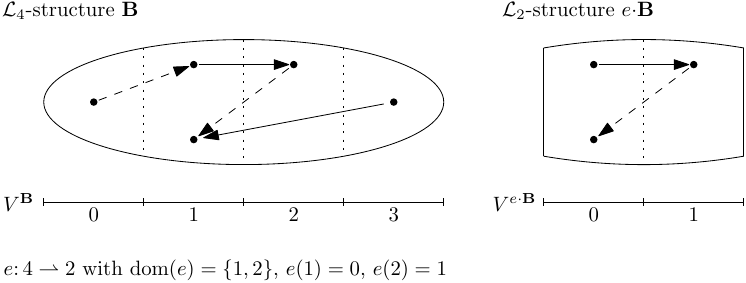}
		\caption{An example of forming $e{\cdot}\bB$ from $e$ and $\bB$.}
		\label{fig:Be}
    \end{figure}

    \begin{notation}  
    \label{Notation:Manipulate_Classes}
  Given an injection $e\colon d_0\to d$, we treat $e^{-1}\colon d\rightharpoonup d_0$ as a partial function with domain $\im(e)$. By abusing notation, we can identify a finite 
    set
    $S\subseteq \omega$ with its increasing enumeration. When doing this, we interpret $S$ as a total function with domain $|S|$ and with codomain given from context, i.e.\ writing $S\subseteq d$ indicates that the codomain is $d$. If $S\subseteq d$, $\cA$ is a class of finite $\cL_d$-structures, and keeping in mind the convention before Definition~\ref{Def:Maps_and_Gluings}, then we can write $\cA{\cdot}S$, $\rho\circ S$, etc. We write $\cA^i$ for $\cA{\cdot}\{i\}$. We write $\rm{id}_{S, d}\colon d\rightharpoonup d$ for the partial function with domain $S$ which is the identity on $S$. 
    \end{notation}

\begin{prop}
    \label{Prop:Manipulating_Classes}
    Let $d< \omega$, and fix a sort $\rho\colon d\to \sfU$ and $\cA\in P(\rho)$.
    \begin{enumerate}
        \item
        If $d_0< \omega$ and $e\colon d_0\to d$ is a total function, then $\cA{\cdot}e\in P(\rho\circ e)$.
        \item
        If $f\colon d\rightharpoonup d_0$ is a partial function, $\theta\colon d_0\to \sfU$ is a sort satisfying $\theta\circ f = \rho|_{\dom(f)}$, and $\cB\in P(\theta)$, then $\cA\cap \cB{\cdot}f\in P(\rho)$.
    \end{enumerate}
\end{prop}

\begin{proof}
    We first note that both $\cA{\cdot}e$ and $\cA\cap \cB{\cdot}f$ contain all unaries since $\cA$ and $\cB$ do.
    So let $\cA = \cK\cdot\gamma_\cA$ and $\cB = \cK\cdot\gamma_\cB$ with $\gamma_\cA = (\bX_\cA, \rho, \eta_\cA)$ and $\gamma_\cB = (\bX_\cB, \rho\circ e, \eta_\cB)$. We can assume that the underlying sets $X_\cA$ and $X_\cB$ are disjoint.
    
    To see that $\cA{\cdot}e\in P(\rho\circ e)$, we can define the gluing $\gamma_\cA{\cdot}e = (\bX_\cA, \rho\circ e, \eta_\cA{\cdot}e)$, where given $c< d_0$ and $x\in X_\cA$, we have $(\eta_\cA{\cdot}e)(c, x) = \eta_\cA(e(c), x)$. It is routine to check that $\gamma_\cA{\cdot}e$ is a gluing with $\cK{\cdot} (\gamma_{\cA}{\cdot}e) = (\cK{\cdot}\gamma_\cA){\cdot}e = \cA{\cdot}e$.

    To see that $\cA\cap \cB{\cdot}f\in P(\rho)$, first form the gluing $\gamma_\cB{\cdot}f = (\bX_\cB, \rho, \eta_\cB{\cdot}f)$, where given $c< d$ and $x\in X_\cB$, we have 
    \begin{align*}
        (\eta_{\cB}{\cdot}f)(c, x) = \begin{cases}
        0 \quad &\text{if } c\not\in \dom(f)\\[1 mm]
        \eta_\cB(f(c), x) \quad &\text{if } c\in \dom(f).
    \end{cases}
    \end{align*}
    Letting $\gamma = \gamma_\cA\cup \gamma_\cB{\cdot}f$ denote the union gluing, it is routine to check that $\cK{\cdot}\gamma = \cA\cap \cB{\cdot}f$.
\end{proof}

\begin{defin}
    \label{Def:Consecutive}
    Fix a sort $\rho\colon d\to \sfU$. We say that $\cA\supseteq \cB\in P(\rho)$ are \emph{consecutive} if there is no $\cC\in P(\rho)$ with $\cA\supsetneq \cC\supsetneq \cB$. Write $\con(\rho) = \{(\cA, \cB)\in P(\rho)^2: \cA\supseteq \cB\text{ are consecutive}\}$. By Proposition~\ref{Prop:Manipulating_Classes}, if $(\cA, \cB)\in \con(\rho)$ and $S\subseteq d$, then either $(\cA{\cdot}S, \cB{\cdot}S)\in \con(\rho)$ or $\cA{\cdot}S = \cB{\cdot}S$. 
    
    We say that $(\cA, \cB)\in \con(\rho)$ is \emph{essential} on some $S\subseteq d$ if there is a minimal-under-inclusion $\bC\in \cA\setminus \cB$ such that $S = \{V^\bC(x): x\in C\}$. We say that $(\cA, \cB)\in \con(\rho)$ is \emph{essential} if it is essential on $d$, and we write $\econ(\rho) = \{(\cA, \cB)\in \con(\rho): (\cA, \cB) \text{ is essential}\}$.  
\end{defin}

\begin{fact}
    \label{Fact:Bound_Essential}
    By arguments very similar to those of Proposition~\ref{Prop:Finite_Posets}, such a $\bC$ must satisfy $|\bC|< \|\cF\|$, so in particular, $|S|<\|\cF\|$. 
\end{fact}

\begin{prop}
    \label{Prop:Consecutive}
    With notation as in Definition~\ref{Def:Consecutive}, if $(\cA, \cB)\in \con(\rho)$, then there is a unique $S\subseteq d$ on which $(\cA, \cB)$ is essential. Fixing this $S$, if $S'\subseteq d$ and $S\subseteq S'$, then $(\cA{\cdot}S', \cB{\cdot}S')\in \con(\rho\circ S')$, and if $S\not\subseteq S'$, we have $\cA{\cdot}S' = \cB{\cdot}S'$.
\end{prop}

\begin{proof}
    Towards a contradiction, suppose $(\cA, \cB)$ was essential on both $S_0\neq S_1\subseteq d$. Then for each $i< 2$, we have $\cA\supsetneq \cA\cap (\cB{\cdot}\rm{id}_{S_i,d}) \supseteq \cB$, so as $(\cA, \cB)\in \con(\rho)$, we have $\cB = \cA\cap(\cB{\cdot}\rm{id}_{S_i, d})$. In particular, if $\cA = \cK{\cdot}\gamma$ for some gluing $\gamma$, then as in the proof of Proposition~\ref{Prop:Manipulating_Classes}, we can for each $i<2$ write $\cB = \cK\cdot(\gamma\cup \delta_i)$, where $\delta_i = (\bX_i, \rho, \eta_i)$ is a gluing with the property that  $\eta_i(c, x) = 0$ whenever $c\not\in S_i$. 
    
    Now without loss of generality assume that $S_0\setminus S_1\neq \emptyset$, and let $\bC_0\in \cA\setminus \cB$ witness that $(\cA, \cB)$ is essential on $S_0$. However, considering the gluing $\gamma\cup \delta_1$, and since $\cK$ has free amalgamation, we see that if $\bC_0\not\in \cB = \cK(\gamma\cup \delta_1)$, then a proper substructure of $\bC_0$ (namely, the substructure induced on points with unary in $S_1$) must also fail to be in $\cB$, a contradiction.

    The remaining claims about $S'\subseteq d$ are straightforward.
\end{proof}

\subsection{Paths and path sorts}\label{sec:paths}

 Given $i< \omega$, let $\iota_i\colon 1\to \omega$ be the function with $\iota_i(0) = i$.  Sometimes the intended codomain of $\iota_i$ is some $d< \omega$ (for the purposes of forming the partial function $\iota_i^{-1})$, and this will be clear from context. If $i< \sfU$, then $\iota_i$ is also a sort, and we write $P_i$ for $P(\iota_i)$. 

\begin{defin}
    \label{Def:Paths}
     A \emph{path} through $P_i$ is any
     subset $\frak{p}\subseteq P_i$ linearly ordered by inclusion. Write $\path_i$ for the set of paths through $P_i$ and $\path := \bigsqcup_{i< \sfU} \path_i$. A \emph{maximal} path is maximal under inclusion; write $\sf{MP}_i$ for the set of maximal paths through $P_i$ and  $\sf{MP} := \bigsqcup_{i< \sfU}\sf{MP}_i$.   We equip $\sf{MP}$ with an arbitrary linear order $\leq_{\MP}$. We write $\sf{u}\colon \sf{MP}\to \sfU$ for the map sending $\frak{p}\in \sf{MP}$ to the $i< \sfU$ with $\frak{p}\in \sf{MP}_i$.
     
     A \emph{full} path is an initial segment of a maximal path; write $\sf{Full}_i$ for the set of full paths through $P_i$ and $\full := \bigsqcup_{i< \sfU} \full_i$. Given $\frak{p}\in \sf{Full}_i$, $\max(\frak{p})$ and $\min(\frak{p})$ denote the maximal and minimal members of $\frak{p}$ under inclusion. We note that $\min(\frak{p}) = \min(P_i) = \bigcap_{\cA\in P_i} \cA$ (Proposition~\ref{Prop:Intersections}). When $\frak{p}\in \MP_i$, then $\max(\frak{p}) = \max(P_i) = \cK{\cdot}\tilde{\iota}_i$, where we recall that $\tilde{\iota}_i$ is the gluing $(\emptyset, \iota_i, \emptyset)$.
     Up to bi-interpretation, $\cK{\cdot}\tilde{\iota}_i$ is the class of all structures in $\mathcal{K}$
     with all vertices having unary relation $i$.
     Given $\frak{p}\in \sf{Full}_i$, we write $\frak{p}'\in \sf{Full}_i$ for the full path $\frak{p}\setminus \{\max(\frak{p})\}$.  \qed  
\end{defin}

\begin{exa}
    \label{Exa:H_Ages_Paths}
    Let $\cH$ be as defined in Subsection~\ref{SubSec:Organization}. So $\sfU = 1$ and $k = 3$, where $1$ and $2$ represent blue and red edges, respectively. Then $P(1) = P_0$ has four elements, which up to bi-interpretability are:
    \begin{itemize}
        \item 
        $\cH$,
        \item 
        $\cH_r$, the class of finite, triangle-free graphs with all red edges,
        \item 
        $\cH_b$, the class of finite, triangle-free graphs with all blue edges,
        \item 
        $\cH_0$, the class of graphs with no edges.
    \end{itemize}
    The set $\MP = \MP_0$ has $2$ elements, the path $\frak{p}_r:= \{\c{H}_0, \c{H}_r, \cH\}$ and the path $\frak{p}_b:= \{\cH_0, \cH_b, \cH\}$. We declare that $\frak{p}_r\leq_{\MP} \frak{p}_b$.

    For $d> 1$, a typical element of $P(d)$ is described by a pair of $(d, 3)$-graph-age functions as in Example~\ref{Exa:LFreeGraphs}, one for red cliques and one for blue cliques. While the single unary in the class $\cH$ also is not free (Definition~\ref{Def:Free}), we postpone the description of which subset of $P(d)$ will be of interest until Example~\ref{Exa:HDiaries}. \qed
\end{exa}

It will be useful to work with a more general notion of sort.

\begin{defin}
    \label{Def:Path_Sorts}
     A \emph{path sort} is a function $\rho\colon d\to \sf{MP}$ for some $d< \omega$. We set 
     \begin{align*}
        \wt{\rho} &:= (\emptyset, \sf{u}\circ \rho, \emptyset)\in \rm{Glue}(\sf{u}\circ \rho)\\[1 mm]
        P(\rho) &:= \{\cA\in P(\sf{u}\circ \rho): \forall\, i< d \, (\cA^i\in \rho(i))\}. 
     \end{align*}
     Note that if two elements of $P(\rho)$ are consecutive in $P(\rho)$, then they are also consecutive in $P(\sf{u}\circ\rho)$. Thus we can unambiguously define $\con(\rho) = \con(\sf{u}\circ\rho)\cap P(\rho)^2$ and $\econ(\rho)= \econ(\sf{u}\circ \rho)\cap P(\rho)^2$. \qed
    \end{defin}

\section{Lower Bounds: Diaries}
\label{Sec:Diaries}

Given $\bA\in \cK$, recall that a finite coloring $\gamma$ of $\emb(\bA, \bK)$ is \emph{unavoidable} if for every $\eta\in \emb(\bK, \bK)$, we have $\im(\gamma\circ \eta) = \im(\gamma)$. In particular, $\rm{BRD}(\bA, \cK)\geq \ell$ iff there is an unavoidable coloring $\gamma\colon \emb(\bA, \bK)\to \ell$. This section will produce a finite relational language $\cL^*\supseteq \cL$ and an $\cL^*$-expansion $\bK^*$ of $\bK$ such that for every $\bA\in \cK$, the map $\chi_\bA\colon \emb(\bA, \bK)\to \bK^*(\bA)$ given by $\chi_\bA(f) = \bK^*{\cdot}f$ is unavoidable (Definition~\ref{Def:Expansion}), showing that $\rm{BRD}(\bA, \cK)\geq |\bK^*(\bA)|$. To do this, we introduce a method of coding structures embeddable into $\bK$ using a tree-like object called a \emph{diary}. 

\subsection{Conventions about trees}

Let $(L, \leq_L)$ be either $(\sfU, \leq)$ or $(\MP, \leq_{\MP})$. Given $t\in L\times k^{<\omega}$, we write $t = (t^\sf{p}, t^\sf{seq})$ with $t^\sf{p}\in L$ and $t^\sf{seq}\in k^{<\omega}$. If $L = \MP$, write $t^\sf{u}:= \sf{u}(t^\sf{p})$, and if $L = \sfU$, set $t^\sf{u} = t^\sf{p}$ (here $\sf{p}$ refers to ``path" and $\sf{u}$ refers to ``unary"). Write $\ell(t) = \dom(t^\sf{seq})$, which we call the \emph{level} of $t$, and given $m< \ell(t)$, we often abuse notation and write $t(m)$ for $t^\sf{seq}(m)$. Similarly abusing notation, if $q< k$, we write $t^\frown q$ for $(t^\sf{p}, (t^\sf{seq})^\frown q)$.

\begin{defin}
    \label{Def:Binary_tree_rels}
    We define the partial orders $\leq_\ell$, $\sqsubseteq$, and $\lex$ and the partial binary operation $\wedge$ on $L\times k^{<\omega}$; let $s, t\in L\times k^{<\omega}$.
    \begin{enumerate}
        \item 
        We write $s\leq_\ell t$ iff $\ell(s)\leq \ell(t)$.
        \item 
        If $m< \ell(s)$, we write $s|_m$ or $\pi_m(s)$ for the initial segment of $s$ at level $m$. We write $s\sqsubseteq t$ if $s^\sf{p} = t^\sf{p}$ and $s^\sf{seq} = t^\sf{seq}|_{\ell(s)}$.
        \item 
        We set $s\lex t$ iff $s^\sf{p}<_L t^\sf{p}$ or $s^\sf{p} = t^\sf{p}$ and $s^\sf{seq}\lex t^\sf{seq}$ (where the latter $\lex$ is the usual lexicographic order on $k^{<\omega}$).
        \item 
        If $s^\sf{p} = t^\sf{p}$, we write $s\wedge t$ for the largest common initial segment of $s$ and $t$. \qed
    \end{enumerate}
\end{defin}

A \emph{level subset} of $L\times k^{<\omega}$ is simply a subset of $L\times k^m$ for some $m< \omega$; given a level subset $T\subseteq L\times k^{<\omega}$, write $\ell(\Delta)$ for the common level of every $t\in T$. If $S = \{s_0\lex\cdots\lex s_{d-1}\}\subseteq L\times k^{<\omega}$ is a level subset, we let $\sort(S)\colon d\to L$ be the sort or path sort (depending on $L$) given by $\sort(S)(i) = s_i^\sf{p}$. If $X\subseteq L\times k^{<\omega}$, then the $\sqsubseteq$-downwards closure of $X$ is denoted by $X{\downarrow}$.

A \emph{subtree} $\Delta\subseteq L\times k^{<\omega}$ is any $\sqsubseteq$-downwards-closed subset; thus necessarily subtrees are $\wedge$-closed as well.   The \emph{height} of $\Delta$ is $\rm{ht}(\Delta):= \{\ell(t): t\in \Delta\}\leq \omega$. Given $m< \rm{ht}(\Delta)$, we put $\Delta(m) = \{t\in \Delta: \ell(t) = m\}$. Given $x\in L$, we write $\Delta_x = \{t\in \Delta: t^\sf{p} = x\}$ and $\Delta_x(m) = \Delta_x\cap \Delta(m)$, and given $i< \sfU$, write $\Delta_i = \{t\in \Delta: t^\sf{u} = i\}$ and $\Delta_i(m) = \Delta_i\cap \Delta(m)$. Given $m< n< \rm{ht}(\Delta)$ and $S\subseteq \Delta(m)$, we put $\succ_\Delta(S, n) = \{t\in \Delta(n): t\sqsupset s\text{ for some }s\in S\}$; if $n = m+1$, we write $\is_\Delta(S)$ for $\succ_\Delta(S, m+1)$;  these are the \emph{immediate successors} of $S$ in $\Delta$. When $S = \{s\}$, we simply write $\succ(s, n)$ or $\is_\Delta(s)$. A \emph{splitting node} of $\Delta$ is any $t\in \Delta$ with $|\is_\Delta(t)|> 1$; write $\spnd(\Delta) = \{t\in \Delta: t\text{ a splitting node of }\Delta\}$. Given $s\in \Delta$ and $n> \ell(s)$, we put $\Left_\Delta(s, n)$ to be the $\lex$-least $t\in \Delta(n)$ with $t\sqsupset s$. Similarly, if $S\subseteq \Delta$ and $n> \ell(s)$ for each $s\in S$, we can write $\Left_\Delta(S, n) = \{\Left_\Delta(s, n): s\in S\}$. If the $\Delta$ subscript is omitted in various notation, we intend $\Delta = L\times k^{<\omega}$.

\begin{defin}
    \label{Def:Aged_Coding_Tree}
An \emph{aged coding tree} is a subtree $\Delta\subseteq L\times k^{<\omega}$ where additionally:
\begin{itemize}
    \item 
    We designate a subset $\cdnd(\Delta)\subseteq \Delta$ of \emph{coding nodes} of $\Delta$ with $|\cdnd(\Delta)\cap \Delta(m)|\leq 1$ for every $m< \rm{ht}(\Delta)$. We write $c^\Delta\colon |\cdnd(\Delta)|\to \cdnd(\Delta)$ for the function which enumerates the coding nodes in increasing height and $\ell^\Delta(n)$ for $\ell(c^\Delta(n))$.
    \item 
    We assign to each $m< \rm{ht}(\Delta)$ a class of $\cL_{|\Delta(m)|}$-structures in $P(\sort(\Delta(m)))$ which we denote by $\age_\Delta(m)$. 
\end{itemize} 
If $S\subseteq \Delta(m) = \{s_0\lex\cdots\lex s_{d-1}\}$ and $I = \{i< d: s_i\in S\}$, we can write $\age_\Delta(S)$ for $\age_\Delta(m){\cdot}I$, and if $d_0<\omega$ and $e\colon d_0\to \Delta(m)$ is any function, we define $\age_\Delta(e) = \age_\Delta(m){\cdot}\ol{e}$, where $\ol{e}$ is defined to satisfy $e(i) = s_{\ol{e}(i)}$. If $s\in \Delta$, we write $\age_\Delta(s)$ for $\age_\Delta(\{s\})$. We remark that $\age_\Delta(s)\in s^\sf{p}$. We let $\argpath_\Delta(s) = \{\age_\Delta(s|_m): m\leq \ell(s)\}$; for the aged coding trees we consider in this paper, we will always have $\argpath_\Delta(s)\in \path_i$. \qed
\end{defin}

\begin{defin}
    \label{Def:Crit}
Given an aged coding tree $\Delta$ and $X\subseteq \Delta$, we define the  set $\crit^\Delta(X)$ of \emph{critical levels} of $X$ in $\Delta$ by declaring that $m\in \crit^\Delta(X)$ iff any of the following happen:
\begin{itemize}
    \item 
    There is $x\in X$ with $m = \ell(x)$.
    \item 
    There are $x, y\in X$ with $m = \ell(x\wedge y)$.
    \item 
    $\age_\Delta(\pi_m[X]) \neq \age_\Delta(\pi_{m+1}[X])$.
\end{itemize}
Note that the above events need not be mutually exclusive. \qed
\end{defin}

\begin{defin}
    \label{Def:Aged_Emb}
Given aged coding trees $\Theta, \Delta\subseteq L\times k^{<\omega}$ (the same $L$ for both $\Theta$ and $\Delta$), an \emph{aged embedding} $\phi\colon \Theta\to \Delta$ is an injection satisfying all of the following.
\begin{enumerate}
    \item 
    $\phi$ preserves $\sqsubseteq$, $\wedge$, $\lex$, and $\leq_\ell$ (in both the positive and negative sense). Write $\wt{\phi}\colon \rm{ht}(\Theta)\to \rm{ht}(\Delta)$ for the function satisfying $\phi[\Theta(m)]\subseteq \Delta(\wt{\phi}(m))$.
    \item 
    $\phi[\cdnd(\Theta)]\subseteq \cdnd(\Delta)$.
    \item 
    For each $x\in L$, $\phi[\Theta_x] \subseteq \Delta_x$.
    \item 
    For each $t\in \Theta$ and $m< \ell(t)$, we have $t(m) = \phi(t)(\wt{\phi}(m))$.
    \item 
    For each $m< \rm{ht}(\Theta)$, $\age_\Theta(m) = \age_\Delta(\phi[\Theta(m)])$.
\end{enumerate}
Write $\aemb(\Theta, \Delta)$ for the set of aged embeddings from $\Theta$ to $\Delta$. \qed
\end{defin}

In this paper, we work with two types of aged coding trees. First, we recall that an \emph{enumerated structure} is simply a structure $\bA$ with $A = |A|$. The following is a mild modification of Definition 2.1 from \cite{Zuc22}.

\begin{defin}
    \label{Def:Coding_Tree}
Given an enumerated structure $\bA\leq \bK$, we define $c^\bA\colon A\to \sfU\times k^{<\omega}$ via $c^\bA(n) = (U^\bA(n), \la R^\bA(n, 0),..., R^\bA(n, n-1)\ra)$. We define the \emph{coding tree of $\bA$} to be $\ct^\bA := \im(c^\bA){\downarrow}$, and we set $\cdnd(\ct^\bA) = \im(c^\bA)$. Hence $c^\bA = c^{\ct^\bA}$, though we always write the former. To assign each level an age, consider some $n< A$, and write $\ct^\bA(n) = \{s_0\lex\cdots\lex s_{d-1}\}$. Consider the rank $d$ gluing $\gamma = (\bA|_n, \rho, \eta)$, where given $i< d$, we have $\rho(i) = s_i^\sf{p}$, and given $i< d$ and $m< n$, we have $\eta(i, m) = s_i(m)$. Note that $\gamma$ admits all unaries since $s_i\in \im(c^\bA){\downarrow}$ for each $i< d$. We set $\age_{\ct^\bA}(n) := \age_\bA(n)= \cK{\cdot}\gamma\in P(\sort(\ct^\bA(m)))$. More generally, we write $\succ_\bA(S, n)$, $\is_\bA(S)$, etc., though we note that always $\Left_\bA(S, n) = \Left(S, n)$ for $n< A$. In addition to the notational conventions for general aged coding trees, in the case of $\ct^\bA$, we define $\age_\bA(e)$ for any $e\colon d_0\to \ct^\bA$ (even when the image of $e$ is not a level set) by choosing some $n\geq \max(\{\ell(e(i)): i< d_0\})$ and setting $\age_\bA(e) = \age_\bA(\Left(e(-), n))$. \qed
\end{defin}

We differ slightly from the definition of $\ct^\bA$ given in \cite{Zuc22} in two ways. First, we now define $\ct^\bA$ to be an aged coding tree rather than just the coding function. Second, instead of recording the unary predicates just at the coding nodes, we now separate out this information at the very start by working with $\sfU\times k^{<\omega}$ instead of just $k^{<\omega}$.   

The second type of aged coding tree we will consider are \emph{diaries}, whose definition we develop in this section. As trees, diaries will be subtrees of $\MP\times k^{<\omega}$. In contrast to $\ct^\bA$, where the age of each level was completely determined by the placement of the coding nodes, we have some freedom with the age that we assign to each level of a diary. However, since we want these ages to say something meaningful about the structure we are coding (see Proposition~\ref{Prop:Diary_Correct_Age}), we will need the assignment of ages to follow some rules which more-or-less assert that our assignment is compatible with the coding and splitting of the diary and with the class $\cK$. These rules also ensure that diaries are the correct objects for encoding exact big Ramsey degrees.

\subsection{Controlled coding}

\begin{defin}
    \label{Def:CodingClass}
    Given $j< d$, a function $\phi\colon (d-1)\to k$, and $\bB$ an $\cL_{d-1}$-structure, we define $\rm{Add}_{j, \phi}(\bB)$ an $\cL_d$-structure with underlying set $B\cup \{x\}$, where $x\not\in B$ is some new point. On $B$, $\rm{Add}_{j, \phi}(\bB)$ induces the structure $(d{\setminus}\{j\}){\cdot}\bB$. As for the new point $x$, we set $V^{\rm{Add}_{j, \phi}(\bB)}(x) = j$ and given $b\in B$, we set $R^{\rm{Add}_{j, \phi}(\bB)}(b, x) = \phi(V^\bB(b))$. In words, $\rm{Add}_{j, \phi}(\mathbf{B})$ is the $\mathcal{L}_d$-structure obtained from $\bB$ by shifting the $V$ values from $d-1$ to $d\setminus \{j\}$ via 
    increasing bijection, preserving the binary relations on pairs of vertices in $B$, giving one new vertex $x$ the $V$-value $j$, and letting 
    $\phi$ determine
    the binary relations between $x$ and the vertices in $B$. Thus $\rm{Add}_{j, \phi}(\bB)$ contains exactly one point with unary $j$, namely $x$.

    Given a class $\cA$ of $\cL_d$-structures, we define 
    $$\cA{\cdot}\rm{Add}_{j, \phi} = \{\bB\in \fin(\cL_{d-1}): \rm{Add}_{j, \phi}(\bB)\in \cA\}.$$ 
    With our convention identifying classes of structures with their characteristic function, we have $(\cA{\cdot}\rm{Add}_{j, \phi})(\bB) = \cA(\rm{Add}_{j, \phi}(\bB))$. \qed
\end{defin}

We note that if $\bA\leq \bK$ is an enumerated structure, $S = \{s_0,...,s_{d-1}\}\subseteq \ct^\bA(m)$ satisfies $\age_\bA(S) = \cA$, $s_j = c^\bA(m)$, and $T = \{t_0,..., t_{d-1}\}\subseteq \ct^\bA(m+1)$ is such that $t_i = (s_{(d{\setminus}\{j\})(i)})^\frown \phi(i)$ for each $i< d-1$, then $\age_\bA(T) = \cA{\cdot}\rm{Add}_{j, \phi}$. 
    
More generally, if $\cA = \cK{\cdot}\gamma$ for some gluing $\gamma$, then there is also a gluing $\gamma'$ such that $\cA{\cdot}\rm{Add}_{j, \phi} = \cK{\cdot}\gamma'$.

\begin{defin}
\label{Def:ControlledCodingTriple}   
    Given a path sort $\rho\colon d\to \sf{MP}$, $\cA\in P(\rho)$, $j< d$, and $\phi\colon (d-1)\to k$, we call $(\cA, j, \phi)$ a \emph{controlled coding triple} if the following all hold.
    \begin{enumerate}
        \item 
        $\cA{\cdot}\rm{Add}_{j, \phi} = \cA{\cdot}(d{\setminus}\{j\})$.
        \item
        Writing $i = \sf{u}\circ \rho(j) < \sfU$, we have $\cA^j = \min(P_i)$.
        \item
        If $\cB\in P(\rho)$ satisfies $\cB\subseteq \cA$ and $\cB{\cdot}(d{\setminus}\{j\}) = \cB{\cdot}\rm{Add}_{j, \phi} = \cA{\cdot}(d{\setminus}\{j\})$, then $\cB = \cA$.
    \end{enumerate}

    Call $(\c{A}, j)$ a \emph{controlled coding pair} if there is some function $\phi\colon (d-1)\to k$ making $(\c{A}, j, \phi)$ a controlled coding triple. \qed
\end{defin}

\begin{exa}
    \label{Exa:LFreeCoding}
    Consider the class $\c{G}_\ell$ of $K_\ell$-free finite graphs. Fix $d< \omega$, $z$ a $(d,\ell)$-graph-age function (see Example~\ref{Exa:LFreeGraphs}), $j< d$ and $\phi\colon (d-1)\to 2$. In order to understand when $(\c{G}_z, j, \phi)$ is a controlled coding triple, we first need to understand when $\c{G}_z{\cdot}\rm{Add}_{j, \phi} = (\c{G}_z){\cdot}(d{\setminus\{j\}})$. Writing $X = (d{\setminus}\{j\})[\phi^{-1}(\{1\})]$, one can show that this happens exactly when both of the following occur:
    \begin{enumerate}
        \item 
        For every $T\in \binom{X}{{<}\ell}$, we have $z(T)< \ell-1$,
        \item
        For every $T\in \binom{X}{{<}(\ell-1)}$, we have $z(T\cup \{j\}) = z(T)+1$.
    \end{enumerate}
    Now suppose $(\c{G}_z, j, \phi)$ satisfies the above two items. When is it a controlled coding triple? First, we need $z(\{j\}) = 1$ so that item $1$ of Definition~\ref{Def:ControlledCodingTriple} holds. Aside from this, we need to enforce item $2$ by ensuring that $z$ is ``as small as possible." Given $S\in \binom{d}{{<}\ell}$, call $S$ \emph{determined} if either $S\subseteq d\setminus \{j\}$ or $S\subseteq X\cup \{j\}$; we demand for every $T\in \binom{d}{<{\ell}}$ that $z(T) = \max\{z(S): S\subseteq T \text{ is determined}\}$. 
    
    In summary, $(\c{G}_z, j)$ is a controlled coding pair iff all of the following hold: 
    \begin{enumerate}
        \item 
        $z(\{j\}) = 1$,
        \item
        Writing $X = \{\alpha\in d\setminus \{j\}: z(\{\alpha, j\}) = z(\{\alpha\})+1\}$, then for every $T\in \binom{X}{{<}\ell}$, we have $z(T)< \ell-1$, and for every $T\in \binom{X}{{<}(\ell-1)}$, we have $z(T\cup \{j\}) = z(T)+1$,
        \item
        With the same notion of ``determined" as above, we have for every $T\in \binom{d}{<{\ell}}$ that $z(T) = \max\{z(S): S\subseteq T \text{ is determined}\}$. 
    \end{enumerate}
    It turns out that whenever $(\c{G}_z, j)$ is a controlled coding pair, then there is a unique $\phi$ making $(\c{G}_z, j, \phi)$ a controlled coding triple; $\phi$ is defined for $\alpha < (d-1)$ by setting $\phi(\alpha) = 1$ iff $(d{\setminus}\{j\})(\alpha)\in X$. \qed
\end{exa}

We will discuss the motivation for controlled coding triples after stating Theorem~\ref{Thm:EmbDiary}. We first show that they exist in abundance. 

\begin{lemma}
    \label{Lem:Controlled_Coding_triples}
    Given a path sort $\rho\colon d\to \sf{MP}$, $\cA\in P(\rho)$, $j< d$, and $\phi\colon (d-1)\to k$, then if $\cA{\cdot}\rm{Add}_{j, \phi}$ contains all unaries, then there is a unique $\cB\in P(\rho)$ with $\cB\subseteq \cA$, $\cB{\cdot}\rm{Add}_{j, \phi} = \cA{\cdot}\rm{Add}_{j, \phi}$, and with $(\cB, j, \phi)$ a controlled coding triple. 
\end{lemma}

We denote this unique $\cB$ by $\la \cA, j, \phi\ra$.

\begin{proof}
    Write $\cC = \cA{\cdot}\rm{Add}_{j, \phi}$. We let 
    $$\c{B} = \bigcap \{\c{D}\in P(\rho): \cD{\cdot}\rm{Add}_{j, \phi} = \cC\}.$$
    We have $\c{B}\in P(\rho)$ by Proposition~\ref{Prop:Finite_Posets}, and we also have $\c{B}{\cdot}\rm{Add}_{j, \phi} = \cC$. We also see that for any $\c{B}'\in P(\rho)$ with $\c{B}'\subsetneq\c{B}$, we must have $\c{B}'{\cdot}\rm{Add}_{j, \phi}\subsetneq \cC$, showing that item $3$ of Definition~\ref{Def:ControlledCodingTriple} holds. To see that items $1$ and $2$ hold, write $i = \sf{u}\circ \rho(j)$, and consider the classes $\c{A}\cap \min(P_i){\cdot}\iota_j^{-1}$ and $\c{A}\cap \cC{\cdot}(d{\setminus}\{j\})^{-1}$. Letting $\c{D}$ denote either of these classes, we have that $\c{D}\in P(\rho)$ by Proposition~\ref{Prop:Manipulating_Classes} (in the second case, this uses that $\cC$ contains all unaries as well as the observation immediately before Definition~\ref{Def:ControlledCodingTriple}), and it is straightforward to show that $\cD{\cdot}\rm{Add}_{j, \phi} = \cC$. This implies that $\c{B}^j = \min(P_i)$ and $\c{B}{\cdot}(d{\setminus}\{j\}) = \cC$, showing that items $1$ and $2$ hold.
\end{proof}

Next, we show that restricting a controlled coding triple to a subset of coordinates gives rise to another controlled coding triple.

\begin{lemma}
    \label{Lem:Controlled_Coding_Restriction}
    Let $\rho\colon d\to \MP$ be a path sort, $\cA\in P(\rho)$, $j< d$ and $\phi\colon (d-1)\to k$ be such that $(\cA, j, \phi)$ is a controlled coding triple. Let $e\colon d_0\to d$ be an injection and $j_0< d_0$ satisfy $e(j_0) = j$. Write $\sigma\colon (d_0-1)\to (d-1)$ for the injection satisfying $e\circ (d_0{\setminus}\{j_0\}) = (d{\setminus}\{j\})\circ \sigma$. Then $(\cA{\cdot}e, j_0, \phi\circ \sigma)$ is a controlled coding triple.
\end{lemma}

\begin{proof}
   We verify the three items of Definition~\ref{Def:ControlledCodingTriple}. 
    
   For item $1$, we have:
    \begin{align*}
        (\c{A}{\cdot}e){\cdot}(d_0{\setminus}\{j_0\}) &= (\c{A}{\cdot}(d{\setminus}\{j\})){\cdot}\sigma\\
        &= (\cA{\cdot}\rm{Add}_{j, \phi}){\cdot}\sigma\\
        &= (\c{A}{\cdot}e)(j_0, \phi\circ \sigma).
    \end{align*}
    
     For item $2$, writing $i = \sf{u}\circ \rho(j)$, we have $\c{A}^j = \min(P_i)$, so $(\c{A}{\cdot}e)^{j_0} = \min(P_i)$ as well.

    For item $3$, suppose $\c{B}\subseteq \c{A}{\cdot}e$ satisfies $\c{B}{\cdot}(d_0{\setminus}\{j_0\}) = \c{B}(j_0, \phi\circ \sigma) = (\c{A}{\cdot}e)(j_0, \phi\circ \sigma)$. We need to show that $\c{B} = \c{A}{\cdot}e$. To that end, form the class $\c{C} := \c{A}\cap \cB{\cdot}e^{-1}$. It is routine to check that $\c{C}{\cdot}(d{\setminus}\{j\}) = \c{A}{\cdot}(d{\setminus}\{j\})$. To check that $\c{C}{\cdot}\rm{Add}_{j, \phi} = \cA{\cdot}\rm{Add}_{j, \phi}$, the left-to-right inclusion is clear. For the right-to-left inclusion, fix $\b{D}\in \cA{\cdot}\rm{Add}_{j, \phi}$. This means that $\rm{Add}_{j, \phi}(\b{D})\in \c{A}$, and we want to show that it is also in $\c{C}$. To that end, we note that $e^{-1}{\cdot}(\rm{Add}_{j, \phi}(\b{D})) = (j_0, \phi\circ \sigma)(\sigma^{-1}{\cdot}\bD)$. Then $\sigma^{-1}{\cdot}\b{D}\in (\c{A}{\cdot}e)(j_0, \phi\circ \sigma) = \c{B}(j_0, \phi\circ \sigma)$. So $e^{-1}{\cdot}(\rm{Add}_{j, \phi}(\b{D})) = (j_0, \phi\circ \sigma)(\sigma^{-1}{\cdot}(\bD))\in \c{B}$, which means that $\rm{Add}_{j, \phi}(\b{D})\in \c{C}$. Because $(\c{A}, j, \phi)$ is a controlled coding triple, we must have $\c{A} = \c{C}$, implying that $\c{A}{\cdot}e = \c{B}$.    
\end{proof}

\subsection{Abstract splitting events}

Before proceeding further, it will help to assume that every unary predicate in $\cL$ interacts with suitably interesting members of $\cK$.

\begin{defin}
    \label{Def:NonDegenerate}
    Fix $i< \sfU$. We call $i$ \index{non-degenerate}\emph{non-degenerate for $\c{K}$} if there is some $\bA\in \c{K}$ with $A = \{a, b\}$, $U^\bA(a) = i$, and $R^\bA(b, a)\neq 0$. Otherwise, we call $i$ \emph{degenerate}. \qed
\end{defin}

A good example to keep in mind is the class $\cK$ of finite sets, where the single unary is degenerate. Degenerate unaries will behave rather differently than non-degenerate ones, but will be easy to deal with at the very end. So \textbf{until Subsection~\ref{Subsec:Degenerates}, we will assume that every $i< \sfU$ is non-degenerate for $\c{K}$}. 

Given a path sort $\rho\colon d\to \sf{MP}$, a class $\cA\in P(\rho)$, and some $i <  d$, we discuss an abstract notion of what it should mean to ``split" the class $\cA$ at position $i$. This is the operation that will occur at ``splitting levels" of diaries. First, let $\sp(d, i)\colon (d+1)\to d$ be the unique non-decreasing surjection with $\sp(d, i)^{-1}(\{i\}) = \{i, i+1\}$. If $\sigma$ is any function with domain $d$, we can write $\sp(\sigma, i) := \sigma\circ \sp(d, i)$. In particular, we obtain a new sort $\sp(\rho, i)$. 

Our goal is to take $\cA\in P(\rho)$ and form $\sp(\cA, i)$ (read ``split $\cA$ at $i$"). The definition of $\sp(\cA, i)$ will depend on whether or not $\sf{u}\circ \rho(i)$ is a \emph{free} or $\emph{non-free}$ unary. 

\begin{defin}
    \label{Def:Free}
    Fix $i< \sfU$. We say that $i$ is \emph{free} if there are $j< \sfU$ and $1\leq q< k$ such that, letting $\gamma_{j, i, q} = (\bX, \iota_i, \eta)$ denote the gluing with $X = \{0\}$, $U^\bX(0) = j$, and $\eta(0, 0) = q$, then $\cK{\cdot}\gamma_{j, i, q} = \cK{\cdot}\tilde{\iota}_i = \max(P_i)$. We call such $(j, q)$ a \emph{free pair} for $i$.

    Otherwise, we call $i$ a \emph{non-free} unary. Let $\sfU_{fr}\subseteq \sfU$ denote the set of free unaries and $\sf{U}_{non} = \sfU\setminus \sf{U}_{fr}$ the set of non-free unary predicates. 

    We call $\frak{p}\in \sf{MP}$ a \emph{free path} or a \emph{non-free path} depending on whether $\sf{u}(\frak{p})$ is free or non-free. Write $\sf{MP}_{fr}$ for the free paths and $\sf{MP}_{non}$ for the non-free paths. \qed
\end{defin}

\begin{exa}
    For $\ell\geq 3$, the single unary $i = 0$ for the class $\c{G}_\ell$ of finite $K_\ell$-free graphs (Example~\ref{Exa:LFreeGraphs}) is non-free. The only choice for $(j, q)$ would be $(0, 1)$, but the class $\c{K}{\cdot}\gamma_{0, 0, 1}$ is bi-interpretable with $\c{G}_{\ell-1}$, the class of finite graphs which forbid $\cK_{\ell-1}$. 
    
    The single unary $i = 0$ in the class $\cG_\bT$ of finite directed graphs forbidding cyclic triangles (Example~\ref{Exa:TFreeAges}) is free, as witnessed by $(j, q) = (0, 1)$ or $(0, 2)$. \qed 
\end{exa}

\begin{fact}
    \label{Fact:NonFree_Unaries}
    For $i\in \sfU_{non}$, while we cannot find $j< \sfU$ and $1\leq q< k$ with $\cK{\cdot}\gamma_{j, i, q} = \max(P_i)$, we do have the next best thing: If $\frak{p}\in \MP_i$, then there are $j< \sfU$ and $1\leq q< k$ so that $\cK{\cdot}\gamma_{j, i, q} = \max(\frak{p}')$. \qed 
\end{fact}

\begin{defin}
    \label{Def:SplittingClass}
    With notation as above, $\sp(\cA, i)\in P(\sp(\rho, i))$ is defined as follows.
    \begin{enumerate}
        \item 
        If either $\rho(i)\in \sf{MP}_{fr}$ or $\cA^i\neq \max(\frak{p})$, we set $\sp(\cA, i) = \cA{\cdot}\sp(d, i)$.
        \item
        If $\rho(i)\in \sf{MP}_{non}$ and $\cA^i = \max(\frak{p})$, we set
        $\sp(\cA, i) = \cA{\cdot}\sp(d, i)\cap \max(\frak{p}'){\cdot}\iota_{i+1}^{-1}$. \qed
    \end{enumerate}
\end{defin}
In case $2$, notice that $\sp(\cA, i)^{i+1} = \max(\frak{p}')$. We remark that this is one of the main ways that the difference between free and non-free paths will arise. 

The difference between these two cases is motivated by what happens in coding trees. If $\bA\leq \bK$ is an enumerated structure, $i\in \sfU_{non}$, and $v = (i, 0^n)\in \ct^\bA$, then $\age_\bA(v) = \cK{\cdot}\tilde{\iota}_i$, but for any $0\neq q < k$ with $v^\frown q \in \ct^\bA$, we must have $\age_\bA(v^\frown q) \subsetneq \cK{\cdot}\tilde{\iota}_i$. Hence starting from $v$, splitting and going right must result in an age change. As diaries are defined to mimic certain properties of coding trees, we must incorporate this feature into our definition of a diary.

Let us clarify one aspect of case $2$, which we phrase a touch more generally.

\begin{lemma}
    \label{Lem:Split_Consecutive}
    Let $\rho\colon d\to \MP$ be a path sort, fix $i< d$, and suppose $\rho(i) = \frak{p}\in \MP_{non}$. Then if $\cB\in P(\rho)$ is such that $\cB^i = \max(\frak{p})$, then $(\cB, \cB\cap \max(\frak{p}'){\cdot}\iota_i^{-1})\in \con(\rho)$. 
\end{lemma}

\begin{proof}
    Write $\cD = \cB\cap \max(\frak{p}'){\cdot}\iota_i^{-1}$. Suppose $\cC\in P(\rho)$ satisfied $\cD\subseteq \cC\subseteq \cB$. We first observe that since $\cD{\cdot}(d{\setminus}\{i\}) = \cB{\cdot}(d{\setminus}\{i\})$, we have $\cC{\cdot}(d{\setminus}\{i\}) = \cB{\cdot}(d{\setminus}\{i\})$. Write $\cC = \cK{\cdot}\gamma$ for some gluing $\gamma = (\bX, \rho, \eta)$. If $\eta(i, x)\neq 0$ for some $x\in X$, then as $\frak{p}\in \MP_{non}$, we would have $\cC^i\subsetneq \max(\frak{p})$, implying that $\cC^i\subseteq \max(\frak{p}')$, and hence $\cC = \cD$. However, if $\eta(i, x) = 0$ for every $x\in X$, then free amalgamation in $\cK$ along with our first observation imply that $\cC = \cB$. 
\end{proof}

We conclude the subsection by discussing what happens upon restricting our attention to a subset of $d+1$. 

\begin{prop}
    \label{Prop:Split_restrict}
    With notation as in Definition~\ref{Def:SplittingClass}, fix $d_0< \omega$ and an increasing injection $e\colon d_0\to (d+1)$. 
    \begin{enumerate}
        \item 
        If $\{i, i+1\}\subseteq \im(e)$, then letting $i_0< d_0$ be such that $e(i_0) = i$, and letting $e_0\colon (d_0-1)\to d$ be the map with $e_0\circ \sp(d_0-1, i_0) = \sp(d, i)\circ e$, we have $\sp(\cA, i){\cdot}e = \sp(\cA{\cdot}e_0, i_0)$.
        \item 
        If $\{i, i+1\}\not\subseteq \im(e)$, then if either $i+1\not\in \im(e)$ or if $i+1\in \im(e)$ and we are in case $1$ of Definition~\ref{Def:SplittingClass}, then $\sp(\cA, i){\cdot}e = \cA{\cdot}(\sp(i, d)\circ e)$.
        \item 
        If $i\not\in \im(e)$, $i+1\in\im(e)$, and we are in case $2$ of Definition~\ref{Def:SplittingClass}, then if $i_0< d_0$ satisfies $e(i_0) = i+1$, then $\sp(\cA, i){\cdot}e = \cA{\cdot}(\sp(i, d)\circ e)\cap \max(\frak{p}'){\cdot}\iota_{i_0}^{-1}$. In particular, by Lemma~\ref{Lem:Split_Consecutive} we have $(\cA{\cdot}(\sp(d, i)\circ e), \sp(\cA, i){\cdot}e)\in \con(\sp(\rho, i)\circ e)$.
    \end{enumerate}
\end{prop}

\begin{proof}
    All three items are relatively straightforward checking of the various cases involved. We omit the details.
\end{proof}

\subsection{Diaries and their embeddings}

\begin{defin}
    \label{Def:Diary}    
    A \emph{diary} is an aged coding tree $\Delta\subseteq \sf{MP}\times k^{<\omega}$ satisfying the following.
    \begin{enumerate}
        \item 
        Every non-empty level of $\Delta(m)$ contains at most one node $t$ with $|\is_\Delta(t)|\neq 1$. 
        \begin{itemize}
            \item
            If there is $t\in \Delta(m)$ with $\is_\Delta(t) = \emptyset$, we call $m$ a \emph{coding level} of $\Delta$.
            \item
            If there is $t\in \Delta(m)$ with $|\is_\Delta(t)|> 1$, then $|\is_\Delta(t)| = 2$, and we call $m$ a \emph{splitting level} of $\Delta$.
            \item
            If every $t\in \Delta(m)$ satisfies $|\is_\Delta(t)| = 1$, we call $m$ an \emph{age-change} level of $\Delta$.
        \end{itemize}
        We write $\sp(\Delta)$, $\cd(\Delta)$, and $\ac(\Delta)$ for the splitting, coding, and age-change levels of $\Delta$, respectively. We let $\cdnd(\Delta)\subseteq \Delta$ be the set of terminal nodes of $\Delta$. 
        \item
        $\cdnd(\Delta)$ is $\sqsubseteq$-upwards cofinal in $\Delta$.

        \item
        Writing $\rho = \sort(\Delta(0))$, we have $\age_\Delta(0) = \cK{\cdot}\tilde{\rho} = \max(P(\rho))$.
        \item
        If $\Delta(m) = \{t_0\lex\cdots\lex t_{d-1}\}$ and $t_i$ is the splitting node, then we have $\Delta(m+1) = \{t_j^\frown 0: j< d\}\cup \{t_i^\frown 1\}$ and $\age_\Delta(m+1) = \sp(\age_\Delta(m), i)$ (Definition~\ref{Def:SplittingClass}).
        \item
        If $\Delta(m) = \{t_0\lex\cdots\lex t_{d-1}\}$ and $t_j$ is a coding node, then writing $\Delta(m+1) = \{u_0\lex\cdots\lex u_{d-2}\}$ and defining $\phi\colon (d-1)\to k$ via $\phi(i) = u_i(m)$, then $(\age_\Delta(m), j, \phi)$ is a controlled coding triple (Definition~\ref{Def:ControlledCodingTriple}) and $\age_\Delta(m+1) = \age_\Delta(m){\cdot}\rm{Add}_{j, \phi} = \age_\Delta(m){\cdot}(d{\setminus}\{j\})$.
        \item
        If $m\in \ac(\Delta)$, then we have $\Delta(m+1) = \{t^\frown 0: t\in \Delta(m)\}$, and writing $\rho = \sort(\Delta(m))$, we have $(\age_\Delta(m), \age_\Delta(m+1))\in \con(\rho)$ (Definition~\ref{Def:Consecutive}). \qed
    \end{enumerate}
\end{defin}

\begin{fact}
\label{Fact:Basic_Diary_Facts}
\begin{enumerate}
    \item 
    If $\Delta$ is a diary and $\frak{p}\in \MP_{non}$, then for any $s\in \Delta$ with $s^\sf{p} = \frak{p}$ with $s\neq (\frak{p}, 0^{\ell(s)})$, then $\age_\Delta(s)\subsetneq \max(\frak{p})$. This follows from case $2$ of Definition~\ref{Def:SplittingClass} and part $1$ of Definition~\ref{Def:ControlledCodingTriple}. 
    \item
    For any $t\in \cdnd(\Delta)$, we have $\argpath_\Delta(t) = t^\sf{p}$.  By part $2$ of Definition~\ref{Def:ControlledCodingTriple}, we have $\age_\Delta(t) = \min(P_i)$. Item $6$ of Definition~\ref{Def:Diary} then implies that $\argpath_\Delta(t)\in \MP$. But since $\age_\Delta(t|_m)\in P(\sort(t)) = t^\sf{p}$ for every $m\leq \ell(t)$, we must have $\argpath_\Delta(t) = t^\sf{p}$. \qed
\end{enumerate}
\end{fact}

Every level subset of a diary has a distinguished subset of \emph{critical nodes}. If $\Delta$ is a diary and $m< \rm{ht}(\Delta)$, we define $\critnd_\Delta(m)\subseteq \Delta(m)$ as follows. If $m\in \sp(\Delta)$, we set $\critnd_\Delta(m) = \spnd(\Delta)\cap \Delta(m)$. If $m\in \cd(\Delta)$, we set $\critnd_\Delta(m) = \cdnd(\Delta)\cap \Delta(m)$. If $m\in \ac(\Delta)$, then by Proposition~\ref{Prop:Consecutive},  $(\age_\Delta(m), \age_\Delta(m+1))$ is essential on a unique $S\subseteq \Delta(m)$, and we set $\critnd_\Delta(m) = S$. 

Given a diary $\Delta$, the \emph{structure coded by $\Delta$}, denoted $\str(\Delta)$, is the $\cL$-structure on underlying set $\cdnd(\Delta)$ such that given $s, t\in \cdnd(\Delta)$, we set $U^{\str(\Delta)}(t) = t^\sf{u}$, and if $\ell(t)> \ell(s)$, we put $R^{\str(\Delta)}(t, s) = t(\ell(s))$. We let $\str^\#(\Delta)$ denote the enumerated structure isomorphic to $\str(\Delta)$ via $c^\Delta$. More generally, given $m< \rm{ht}(\Delta)$ and $S = \{s_0\lex\cdots\lex s_{d-1}\}\subseteq \Delta(m)$, we define the $\cL_d$-structure $\str(\Delta)/S$ on underlying set $\{t\in \cdnd(\Delta): t\sqsupseteq s \text{ for some }s\in S\}$ by setting $V^{\str(\Delta)/S}(t) = j$ iff $t\sqsupseteq s_j$ and with the $\cL^\sf{b}$-part induced from $\str(\Delta)$. Much of the motivation behind Definition~\ref{Def:Diary} is to guarantee a nice cohesion between the ages assigned to each level of $\Delta$ and $\str(\Delta)$. We clarify this in Proposition~\ref{Prop:Diary_Correct_Age}.

\begin{prop}
    \label{Prop:Diary_Correct_Age}
    Let $\Delta$ be a diary, $m< \rm{ht}(\Delta)$, and $S\subseteq \Delta(m)$. 
    \begin{enumerate}
        \item 
        $\rm{Age}(\str(\Delta)/S)\subseteq \age_\Delta(S)$. In particular, we always have $\rm{Age}(\str(\Delta))\subseteq \cK$. 
        \item
        Let $n< \omega$ satisfy $\ell^\Delta(n-1)< m$. Define $\eta\colon |S|\times n\to k$ where given $(i, a)\in |S|\times n$, we set $\eta(i, a) = s_i(\ell^\Delta(a))$. Then writing $\gamma = (\str^\#(\Delta)|_n, \sort(S), \eta)$, we have $\age_\Delta(S)\subseteq \cK{\cdot}\gamma$.  
    \end{enumerate}
\end{prop}

\begin{proof}
    $(1)$ Given a finite $\bB\subseteq \str(\Delta)/S$, we prove $\bB\in \age_\Delta(S)$ by induction on $|B|$ for every level set $S\subseteq \Delta$ simultaneously. For $|B| = 1$, the result is clear. So suppose $|B|> 1$, and let $b\in B$ be $\leq_\ell$-least. Write $\pi_{\ell(b)}[B] = \{t_0\lex\cdots\lex t_{\alpha-1}\}:= T$, and let $j< \alpha$ satisfy $b = t_j$. For each $i\in \alpha{\setminus}\{j\}$, let $q_i< k$ be unique with $t_i^\frown q_i\in \Delta$, and let $\phi\colon (\alpha-1)\to k$ be the function with $\phi(i) = q_{(\alpha{\setminus}\{j\})(i)}$. By induction, viewing $\bB{\setminus}\{b\}$ as an $\cL_\alpha$-structure, we have $\bB{\setminus}\{b\}\in \age_\Delta(T)$. Then, since $(\age_\Delta(T), j, \phi)$ is a controlled coding triple, we have, viewing $\bB$ as an $\cL_\alpha$-structure, that $\bB\in \age_\Delta(T)$. If $S = \{s_0\lex\cdots\lex s_{d-1}\}$ and $\pi\colon \alpha\to d$ denotes the function with $t_i\sqsupseteq s_{\pi(i)}$ for each $i< \alpha$, then by the properties of splitting, age-change, and coding levels of diaries, we have $\age_\Delta(T)\subseteq \age_\Delta(S){\cdot}\pi$. This implies that $\bB$, as an $\cL_d$-structure, belongs to $\age_\Delta(S)$. 

    $(2)$ We induct on $n$ for every level set $S\subseteq \Delta$ simultaneously. For $n = 0$, there is nothing to show. If $n> 0$, it suffices to consider the case where $S = \{s_0\lex\cdots\lex s_{d-2}\}\subseteq \Delta(\ell^\Delta(n-1)+1)$. Write $\pi_{\ell^\Delta(n-1)}[S]\cup \{c^\Delta(n-1)\} := T:= \{t_0\lex\cdots\lex t_{d-1}\}$, and suppose $j< d$ satisfies $c^\Delta(n) = t_j$. Write $\phi\colon (d-1)\to k$ for the function $\phi(i) = s_i(\ell^\Delta(n-1))$. Since $(\age_\Delta(T), j, \phi)$ is a controlled coding triple by Lemma~\ref{Lem:Controlled_Coding_Restriction}, we have $\age_\Delta(S) = \age_\Delta(T){\cdot}(d{\setminus}\{j\}) = \age_\Delta(T){\cdot}\rm{Add}_{j, \phi}$. Consider the gluing $\gamma' = (\str^\#(\Delta)|_{n-1}, \sort(T), \eta')$, where $\eta'\colon (d\times (n-1))\to k$ is given by $\eta'(i, a) = t_i(\ell^\Delta(a))$. By induction, we have $\age_\Delta(T)\subseteq \cK{\cdot}\gamma'$. It is routine to check that $(\cK{\cdot}\gamma'){\cdot}\rm{Add}_{j, \phi} = \cK{\cdot}\gamma$, where $\gamma$ is as in the proposition statement. Hence $\age_\Delta(S)\subseteq \cK{\cdot}\gamma$ as desired. 
\end{proof}

We now discuss finite diaries for our three main examples; the reader who wishes to skip them can skip to Definition~\ref{Def:EmbDiary}. One quirk of all three examples is that controlled coding pairs uniquely determine controlled coding triples. In practice, this means that if $\Delta$ is a diary, $m\in \cd(\Delta)$, $t\in \Delta(m)$ is the coding node, and $v\in \Delta({>}m)$, then $v(m)$ is completely determined by $\age_\Delta(\{v|_m, t\})$. For classes where controlled coding pairs uniquely determine controlled coding triples, diaries are in one-one correspondence with certain sequences of structures (typically in another language). It is not always the case that controlled coding pairs uniquely determine controlled coding triples; an example where this is not the case is the class $\cG$ of finite graphs. For these classes, at coding levels one must explicitly describe the ``passing number" information.

\begin{exa}
    \label{Exa:3FreeDiaries}
    We describe all finite diaries for the class $\c{G}_3$ of triangle free graphs. By Fact~\ref{Fact:Basic_Diary_Facts}(1), if $\Delta(m) = \{s_0\lex\cdots\lex s_{d-1}\}$ and $\age_\Delta(s_i) = \cG_3{\cdot}\tilde{\iota}_0$, then $i = 0$, and furthermore, $\age_\Delta(m)$ can be recovered from $\age_\Delta(\{s_1\lex\cdots\lex s_{d-1})$. Thus, given a diary $\Delta$, we can recover $\Delta$ just from the nodes of $\Delta$ without maximum age. As in Example~\ref{Exa:LFreeGraphs}, if $z$ is a $(d, 3)$-graph-age function with $z(\{i\}) = 1$ for every $i< d$, we can encode $z$ using a graph on vertex set $d$. Using these ideas, we obtain the following general procedure to produce a diary coding a finite enumerated triangle-free graph:
    \begin{enumerate}
        \item 
        At stage $0$, we let $G_0$ be the empty graph. 
        \item
        Suppose at stage $m< \omega$, we have produced a finite linearly-ordered graph $G_m$. If $m> 0$ and $G_m$ is the empty graph, we can choose to stop. 
        \item
        If at stage $m$ we choose to (or must) continue, we perform exactly one of the following to produce $G_{m+1}$:
        \begin{enumerate}
            \item 
            Introduce a new least vertex which is adjacent to every vertex from $G_m$ (the node $0^m$ has maximum age and splits).
            \item
            If $G_m\neq \emptyset$, choose $v\in G_m$ and duplicate it; the two new vertices will be order consecutive and graph non-adjacent in $G_{m+1}$, and will have the same neighbors in $G_m\setminus \{v\}$ as $v$ did (a node of non-maximal age splits).
            \item
            Delete an edge in $G_m$ (an age change level).
            \item
            Pick $v\in G_m$ whose neighbors form an anti-clique and delete $v$ (a coding level).
        \end{enumerate}
    \end{enumerate}
    Finite diaries for the class $\cG_3$ are in one-one correspondence with finite runs of this procedure. The size of the enumerated triangle-free graph which is coded by this procedure is exactly the number of vertex deletion events, and the edges of this triangle-free graph are determined by the edges in $G_m$ at the time of the vertex deletion events. For instance, if we wish to code an enumerated anti-clique, this is the same as demanding that we only delete isolated vertices from $G_m$. Thus the number of diaries coding the $n$-element anti-clique for $n = 1, 2, 3, 4$ are $1$, $5$, $161$, $134397$. It would be an interesting problem in enumerative combinatorics to obtain an exact formula, or to even understand how this function grows asymptotically (see Subsection~\ref{SubSec:Asymptotics}). \qed
\end{exa}

\begin{exa}
    \label{Exa:TFreeDiaries}
    We describe all finite diaries for the class $\cG_\bT$ of finite \emph{oriented} graphs which forbid the $3$-cycle $\bT$. From Example~\ref{Exa:TFreeAges}, members of $P(d)$ are in one-one correspondence with \emph{directed} graphs on vertex set $d$. Thus we have the following general procedure to produce a diary coding a finite enumerated $\bT$-free oriented graph; note the slight differences from Example~\ref{Exa:3FreeDiaries} in the starting and ending conditions since here, the unary is free. 
    \begin{enumerate}
        \item 
        At stage $0$, we let $G_0$ be the singleton graph.
        \item 
        Suppose at stage $m< \omega$, we have produced a finite linearly-ordered \emph{directed} graph $G_m$. If $G_m = \emptyset$, we must stop.
        \item 
        If $G_m\neq \emptyset$, we perform exactly one of the following to produce $G_{m+1}$:
        \begin{enumerate}
            \item 
            Choose $v\in G_m$ and duplicate it. The two new vertices will be order consecutive and form a $2$-cycle, and they will have the same relations with vertices in $G_m\setminus \{v\}$ as $v$ did (a node splits).
            \item 
            Delete an edge in $G_m$ (an age-change level).
            \item 
            Pick $v\in G_m$ which does not participate in any $2$-cycles and such that there are no edges from the out-neighbors of $v$ to the in-neighbors of $v$ and delete it (a coding level).
        \end{enumerate}
    \end{enumerate}
    Finite diaries for the class $\cG_\bT$ are in one-one correspondence with finite runs of this procedure. The size of the enumerated oriented $\bT$-free graph coded by this procedure is exactly the number of vertex deletion events, and the oriented edges are determined by the edges in $G_m$ at the time of vertex deletion events. \qed
\end{exa}

\begin{exa}
    \label{Exa:HDiaries}
    The description of finite diaries for $\cH$ is slightly more complicated since $|\MP| = 2$. As promised in Example~\ref{Exa:H_Ages_Paths}, we describe the subset of $P(d)$ which is relevant to our description of diaries. We do so implicitly by describing the structures that appear in a sequence describing a typical diary. These will be linearly ordered graphs with red and blue edges where \emph{loops are allowed} and pairs can be unrelated, red related, blue related, \emph{or both}. In addition, there will be two disjoint unary predicates $P_r$ and $P_b$ which partition the vertex set, and $P_r$ will form an initial segment of vertices in the linear order. Let us call structures of this form \emph{RB-graphs}. With these structures, we can describe finite diaries for $\cH$ via finite runs of the following procedure.
    \begin{enumerate}
        \item 
        At stage $0$, $G_0$ is the empty RB-graph.
        \item 
        Suppose at stage $m<\omega$, we have produced an RB-graph $G_m$. If $m> 0$ and $G_m = \emptyset$, we can choose to stop.
        \item
        If at stage $m$ we choose to (or must) continue, we perform exactly one of the following to produce $G_{m+1}$:
        \begin{enumerate}
            \item 
            Introduce a new least vertex among those with unary $P_r$. This vertex will be red adjacent to every member of $P_r$ (including itself), and will be red-and-blue adjacent to every member of $P_b$ (the node $(\frak{p}_r, 0^m)$ has maximum age and splits).
            \item 
            Introduce a new least vertex among those with unary $P_b$. Same as above, switching red and blue (the node $(\frak{p}_b, 0^m)$ has maximum age and splits).
            \item 
            If $G_m\neq \emptyset$, choose $v\in G_m$ and duplicate it. The two new vertices will be order consecutive, will have the same unary relation and loop edges as $v$ did, and will have the same binary relations with vertices of $G_m\setminus\{v\}$ as $v$ did. If $v\in P_r$, the two new vertices will be red adjacent, and if $v\in P_b$, they will be blue adjacent (a node of non-maximal age splits). 
            \item 
            Delete a loop edge (age change).
            \item 
            If $u\neq v\in G_m$ are connected by a red edge and neither $u$ nor $v$ have a red loop, we may delete this edge. Likewise for blue (age change).
            \item 
            If $v\in G_m$ has no loops, is not red-and-blue related to any vertex, its red neighbors contain no red edges among them, and its blue neighbors contain no blue edges among them, we can delete $v$ (coding level). \qed
        \end{enumerate}
    \end{enumerate}
\end{exa}

\begin{defin}
    \label{Def:EmbDiary}
    Let $\Theta$ and $\Delta$ be diaries. A \emph{diary embedding} of $\Theta$ into $\Delta$ is any $\phi\in \aemb(\Theta, \Delta)$ which additionally satisfies:
    \begin{itemize}
        \item
        If $m\in \ac(\Theta)$, then $\wt{\phi}(m)\in \ac(\Delta)$. Defining $\phi^+\colon \Theta(m+1)\to \Delta(\wt{\phi}(m)+1)$ via $\phi^+(s^\frown 0) = \phi(s)^\frown 0$, we have $\age_\Theta(m+1) = \age_\Delta(\phi^+[\Theta(m+1)])$. 
    \end{itemize}
    Write $\demb(\Theta, \Delta)$ for the set of diary embeddings of $\Theta$ into $\Delta$. \qed
\end{defin}

\begin{rem}
      Given $\phi\in \demb(\Theta, \Delta)$, we can define $\phi^+$ on every level. Given any $m< \rm{ht}(\Theta)-1$, define $\phi^+\colon \Theta(m+1)\to \Delta(\wt{\phi}(m)+1)$, where given $s\in \Theta(m)$ and $i< k$ with $s^\frown i\in \Theta(m+1)$, we have $\phi^+(s^\frown i) = \phi(s)^\frown i$. This map is well defined since $\phi\in \aemb(\Theta, \Delta)$. When $m\not\in\ac(\Delta)$, the properties of diaries already imply that $\age_\Theta(m+1) = \age_\Delta(\phi^+[\Theta(m+1)])$, and for $m\in \ac(\Delta)$, we explicitly demand this in the definition of diary embedding.
\end{rem}

Note that if $\phi\colon \Theta\to \Delta$ is a diary embedding, then $\phi|_{\str(\Theta)}\colon \str(\Theta)\to \str(\Delta)$ is an embedding of $\cL$-structures, and given $S\subseteq \Theta(m)$, $\phi|_{\str(\Theta)/S}\colon \str(\Theta)/S\to \str(\Delta)/\phi[S]$ is an embedding of $\cL_d$-structures. Conversely, we have the following.

\begin{prop}
    \label{Prop:Induced_Subdiary}
    Given any diary $\Delta$ and $X\subseteq \cdnd(\Delta)$, there are a unique diary $\Delta\|_X$ and $\phi_{\Delta, X}\in \demb(\Theta, \Delta)$ with $\phi[\cdnd(\Delta\|_X)] = X$. 
\end{prop}

Along with notation from the proposition statement, the map $\pi_{\Delta, X}\colon X\to \Delta\|_X$ defined in the proof (a partial inverse to $\phi_{\Delta, X}$) will also be used later.

\begin{proof}
    Write $|\crit^\Delta(X)| = N\leq \omega$ and $\crit^\Delta(X) = \{a_m: m< N\}\subseteq \omega$ with $a_{m-1}< a_m$ for each $1\leq m< N$. We define a map $\pi_{\Delta, X}\colon X\to \MP\times k^{<\omega}$ where given $x\in X$, we set $\pi_{\Delta, X}(x)^\sf{p} = x^\sf{p}$, $\ell(\pi_{\Delta, X}(x)) = |\ell(x)\cap \crit^\Delta(X)|$, and given $m< \ell(\pi_{\Delta, X}(x))$, we set $\pi_{\Delta, X}(x)(m) = x(a_m)$. One should think of $\pi_{\Delta, X}(x)$ as being formed from $x$ by deleting all of the levels not in $\crit^\Delta(X)$. Note that $\pi_{\Delta, X}$ preserves both $\leq_\ell$ and $\lex$ and that $\im(\pi_{\Delta, X})$ is a $\sqsubseteq$-antichain.  

    We set $\Delta\|_X = \im(\pi_{\Delta, X}){\downarrow}$. Given $m< \rm{ht}(\Delta\|_X)$, we define $\age_\Delta\|_X(m)$ as follows. Pick $\{x_i: i< d\}\in X$ such that $\ell(\pi_{\Delta, X}(x_i))\geq m$ for each $i< d$ and with $\Delta\|_X(m) = \{\pi_{\Delta, X}(x_0)|_m\lex\cdots\lex \pi_{\Delta, X}(x_{d-1})|_m\}$. Then set $\age_{\Delta\|_X}(m) = \age_\Delta(\{x_i|_{a_m}: i< d\})$. That this turns $\Delta\|_X$ into a diary follows from Proposition~\ref{Prop:Consecutive}, Proposition~\ref{Prop:Split_restrict}, and Lemma~\ref{Lem:Controlled_Coding_Restriction}. 

    We define $\phi_{\Delta, X}\colon \Delta\|_X\to \Delta$ as follows. Consider $y\in \Delta\|_X(m)$, and suppose $x\in X$ satisfies $y\sqsubseteq \pi_{\Delta, X}(x)$. We set $\phi_{\Delta, X}(y) = x|_{a_m}$. It is straightforward, if a bit tedious, to check that $\phi_{\Delta, X}\colon \Delta\|_X\to \Delta$ is a diary embedding. 
\end{proof}

Our next goal, Proposition~\ref{Prop:Exists_Diary}, is to construct a diary coding $\bK$.

\begin{defin}
	\label{Def:ValidPassingNumbers}
    Given a path sort $\rho\colon 2\to \MP$, $\cA\in P(\rho)$, $j< 2$, and $q< k$, we say that $q$ is a \emph{valid passing number} for $(\rho, \cA, j)$ if $\cA{\cdot}\rm{Add}_{j, \iota_q}\in \rho(1-j)$. In particular, $0$ is always a valid passing number for $(\rho, \cA, j)$.

    If $\Delta\subseteq \MP\times k^{<\omega}$ is an aged coding tree, $m< \rm{ht}(\Delta)$,  and $s\neq t\in \Delta(m)$, we say that $q< k$ is a \emph{valid passing number} for $(\Delta, s, t)$ if, writing $\{s, t\} = \{s_0, s_1\}$ with $s_0\lex s_1$ and letting $j< 2$ satisfy $t = s_j$, we have $q$ a valid passing number for $(\sort(\{s, t\}), \age_\Delta(\{s, t\}), j)$. \qed
\end{defin}

\begin{lemma}
    \label{Lem:Valid_Passing}
    Suppose $\rho$, $\cA$, and $j< 2$ are as in Definition~\ref{Def:ValidPassingNumbers} and that $q< k$ is a valid passing number for $(\rho, \cA, j)$. Suppose $\rho(1-j) = \frak{p}\in \MP_{non}$ and $\cA^{1-j} = \max(\frak{p})$. Then $q$ is a valid passing number for $(\rho, \cA\cap \max(\frak{p}'){\cdot}\iota_{1-j}^{-1}, j)$. 
\end{lemma}

\begin{proof}
    Without loss of generality suppose $j = 1$ and $q\neq 0$. Then $\cA{\cdot}\rm{Add}_{1, \iota_q}\in \frak{p}$, and in particular $\cA{\cdot}\rm{Add}_{1, \iota_q}\subseteq \cA^0 = \max(\frak{p})$. Since $\cA\in P(\rho)$ as witnessed by some gluing, we must have $\cA\subseteq \cK{\cdot}\wt{\rho}$. By the definition of $\cA{\cdot}\rm{Add}_{1, \iota_q}$ and since $\frak{p}\in \MP_{non}$, we must have $\cA{\cdot}\rm{Add}_{1, \iota_q}\subseteq \max(\frak{p}')$. Hence we have $(\cA\cap\max(\frak{p}'){\cdot}\iota_0^{-1}){\cdot}\rm{Add}_{1, \iota_q} = \cA{\cdot}\rm{Add}_{1, \iota_q}$.
\end{proof}

\begin{prop}
    \label{Prop:Exists_Diary}
    There is a diary $\Delta$ with $\str(\Delta) \cong \bK$.
\end{prop}

\begin{proof}
    We start by setting $\Delta(0) = \MP\times \{\emptyset\}$ and $\age_\Delta(0) = \cK{\cdot}\sort(\Delta(0))$. Fix $n< \omega$, and assume inductively that we have build $\Delta$ up to level 
    $$m:=\begin{cases}
        \ell^\Delta(n-1)+1 \quad &\text{if } n\geq 1,\\
        0 \quad &\text{if } n = 0.
    \end{cases}$$ 
    Fix $t\in \Delta(m)$. This is the node we will extend to become $c^\Delta(n)$. The choice of $t$ is arbitrary, so long as we eventually arrange that item $2$ of Definition~\ref{Def:Diary} holds. Our construction up to level $\ell^\Delta(n)+1$ now proceeds as follows.
    \begin{itemize}
        \item 
        $t$ is a splitting node. Write $u = t^\frown 1$.
        \item
        For every $s\in \Delta(m+1)\setminus \{u\}$, let $I_s$ denote the set of valid passing numbers for $(s, u)$. The next several levels are splitting levels satisfying item $4$ of Definition~\ref{Def:Diary}. For each $s\in \Delta(m+1)\setminus\{u\}$, $|I_s|-1$ of these splitting levels will have a splitting node extending $s$. Let $M> m$ denote the level we are at after these splitting levels, and write $\Delta(m) = \{s_0\lex\cdots\lex s_{d-1}\}$. Note that for $s\in \Delta(m+1)\setminus \{u\}$, we have $|\succ_\Delta(s, M)| = |I_s|$, and we have $\succ_\Delta(u, M):= \{s_j\}$ a singleton.  
        
        Write $\phi\colon (d-1)\to k$ for the function such that for each $s\in \Delta(m+1)\setminus\{u\}$, writing $X_s = \{i< d-1: s\sqsubseteq s_{(d{\setminus}\{j\})(i)}\}$, $\phi|_{X_s}$ is the $\lex$-increasing bijection onto $I_s$. 
        \item
        By Lemma~\ref{Lem:Valid_Passing}, $\age_\Delta(M){\cdot}\rm{Add}_{j, \phi}$ contains all unaries. Write $\age_\Delta(M) = \cA_0$, and find $\cA_0\supseteq \cdots\supseteq \cA_r = \la \cA_0, j, \phi\ra$ with $(\cA_i, \cA_{i+1})\in \con(\rho)$. Then the levels in $[M, M+r)$ are age-change levels with $\age_\Delta(M+i) = \cA_i$ for $i< r$.
        \item
        Now for every $i< d$,  is the unique element of $\succ_\Delta(s_i, M+r) = \{\Left(s_i, M+r)\}$. We set $\ell^\Delta(n) = M+r$ and $c^\Delta(n) = \Left(s_j, \ell^\Delta(n))$. We set $\Delta(\ell^\Delta(n)+1) = \{\Left(s_{(d{\setminus}\{j\})(i)}, \ell^\Delta(n))^\frown \phi((d{\setminus}\{j\})(i)): i< d-1\}$ and $\age_\Delta(\ell^\Delta(n)+1) = \cA_0{\cdot}\rm{Add}_{j, \phi} = \cA_r{\cdot}\rm{Add}_{j, \phi}$.
    \end{itemize}
    By our construction, given $s\in \Delta(m)$ and $q\in I_{s^\frown 0}$, there is a unique $u\in \succ_\Delta(s, \ell^\Delta(n)+1)$ with $u(\ell^\Delta(n)) = q$; we denote this $u$ by $s{*}q$. 

    We show that $\str(\Delta)\cong \bK$ by showing that it satisfies the extension property for one-element extensions. Fix $n < \omega$, and suppose $j< \sfU$ and that $\gamma_n:= (\str^\#(\Delta)|_n, \iota_j, \eta)$ is a rank $1$ gluing with $\cK{\cdot}\gamma_n\neq\emptyset$, thus representing an instance of an extension problem. Given $m< n$, set $\gamma_m = (\str^\#(\Delta)|_m, \rho, \eta|_{1\times m})$. Thus $\max(P_j) = \cK{\cdot}\gamma_0\supseteq\cdots\supseteq \cK{\cdot}\gamma_n \in P_j$, so pick some $\frak{p}\in \MP_j$ with $\cK{\cdot}\gamma_m\in \frak{p}$ for every $m\leq n$. 

    Set $s_0 = (\frak{p}, \emptyset)\in \Delta$. Suppose $m< n$ and that $s_m$ has been determined with $\age_\Delta(s_m) = \cK{\cdot}\gamma_m$ and such that if $m> 0$, we have $s_m\in \Delta(\ell(c^{m-1})+1)$. Write $t = c^\Delta(m)|_{\ell^\Delta(m-1)+2}$, and note that $t(\ell^\Delta(m-1)+1) = 1$; in particular $s_m^\frown 0\neq t$. Then $q_m:= \eta(1, m)< k$ is a valid passing number for $(s_m^\frown 0, t)$ since $\cK{\cdot}\gamma_{m+1}\in \frak{p}$. We then set $s_{m+1} = s_m{*}q_m$.

    Upon defining $s_n$, any $t\in \cdnd(\Delta)\cap \succ_\Delta(s_n)$ will witness the desired instance of the extension property.   
\end{proof}

A major theorem of this paper is that any two diaries which code $\bK$ are bi-embeddable.

\begin{theorem}
	\label{Thm:EmbDiary}
	Let $\Theta$ be any diary, and let $\Delta$ be any diary with $\str(\Delta) \cong \bK$. Then $\demb(\Theta, \Delta)\neq \emptyset$. In particular, any two diaries coding $\bK$ are bi-embeddable.
\end{theorem}

We postpone the proof of Theorem~\ref{Thm:EmbDiary} until Section~\ref{Section:Proof_Of_Thm_EmbDiary}. Now is a good time to mention various aspects of why our definition of diary contains all of the features it has. Namely, much of what motivates Definition~\ref{Def:Diary} is that we want Theorem~\ref{Thm:EmbDiary} to be true. One could imagine, for instance, trying to relax item $5$ of the definition to only demand that $\age_\Delta(m+1) = \age_\Delta(m){\cdot}\rm{Add}_{j, \phi}$. But if we did this and $\Delta$ was a diary coding $\bK$, we could find $X\subseteq \cdnd(\Delta)$ with $\str(\Delta\|_X)\cong \bK$ and so that all coding levels of $\Delta\|_X$ had controlled coding triples. Similarly, one could imagine trying to simplify the definition of $\sp(\cA, i)$ to always be as in case $1$ of Definition~\ref{Def:SplittingClass}. But again, if we did this and $\Delta$ was a diary coding $\bK$, we could find $X\subseteq \cdnd(\Delta)$ with $\str(\Delta\|_X)\cong \bK$ and so that for all splitting levels $m$ corresponding to case $2$ of Definition~\ref{Def:SplittingClass}, then level $m+1$ would be an age-change level such that $\age_\Delta(m+2) = \sp(\age_\Delta(m), i)$ as we currently have it defined.

We end the section by showing how Theorem~\ref{Thm:EmbDiary} yields lower bounds for big Ramsey degrees.

\begin{defin}
    \label{Def:Diary_Shape}
Given $\bA\in \cK$, we set 
$$D_\bA:= \{(\Theta, g): \Theta \text{ a diary and $g\colon \bA\to \str(\Theta)$ an isomorphism}\}.$$
Note that by item $2$ of Definition~\ref{Def:Diary} and Proposition~\ref{Prop:Finite_Posets}, $D_\bA$ is finite. 

If $\Delta$ is a diary and $f\in \emb(\bA, \str(\Delta))$, the \emph{$\Delta$-shape of $f$} is defined via
$$\shp_\Delta(f) := (\Delta\|_{\im(f)}, \pi_{\Delta,  \im(f)}\circ f)\in D_\bA.$$
Viewing $\shp_\Delta$ as a class function, we write $\shp_{\Delta, \bA}\colon \emb(\bA, \str(\Delta))\to D_\bA$ for the restriction of $\shp_\Delta$ to $\emb(\bA, \str(\Delta))$. \qed
\end{defin}

\begin{theorem}
    \label{Thm:Lower_Bound}
    Let $\bA\in \cK$. Then $\rm{BRD}(\bA, \cK)\geq |D_\bA|$
\end{theorem}

\begin{proof}
    Fix a diary $\Delta$ with $\bK \cong  \str(\Delta)$. Towards showing that $\shp_{\Delta, \bA}$ is unavoidable, fix $(\Theta, g)\in D_\bA$, and let $\eta\in \emb(\bK, \str(\Delta))$. Write $\Delta' = \Delta\|_{\im(\eta)}$ and $\phi = \phi_{\Delta, \im(\eta)}$. Then $\str(\Delta')\cong \bK$, so in particular by Theorem~\ref{Thm:EmbDiary}, there is some $\sigma\in \demb(\Theta, \Delta')$. Then  $\phi\circ \sigma\circ g\in \emb(\bA, \str(\Delta))$ satisfies $\im(\phi\circ \sigma\circ g)\subseteq \im(\eta)$ and $\shp_{\Delta, \bA}(\phi\circ \sigma\circ g) = (\Theta, g)$.
\end{proof}

Lastly, we show how diaries can be encoded via expansions of $\bK$ in a finite relational language.

\begin{defin}
    \label{Def:Diary_Exp}
    We define a relational language $\cL^*\sqsupseteq \cL$ as follows. Let $\bB_0,..., \bB_{N-1}\in \cK$ list each enumerated structure in $\cK$ with size at most $r:= \max(2(\|\cF\|-1), 4)$. For each $M< N$ and $y \in D_{\bB_M}$, introduce a new $|B_M|$-ary relation $R_y$ into $\cL^*$. 
    
    Given a diary $\Delta$, We define the $\cL^*$-expansion $\str^*(\Delta)$ of $\str(\Delta)$ as follows. Fix $M< N$, and write $|B_M| = d$. Now given $x_0,..., x_{d-1}\in \cdnd(\Delta)$, we set $R_y^{\str^*(\Delta)}(x_0,..., x_{d-1})$ iff the $x_i$ are distinct and, writing $X = \{x_0,..., x_{d-1}\}$ and $y = (\Theta, g)$, we have $\Delta\|_X = \Theta$ and $g(i) = \pi_{\Delta, X}(x_i)$ for each $i< d$.  \qed
\end{defin}

\begin{theorem}
    \label{Thm:Expansion_Codes_Diary}
    There are a finite relational language $\cL^*\supseteq \cL$ and an $\cL^*$-expansion $\bK^*$ of $\bK$ such that given any $\bA\in \cK$, $|\bK^*(\bA)| = |D_\bA|$ and the map from $\emb(\bA, \bK)\to \bK^*(\bA)$ sending $f$ to $\bK^*{\cdot}f$ witnesses that $\rm{BRD}(\bA, \cK)\geq |D_\bA|$. 
\end{theorem}

\begin{proof}
    Fix a diary $\Delta$ with $\bK \cong \str(\Delta)$, and set $\bS^*:= \str^*(\Delta)$. 
    Now fix $\bA\in \cK$; we show that if $f_0, f_1\in \emb(\bA, \str(\Delta))$, then $\bS^*{\cdot}g_0 = \bS^*{\cdot}g_1$ iff $\shp_{\Delta, \bA}(f_0) = \shp_{\Delta, \bA}(f_1)$. Certainly the right-to-left implication holds. For the left-to-right, this amounts to saying that a finite diary $\Theta$ (in fact any diary) is completely determined by the subdiaries $\Theta|_X$ for $X\subseteq \cdnd(\Theta)$ of size at most $r$. Since $r\geq 4$, we accurately recover the relative levels of all coding and splitting events. Since $r\geq (\|\cF\|-1)+2$, we accurately recover the relative levels between coding/splitting and age-change events. And since $r\geq 2(\|\cF\|-1)$, we accurately record the relative levels between any pair of age-change events. Taken together, this suffices to completely reconstruct $\Theta$.
\end{proof}

\section{Proof of Theorem~\ref{Thm:EmbDiary}}
\label{Section:Proof_Of_Thm_EmbDiary}

Our proof of Theorem~\ref{Thm:EmbDiary} is adapted from the characterization of the exact big Ramsey degrees for the class of finite graphs given by Laflamme, Sauer, and Vuksanovic \cite{LSV}.  In their proof, they more-or-less treat all nodes as coding nodes, and while the resulting graph is not the Rado graph, it is bi-embeddable with it. Upon defining a suitable notion of ``diary," the authors then fix an embedding $\psi$ from the universal graph coded by $2^{<\omega}$ and the diary. This leads to an important collection of pairs of nodes, namely those pairs $(u, v)$ where $\psi^{-1}(\succ(v))$ is dense above $u$. 

Our strategy is similar, but we require several notions of largeness which are more sophisticated than density above $u$, as we will need to keep track of the relevant age-set structure.

\textbf{For this section}, we let $\Theta$ be any diary, $\Delta$ a diary with $\str(\Delta)\cong \bK$, and without loss of generality, we take $\bK = \str^\#(\Delta)$.

\subsection{Large subsets of coding nodes}

We begin by developing the notions of largeness we will need. This subsection only refers to subsets of $\ct^\bK$; the diary $\Delta$ won't feature until the next subsection.

\begin{defin}
	Suppose that $S\subseteq \im(c^\bK)$. Fix $i< \sfU$, $\frak{p}\in \sf{Full}_i$, and $u\in \ct^\bK_i$ with $\age_\bK(u) = \max(\frak{p})$. We define the property that $S$ is \index{$(u, \frak{p})$-large set of coding nodes}\emph{$(u, \frak{p})$-large} by induction on $|\frak{p}|$.
	\begin{itemize}
		\item 
		If $|\frak{p}| = 1$, then $S$ is $(u, \frak{p})$-large if it is dense in $\ct^\bK$ over $u$.
		\item
		Otherwise, we say that $S$ is $(u, \frak{p})$-large if for every $v\sqsupseteq u$ with $\age_\bK(v) = \age_\bK(u) = \max(\frak{p})$, there is $w\sqsupseteq v$ with $\age_\bK(w) = \max(\frak{p}')$ so that $S$ is $(w, \frak{p}')$-large. 
	\end{itemize}
\end{defin}

As a convention, if we refer to anything being $(u, \frak{p})$-large for some $\frak{p}\in \sf{Full}_i$, it is assumed that $u\in \ct^\bK_i$ and $\age_\bK(u) = \max(\frak{p})$.

\begin{lemma}
\label{Lem:LargeAllPaths}
	Suppose $i< \sfU$, $u\in \ct^\bK_i$ and $S\subseteq \im(c^\bK)$ is dense in $\ct^\bK$ over $u$. Then $S$ is $(u, \frak{p})$-large for any $\frak{p}\in \full_i$ with $\max(\frak{p}) = \age_\bK(u)$.
\end{lemma}

\begin{proof}
	We induct on $|\frak{p}|$. When $|\frak{p}| = 1$, then we have $\frak{p} = \{\min(P_i)\}$, and this is just the definition of $(u, \frak{p})$-large. Now suppose $|\frak{p}| \geq 2$. Consider $v\sqsupseteq u$ with $\age_\bK(v) = \age_\bK(u) = \max(\frak{p})$. Find $w\sqsupseteq v$ with $\age_\bK(w) = \max(\frak{p}')$. Since $S$ is dense in $\ct^\bK$ over $w$, our inductive assumption implies that $S$ is $(w, \frak{p}')$-large. Hence $S$ is $(u, \frak{p})$-large as desired. 
\end{proof}

\begin{lemma}
	\label{Lem:LargePartitionRegular}
	Suppose $i< \sfU$, $\frak{p}\in \sf{Full}_i$, $u\in \ct^\bK_i$, and that $S\subseteq \im(c^\bK)$ is $(u, \frak{p})$-large. If $n< \omega$ and we write $S = \bigcup_{j< n} S_j$, then for some $j< n$ and some $v\sqsupseteq u$ with $\age_\bK(v) = \age_\bK(u) = \max(\frak{p})$, we have that $S_j$ is $(v, \frak{p})$-large.
\end{lemma}

\begin{proof}
	First assume that $\frak{p} = \{\min(P_i)\}$. Towards a contradiction, suppose we could find $u\in \ct^\bK$ with $\age_\bK(u) = \min(P_i)$, $S\subseteq \im(c^\bK)$ dense in $\ct^\bK$ over $u$, and a partition $S = S_0\cupdots S_{n-1}$ so that the conclusion of the lemma fails. Set $u = u_0$. Suppose $u_j\sqsupseteq\cdots \sqsupseteq u_0$ have been defined with $\age_\bK(u_j) = \min(P_i)$. As $S_j$ is not dense in $\ct^\bK$ over $u_j$, find $u_{j+1}\sqsupseteq u_j$ with $\age_\bK(u_{j+1}) = \min(P_i)$ and so that $S_j\cap \succ(u_{j+1}) = \emptyset$. Continue until $u_n$ has been defined. Then $S$ is dense in $\ct^\bK$ over $u_n$, but also $S_j\cap \succ(u_n) = \emptyset$ for every $j< n$. This is a contradiction.
	
	Now we proceed by induction on $|\frak{p}|$. Set $u = u_0$. Suppose $u_j\sqsupseteq\cdots \sqsupseteq u_0$ have been defined for some $j< n$ with $\age_\bK(u_j) = \max(\frak{p})$. Since $S_j$ is not $(u_j, \frak{p})$-large, we can find $u_{j+1}\sqsupseteq u_j$ with $\age_\bK(u_{j+1}) = \max(\frak{p})$ so that for any $v\sqsupseteq u_{j+1}$ with $\age_\bK(v) = \max(\frak{p}')$, we have that $S_j$ is not $(v, \frak{p}')$-large. However, since $S$ is $(u, \Gamma)$-large, we can find $v\sqsupseteq u_n$ with $\age_\bK(v) = \max(\frak{p}')$ and so that $S$ is $(v, \frak{p}')$-large. By our induction assumption, we can find $j< n$ and $w\sqsupseteq v$ with $\age_\bK(w) = \max(\frak{p}')$ and so that $S_j$ is $(w, \frak{p}')$-large. But since $w\sqsupseteq u_{j+1}$, this is a contradiction. 
\end{proof}

The next proposition says that large sets of coding nodes code suitably rich substructures of $\b{K}$. The proof is straightforward, but a bit long, and can be skipped on a first reading.

\begin{prop}
	\label{Prop:InducedStructureLargeSets}
	Let $m< \omega$, let $X = \{x_0,\dots,x_{d-1}\}\subseteq \ct^\bK(m)$, and write $\rho = \sort(X)$. For each $i< d$, let $\frak{p}_i\in \sf{Full}_{\rho(i)}$ and $S_i\subseteq \im(c^\bK)\cap \succ(x_i)$ be such that $S_i$ is $(x_i, \frak{p}_i)$-large for each $i< d$. Setting $S = \bigcup_{i<d} S_i$ and viewing $(c^\bK)^{-1}(S)$ as an $\c{L}_d$-structure in the natural way, we have $\age((c^\bK)^{-1}(S)) = \age_\bK(X)$. 
\end{prop}

\begin{proof}
	If $|\frak{p}_i| = 1$ for each $i< d$, the assumption on the $S_i$ just says that $S$ is dense in $\ct^\bK$ above each $x_i$. For each $\b{B}\in \age_\bK(X)$, we show that $\b{B}\in \age((c^\bK)^{-1}(S))$ by induction on the size of $\b{B}$ (for all $d$ simultaneously). If $\b{B}$ contains a single point $b$ with $V^\b{B}(b) = i$, the result is clear. Now suppose $\b{B}$ is an $\c{L}_d$-structureand that $|B|\geq 2$. Pick some $b\in B$, and write $V^\b{B}(b) = i$. Let $w_i\sqsupseteq x_i$ be any member of $S_i$, and for each $j\in d\setminus \{i\}$, we set $w_j = \Left(x_j, \ell(w_i))$. Write $W = \{w_0\lex\cdots\lex w_{d-1}\}$, and write $Y = \is_\bK(W) := \{y_0\lex\cdots\lex y_{\alpha-1}\}$. Note that since $\age_\bK(x_i) = \max(\frak{p}_i) = \min(P_{\rho(i)})$ (since $|\frak{p}_i| = 1$), we have $\age_\bK(X) = \age_\bK(W)$. Let $\b{C}$ be the $\c{L}_\alpha$-structure on underlying set $B\setminus \{b\}$ defined as follows. The $\c{L}^\sf{b}$-part of $\b{C}$ is induced from $\b{B}$. For the unary part, fix $a\in B\setminus \{b\}$, and suppose that $V^\b{B}(a) = j$ and  $R^\b{B}(a, b) = q$.  Then $w_j^\frown q\in \ct^\bK$. Now if $\beta< \alpha$ is such that $y_\beta = w_j^\frown q$, we set $V^\b{C}(a) = \beta$. Since $S$ is dense above each member of $Y$, our inductive hypothesis shows that viewing $(c^\bK)^{-1}(S\cap \succ(Y))$ as an $\c{L}_\alpha$-structure, we have $\b{C}\in \age((c^\bK)^{-1}(S\cap \succ(Y)))$. It follows that $\b{B}\in \age((c^\bK)^{-1}(S))$. 
	
	Now suppose we are given $\{\frak{p}_j: j< d\}$, where $|\frak{p}_j|\leq n$ for each $j< d$. Suppose we know the result (for all $d$) whenever the paths have size less than $n$. Write $I = \{i< d: |\frak{p}_i| = n\}$. Let $\b{B}\in \age_\bK(X)$, and write $\b{B}_i$ for the induced substructure on $\{b\in B: V^\b{B}(b) = i\}$. Fix some $N \geq \max\{|\b{B}_i|: i\in I\}$, and fix some $i\in I$. Since $S_i$ is $(x_i, \frak{p}_i)$-large, there is $x'_{i, 0}\sqsupseteq x_i$ so that $S_i$ is $(x'_{i, 0}, \frak{p}_i')$-large. If $x'_{i, r}$ has been defined, then since $S_i$ is $(x_i, \frak{p}_i)$-large and $\age_\bK(\Left(x_i, \ell(x'_{i, r}))) = \age_\bK(x_i) = \max(\frak{p}_i)$, there is $x'_{i, r+1}\sqsupseteq \Left(x_i, \ell(x'_{i, r}))$ so that $S_i$ is $(x'_{i, r+1}, \frak{p}_i')$-large. Continue until $x'_{i, N-1}$ is defined, and write $X_i = \{x'_{i, N-1}\lex\cdots\lex x'_{i, 0}\}$. 
	
	Do the same procedure for every $i\in I$, at each stage moving relevant points up leftmost far enough to be above anything previously considered, producing a set $X_i$ as above. If $i\in d\setminus I$, set $X_i = \{x_i\}$. Pick some level $L$ above everything we have considered, and set $Y = \bigcup_{i< d} \Left(X_i, L) = \{y_0\lex\cdots\lex y_{\alpha-1}\}$. Let $\pi\colon \alpha\to d$ be the map so that $y_\beta\sqsupseteq x_{\pi(\beta)}$ for each $\beta< \alpha$. Then by the construction of $Y$, we have 
	$$\age_\bK(Y) = \age_\bK(X){\cdot}\pi\cap \bigcap_{i\in I} \bigcap_{\beta\in \pi^{-1}(\{i\})} \max(\frak{p}_i'){\cdot}\iota_\beta^{-1}.$$
    Using $\b{B}$, we form an $\c{L}_\alpha$-structure $\b{C}$ on underlying set $B$ as follows. The $\c{L}^\sf{b}$-part is the same as $\b{B}$. For each $i\in d\setminus I$, let $\beta_i< \alpha$ be the unique member of $\pi^{-1}(\{i\})$, and set $V^\b{C}(b) = \beta_i$ for each $b\in B_i$. If $i\in I$, then for each $b\in B_i$, choose a distinct $\beta_b\in \pi^{-1}(\{i\})$, and set $V^\b{C}(b) = \beta_b$. Now $\b{C}\in \age_\bK(Y)$, so by our inductive hypothesis, $\b{C}\in \age((c^\bK)^{-1}(S\cap \succ(Y)))$. So also $\b{B}\in \age((c^\bK)^{-1}(S))$ as desired. 
\end{proof}

\subsection{Compatible pairs}
\label{Subsec:Pairs}

Our proof of Theorem~\ref{Thm:EmbDiary} will proceed by building $\phi\in \demb(\Theta, \Delta)$ by induction on the levels of $\Theta$. To get started, we must have $\phi(\frak{p}, \emptyset) \in \Delta_\frak{p}$ whenever $\Theta_\frak{p}\neq \emptyset$. However, this implies $\Delta_\frak{p}\neq \emptyset$; we need to prove this just from knowing that $\Delta$ codes $\bK$. This is Proposition~\ref{Prop:BRD_Lower_Bound_Singletons}, which provides a lower bound for the big Ramsey degrees of singleton structures and sets the stage for the inductive construction of $\phi$ in the next subsection. 

\textbf{For the rest of the section}, write $\psi\colon \im(c^\bK)\to \Delta$ for the map with $\psi(c^\bK(n)) = c^\Delta(n)$.

\begin{defin}
	\label{Def:CompatiblePairs}
	Suppose $\frak{p}\in \full_i$. We say that a pair $(u, v)\in \ct^\bK\times \Delta$ is \index{$\frak{p}$-compatible pair}\emph{$\frak{p}$-compatible} if $\age_\Delta(v) = \max(\frak{p})$ and $\psi^{-1}(\succ(v))$ is $(u, \frak{p})$-large. In particular, note that $\age_\bK(u) = \age_\Delta(v) = \max(\frak{p})$. Write $\pair(\frak{p})$ for the set of $\frak{p}$-compatible pairs.

    Given $v\in \Delta$, the \emph{path continuation} of $v$, denoted $\pcon(v)$, is the unique $\frak{p}\in \full$ so that $\age_\Delta(v) = \max(\frak{p})$ and  $v^\sf{p} = \argpath_\Delta(v)\cup \frak{p}$. Corollary~\ref{Cor:PairImpliesPath} will tell us that if $(u, v)\in \pair(\frak{p})$, then $\frak{p} = \pcon(v)$. We say that $(u, v)\in \ct^\bK\times \Delta$ is a \emph{compatible pair}  and write $(u, v)\in \pair$ if $(u, v)\in \pair(\pcon(v))$. 

    Given $d< \omega$, we say that functions $\sigma\colon d\to \ct^\bK$ and $\xi\colon d\to \Delta$ are \emph{compatible} if $(\sigma(i), \xi(i))$ is compatible for every $i< d$. In a mild abuse of notation, we simply write $(\sigma, \xi)\in \pair$ when this happens.  As a warning, we do not necessarily have $\age_\bK(\sigma) = \age_\Delta(\xi)$; indeed, if $\xi$ isn't level, then $\age_\Delta(\xi)$ isn't even defined. If $\xi$ is level, Propositions~\ref{Prop:InducedStructureLargeSets} and \ref{Prop:Diary_Correct_Age}(1) give us that $\age_\Delta(\xi) \supseteq \age_\bK(\sigma)$. \qed
\end{defin}

Notice that if $(u,v)\in \pair(\frak{p})$ and $u_0\sqsupseteq u$ satisfies $\age_\bK(u_0) = \age_\bK(u)$, then also $(u_0, v)\in \pair(\frak{p})$. 

We collect some straightforward ``pair-extension" properties and their corollaries. Before proving these, let us rephrase Proposition~\ref{Prop:Diary_Correct_Age} in a way that is more suitable for our setting.

\begin{lemma}
	\label{Lem:MostImageAge}
	Suppose $\sigma\colon d\to \ct^\bK$ is a function, and suppose $N > \max\{\ell^\Delta(n): n< \max\{\ell(\sigma(i)): i< d\}\}$. For each $i< d$, suppose $s_i\in \im(c^\bK)\cap \succ(\sigma(i))$ satisfies  $\ell(\psi(s_i)) \geq N$. Then letting $\xi\colon d\to \Delta$ be given by $\xi(i) = \pi_N\circ \psi(s_i)$, we have $\age_\Delta(\xi)\subseteq \age_\bK(\sigma)$. 
\end{lemma} 

\begin{proof}
	We have for each $n< \max\{\ell(\sigma(j)): j< d\}$ that $\psi(s_j)(\ell^\Delta(n)) = s_j(n)$. The result now follows from Proposition~\ref{Prop:Diary_Correct_Age}.
\end{proof}

\begin{prop}
	\label{Prop:LiftPairFixedLevel}
	Suppose $i< \sfU$, $\frak{p}\in \full_i$, and $(u, v)\in \pair(\frak{p})$.
	\begin{enumerate}
	\item 
	If $N> \ell(v)$, then there are $u_0\sqsupseteq u$ and $v_0\sqsupseteq v$ with $\ell(v_0) = N$ and $(u_0, v_0)\in \pair(\frak{p})$. 
	\item
	Suppose $|\frak{p}|\geq 2$. Then there is $(u', v')\in \pair(\frak{p}')$ with $u'\sqsupseteq u$ and $v'\sqsupseteq v$.
	\end{enumerate}
\end{prop}

\begin{proof}
	$(1)$ Write $\succ(v)\cap \Delta(N) = \{t_0,\dots,t_{d-1}\}$. There is a finite set $F$ for which we have
	$$\psi^{-1}(\succ(v))\setminus F= \bigcup_{j< d} \psi^{-1}(\succ(t_j)).$$ 
	By Lemma~\ref{Lem:LargePartitionRegular}, there is $u_0\sqsupseteq u$ and $j< d$ so that $\psi^{-1}(\succ(t_j))$ is $(u_0, \frak{p})$-large. It follows by Proposition~\ref{Prop:InducedStructureLargeSets} that $t_j := v_0$ must have $\age_\Delta(v_0) = \max(\frak{p})$, i.e.,\ $(u_0, v_0)\in \pair(\frak{p})$.

	$(2)$ Because $\psi^{-1}(\succ(v))$ is $(u, \frak{p})$-large, find $u_0\sqsupseteq u$ with $\age_\Delta(u_0) = \max(\frak{p}')$ and so that $\psi^{-1}(\succ(v))$ is $(u_0, \frak{p}')$-large. Using Lemma~\ref{Lem:MostImageAge}, find $N$ so that whenever $s\in \psi^{-1}(\succ(v))\cap \succ(u_0)$ and $\ell(\psi(s))> N$, we have $\age_\Delta(\pi_N\circ \psi(s))\subseteq \max(\frak{p}')$. Let us write
	$$S = \{s\in \psi^{-1}(\succ(v))\cap \succ(u): \ell(\psi(s))> N\}.$$
	Then $S$ is a cofinite subset of $\psi^{-1}(\succ(v))\cap \succ(u)$, a $(u_0, \frak{p}')$-large set; hence $S$ is also $(u_0, \frak{p}')$-large. Writing $\pi_N\circ \psi[S] = \{t_0,\dots,t_{d-1}\}$, let $S_j = \{s\in S: \pi_N\circ \psi(s) = t_j\}$. We can find $u'\sqsupseteq u_0$ so that some $S_j$ is $(u', \frak{p}')$-large. Set $v' = t_j$. It remains to check that $\age_\Delta(v') = \max(\frak{p}')$. By choice of $v'$ we must have $\age_\Delta(v')\subseteq \max(\frak{p}')$. The reverse inclusion follows from Propositions~\ref{Prop:InducedStructureLargeSets} and \ref{Prop:Diary_Correct_Age}.
\end{proof}

\begin{cor}
    \label{Cor:PairImpliesPath}
    Suppose $i< \sfU$, $\frak{p}\in \full_i$, and $(u, v)\in \pair(\frak{p})$. Then  $\frak{p} = \pcon(v)$.
\end{cor}

\begin{proof}
    First note that $\frak{p}\cap \argpath_\Delta(v) = \{\max(\frak{p})\}$. Our proof proceeds by induction on $|\frak{p}|$. When $|\frak{p}| = 1$, then $\argpath_\Delta(v) = v^\sf{p}\in \MP_i$, and the result is clear. Now suppose $|\frak{p}|\geq 2$. Use Proposition~\ref{Prop:LiftPairFixedLevel} to find $u'\sqsupseteq u$ and $v'\sqsupseteq v$ with $(u', v')\in \pair(\frak{p}')$. By our inductive assumption, $\frak{p}'\cup \argpath_\Delta(v') = v^\sf{p}$. But also $\frak{p}'\cup \argpath_\Delta(v') =  \frak{p}\cup \argpath_\Delta(v)$. 
\end{proof}

\begin{prop}
    \label{Prop:BRD_Lower_Bound_Singletons}
    $\Delta(0) = \MP\times \{\emptyset\}$. Furthermore, if $\xi\colon |\MP|\to \Delta(0)$ lists $\Delta(0)$ in $\lex$-order, then there is $\sigma\colon |\MP|\to \ct^\bK$ with $\age_\bK(\sigma) = \age_\Delta(\xi) = \age_\Delta(0)$  and $(\sigma, \xi)\in \pair$.  
\end{prop}

\begin{proof}
    Write $\MP = \{\frak{p}_0\leq_{\MP}\cdots \leq_{\MP} \frak{p}_{d-1}\}$, set $n_0 = 0$, and let $\sigma_0\colon d\to \ct^\bK$ be given by $\sigma_0(i) = (\sf{u}(\frak{p}_i), \emptyset)$. 
    
    Assume for some $i< d$ that $0 = n_0<\cdots < n_i$ and $\sigma_i\colon d\to \ct^\bK$ have been determined, and  $\ell(\sigma_i(j)) = n_{j+1}$ for $j< i$ and $\sigma_i(j) = \sigma_0(j)$ for $j\geq i$. Write $h = \sf{u}(\frak{p}_i)$, and first observe that $\im(c^\bK)\cap (\{h\}\times k^{<\omega})$ is $((h, 0^{n_i}), \frak{p})$-large. It follows from Lemma~\ref{Lem:LargePartitionRegular} that for some $u\sqsupseteq (h, 0^{n_i})$ with $\age_\bK(u) = \max(P_h)$ and some $\frak{q}\in \MP_h$, we have $(u, (\frak{q}, \emptyset))\in \pair(\frak{p}_i)$. By Fact~\ref{Fact:Basic_Diary_Facts}(2) and Corollary~\ref{Cor:PairImpliesPath}, we must have $\frak{q} = \frak{p}_i$. In particular, $(\frak{p}_i, \emptyset)\in \Delta(0)$. We set $n_{i+1} = \ell(u)$, $\sigma_{i+1}(i) = u$, and $\sigma_{i+1}(j) = \sigma_i(j)$ for $j\in d\setminus \{i\}$. 
 
    Set $\sigma = \sigma_d$; since at each stage, we moved up and leftmost in $\ct^\bK$ above all relevant nodes, we have $\age_\bK(\sigma) = \age_\bK(\sigma_0) = \age_\Delta(0)$. 
\end{proof}

\begin{rem}
    If $\bA\in \cK$ is a singleton with unary $i$, then Proposition~\ref{Prop:BRD_Lower_Bound_Singletons} recovers the result of Sauer \cite{Sauer03} that $\rm{BRD}(\bA, \cK) \geq |\MP_i|$.
\end{rem}

We end the subsection with the following strengthening of Proposition~\ref{Prop:LiftPairFixedLevel}, which allows us to run the proposition on pairs of functions while preserving the ages of said functions.

\begin{prop}
	\label{Prop:MovingPairsUpAgePreserve}
	Suppose $\sigma\colon d\to \ct^\bK$ and $\xi\colon d\to \Delta$ are functions with $(\sigma, \xi)\in \pair$. Let $N< \omega$ satisfy $N\geq \max(\{\ell(\xi(i)): i< d\}$ and $N > \{\ell^\Delta(n): n< \max\{\ell(\sigma(i)): i< d\}\}$.  
	Then there are functions $\sigma'\sqsupseteq \sigma$ and $\xi'\sqsupseteq \xi$ so that $\ell(\xi'(i)) = N$ for each $i< d$, $(\sigma', \xi')\in \pair$, and $\age_\bK(\sigma) = \age_\bK(\sigma') = \age_\Delta(\xi')$. 
\end{prop}

\begin{rem}
    Suppose $i< d$ and $\frak{p} = \pcon(\xi(i))$, so that $(\sigma(i), \xi(i))\in \frak{p}$. Then since $\xi'(i)\sqsupseteq \xi(i)$ and $\age_\bK(\sigma) = \age_\Delta(\xi')$, we will also have $\pcon(\xi'(i)) = \frak{p}$.
\end{rem}

\begin{proof}
	Suppose $i< d$ and that we have chosen $\sigma'(j)$ and $\xi'(j)$ for every $j< i$. If $\xi(i) = (h, 0^r)$ for some $r< \omega$ and $h\in \sfU_{non}$, we simply set $\xi'(i) = (h, 0^N)$ and $\sigma'(i) = \sigma(i)$. If $\xi(i)$ is not of this form, write $\frak{p} = \pcon(\xi(i))$, pick some $M> \max\left(\{\ell(\sigma'(j)): j< i\}\cup\{\ell(\sigma_j: j< d\}\right)$ and run Proposition~\ref{Prop:LiftPairFixedLevel} on $(\Left(\sigma(i), M), \xi(i))\in \pair(\frak{p})$ to obtain $\sigma'(i)\sqsupseteq \Left(\sigma(i), M)$ and $\xi'(i)\sqsupseteq \xi(i)$ with $\ell(\xi'(i)) = N$ and $(\sigma'(i), \xi'(i))\in \pair(\frak{p})$. Then $\age_\bK(\sigma') = \age_\bK(\sigma)$ by construction, the key here being that at each stage, we moved up and left beyond anything built so far before applying Proposition~\ref{Prop:LiftPairFixedLevel}. As $(\sigma', \xi')\in \pair$, we have  $\age_\Delta(\xi') \supseteq \age_\bK(\sigma')$. If $N$ is suitably large, Lemma~\ref{Lem:MostImageAge} implies $\age_\bK(\xi') = \age_\bK(\sigma)$.
\end{proof}

\subsection{The construction}
\label{Subsec:Construction}

We now prove Theorem~\ref{Thm:EmbDiary}, constructing $\phi\in \demb(\Theta, \Delta)$ level by level. We start by setting $\phi^+(\frak{p}, \emptyset) = (\frak{p}, \emptyset)$ for each $(\frak{p}, \emptyset)\in \Theta(0)$. Now fix $m< \rm{ht}(\Theta)$, and inductively assume we have defined $\phi\colon \Theta({<}m)\to \Delta$ and $\phi^+\colon \Theta({\leq}m)\to \Delta$. Write $\Theta(m) = \{t_0\lex\cdots\lex t_{d-1}\}$ and $\phi^+[\Theta(m)] = \{v_0\lex\cdots\lex v_{d-1}\}$. We assume that our construction to this point satisfies all of the following:
\begin{enumerate}
    \item 
    If $m\geq 1$, the map $\phi^+\colon \Theta(m)\to \Delta$ satisfies $\phi^+(v^\frown i) = \phi(v)^\frown i$ for every $v\in \Theta(m-1)$ and $i< k$ with $v^\frown i\in \Theta$, and furthermore, $\age_\Theta(m) = \age_\Delta(\phi^+[\Theta(m)])$.
    \item
    There are $\sigma_m\colon d\to \ct^\bK$ and $\xi_m\colon d\to \Delta$ with $\xi_m(i)\sqsupseteq v_i$ for each $i< d$, with $\im(\xi_m)$ a level set, with $\age_\Delta(\phi^+[\Theta(m)]) = \age_\Delta(\xi_m) = \age_\bK(\sigma_m)$, and with $(\sigma_m, \xi_m)\in \pair$. 
\end{enumerate}
When $m = 0$, we let $\xi_0(i) = v_i\in \MP\times \{\emptyset\}$ and $\sigma_0(i)$ be given by Proposition~\ref{Prop:BRD_Lower_Bound_Singletons}.

There are now three cases depending on whether $m\in \sp(\Theta)$, $\ac(\Theta)$, or $\cd(\Theta)$. 
\vspace{2 mm}

\textbf{Case 1:} $m\in \sp(\Theta)$. Suppose $t_i\in \Theta(m)$ is the splitting node. Write $t_i^\sf{p} = \frak{p}\in \MP$ and $t_i^\sf{u} = h< \sfU$. This case splits into three further sub-cases. Two of these cases are quite similar. In the first case, we have $h\in \sfU_{fr}$. Let $1\leq q< k$ and $j< \sfU$ witness that $h\in \sfU_{fr}$, i.e.\ that $\cK{\cdot}\gamma_{j, h, q} = \cK{\cdot}\tilde{\iota}_i$. In the second case, we have $h \in \sfU_{non}$, but with $t_i\neq (\frak{p}, 0^m)$.  We note that $\age_\Theta(t_i) = \age_\Delta(\xi_m(i)) = \age_\bK(\sigma_m(i))\subsetneq \max(P_h)$. It follows that $\sigma_m(i)^\sf{seq}\neq 0^{\ell(\sigma_m(i))}$. If $n< \ell(\sigma_m(i))$ is least with $\sigma_m(i)(n)\neq 0$, write $q = \sigma_m(i)(n)$, and write $c^\bK(n)^\sf{u} = j$. For these two cases, the proof is now identical with these choices of $q< k$ and $j< \sfU$. Find some $r> \max\{\ell(u): u\in \im(\sigma_m)\}$ so that $c^\bK(r) = (j, 0^r)$ and $\ell^\Delta(r) > \max\{\ell(\xi_m(\alpha)): \alpha< d\}$. It follows that we have
$$\age_\bK(\sigma_m(i)) = \age_\bK(\Left(\sigma_m(i), r+1)) = \age_\bK(\Left(\sigma_m(i), r)^\frown q).$$
In particular, writing $\frak{q} = \pcon(\xi_m(i))$, we have that both $(\Left(\sigma_m(i), r+1), \xi_m(i))$ and $(\Left(\sigma_m(i), r)^\frown q, \xi_m(i))$ are in $\pair(\frak{q})$. Using Proposition~\ref{Prop:MovingPairsUpAgePreserve}, find $w_0\sqsupseteq \Left(\sigma_m(i), r+1)$, $w_q\sqsupseteq \Left(\sigma_m(i), r)^\frown q$, and $x_0, x_q\sqsupseteq \xi_m(i)$ with $\ell(x_0) = \ell(x_q) = \ell^\Delta(r)+1$, $(w_a, x_a)\in \pair(\frak{q})$ for each $a\in \{0, q\}$, and $\age_\bK(\{w_0, w_q\}) = \sp(\max(\frak{q}), 0)$; note that $w_0\lex w_q$ or $w_q\lex w_0$ are both possible. For $a\in \{0, q\}$, note that since $w_a(r) = a$ and since $(w_a, x_a)\in \pair(\frak{q})$, we must have $x_a(\ell^\Delta(r)) = a$ and $x_q(\ell^\Delta(r)) = q$. In particular, $x_0\neq x_q$. We set 
\begin{align*}
    \phi(t_i) = x_0\wedge x_q
\end{align*}
 and hence $\wt{\phi}(m) = \ell(x_0\wedge x_q)$. Put $y_i = \phi(t_i)^\frown 0$ and $y_{i+1} = \phi(t_i)^\frown 1$. For $a\in \{0, q\}$, Write $w_a = \ol{w}_i$ or $\ol{w}_{i+1}$ depending on if $\pi_{\wt{\phi}(m)+1}(x_a) = y_i$ or $y_{i+1}$. Define $\sigma'\colon (d+1)\to \ct^\bK$ and $\xi'\colon (d+1)\to \Delta$ be given by 
\begin{align*}
\sigma'(b) &= \begin{cases}
\sp(\sigma_m, i)(b) \quad &\text{if } b\not\in \{i, i+1\}\\
\ol{w}_b \quad &\text{if } b\in \{i, i+1\}
\end{cases}\\
\xi'(b) &= \begin{cases}
\sp(\xi_m, i)(b) \quad &\text{if } b\not\in \{i, i+1\}\\
y_b \quad &\text{if } b\in \{i, i+1\}.
\end{cases}
\end{align*}

We note that $\age_\bK(\sigma') = \sp(\age_\bK(\sigma_m), i) = \sp(\age_\Theta(m), i)$ and that $(\sigma', \xi')\in \pair$. Now use Proposition~\ref{Prop:MovingPairsUpAgePreserve} on $(\sigma', \xi')$
using a suitably large $N> \wt{\phi}(m)+1$. Doing this, we obtain functions $\sigma_{m+1}\colon (d+1)\to \ct^\bK$ and $\xi_{m+1}\colon (d+1)\to \Delta(N)$ with $\sigma_{m+1}\sqsupseteq \sigma'$, $\xi_{m+1}\sqsupseteq \xi'$, $(\sigma_{m+1}, \xi_{m+1})\in \pair$, and $\age_\bK(\sigma') = \age_\bK(\sigma_{m+1}) = \age_\Delta(\xi_{m+1})$. For each $b\in d\setminus \{i\}$, set 
$$\phi(t_b) = \pi_{\wt{\phi}(m)}\circ \xi_{m+1}\circ ((d+1){\setminus}\{i\})(b).$$ 
We need to check that $\age_\Theta(m) = \age_\Delta(\phi[\Theta(m)])$. To do this, we define $\phi^+\colon \Theta(m+1)\to \Delta$  from $\phi|_{\Theta(m)}$ in the only way we can and verify that $\age_\Theta(m+1) = \age_\Delta(\phi^+[\Theta(m+1)])$. Note that if we write $\phi^+[\Theta(m+1)] = \{z_0\lex\cdots\lex z_d\}$, we have $\xi_{m+1}(b)\sqsupseteq z_b$ for each $b< d+1$. Hence, we have:
\begin{align*}
    \age_\Theta(m+1) &= \sp(\age_\Theta(m), i)\\
    &= \age_\bK(\sigma')\\
    &= \age_\Delta(\xi_{m+1})\\
    &\subseteq \age_\Delta(\phi^+[\Theta(m+1)]).
\end{align*}
The reverse inequality follows from the observation that
\begin{align*}
    \age_\Delta(\phi^+[\Theta(m+1)]) &\subseteq \sp(\age_\Delta(\phi[\Theta(m)]), i)\\
    &= \sp(\age_\Theta(m), i)\\
    &= \age_\Theta(m+1).
\end{align*}

In the third case, we have $h\in \sfU_{non}$ and $t_i = (\frak{p}, 0^m)$. Since we have $\age_\Delta(\xi_m) = \age_\Delta(\phi^+[\Theta(m)])$, we must also have $\xi_m(i)^\sf{seq} = 0^{\ell(\xi_m(i))}$. We first apply Proposition~\ref{Prop:LiftPairFixedLevel} to find $u\sqsupseteq \sigma_m(i)$ and $v\sqsupseteq \xi_m(i)$ with $(u, v)\in \pair(\frak{p}')$. As $\age_\Delta(v) = \max(\frak{p}')\subsetneq \max(\frak{p})$, we have $v^\sf{seq}\neq 0^\ell(v)$. If $n< \ell(v)$ is least with $v(n)\neq 0$, then $n> \ell(\xi_m(i))$, and we set 
$$\phi(t_i) = v|_n$$
and hence $\wt{\phi}(m) = n$. Form the functions $\sigma'\colon (d+1)\to \ct^\bK$ and $\xi'\colon (d+1)\to \Delta$ as follows.
\begin{align*}
    \sigma'(b) &= \begin{cases}
    \sp(\sigma_m, i)(b) \quad &\text{if } b\neq i+1\\
    u \quad &\text{if } b = i+1
    \end{cases}\\
    \xi'(b) &= \begin{cases}
    \sp(\xi_m, i)(b) \quad &\text{if } b\neq i+1\\
    v \quad &\text{if } b = i+1.
    \end{cases}
\end{align*}
We note that $(\sigma', \xi')\in \pair$. From here, the definitions of $\sigma_{m+1}$, $\xi_{m+1}$, and $\phi(t_b)$ for $b\in d\setminus \{i\}$ as well as the verification that $\age_\Theta(m+1) = \age_\Delta(\phi^+[\Theta(m+1)])$ are identical to the first two cases.  
\vspace{2 mm}

\textbf{Case 2:} $m\in \ac(\Theta)$. Write $\rho = \sort(\Theta(m))$, and write $e\colon d_0\to d$ for the unique increasing injection with $(\age_\Theta(m){\cdot}e, \age_\Theta(m+1){\cdot}e)\in \econ(\rho\circ e)$ (Proposition~\ref{Prop:Consecutive}). There are two cases. If $d_0 > 1$, we first find a function $\sigma'\colon d\to \ct^\bK$ with $\sigma'\sqsupseteq \sigma_m$ so that $\age_\bK(\sigma') = \age_\Theta(m+1)$. When $d_0> 1$, we still have $(\sigma', \xi_m)\in \pair$, so run Proposition~\ref{Prop:MovingPairsUpAgePreserve} on $(\sigma', \xi_m)$ with some suitably large $N$, obtaining functions $\sigma_{m+1}\colon (d+1)\to \ct^\bK$ and $\xi_{m+1}\colon (d+1)\to \Delta(N)$ with $\sigma_{m+1}\sqsupseteq \sigma'$, $\xi_{m+1}\sqsupseteq \xi_m$, $(\sigma_{m+1}, \xi_{m+1})\in \pair$, and $\age_\bK(\sigma') = \age_\bK(\sigma_{m+1}) = \age_\Delta(\xi_{m+1})$. To determine $\phi$, we search between the levels of $\xi_m$ and $\xi_{m+1}$ to find the level of the age-change, and call this level $\wt{\phi}(m)$, and then set $\phi(t_b) = \pi_{\wt{\phi}(m)}\circ \xi_{m+1}(b)$. Define $\phi^+|_{\Theta(m+1)}$ in the only way possible, and by construction $\age_\Theta(m+1) = \age_\Delta(\phi^+[\Theta(m+1)])$.

If $d_0 = 1$, suppose $e = \iota_i$ for some $i< d$. Set $\frak{p} = \pcon(\xi_m(i))$. We use Proposition~\ref{Prop:LiftPairFixedLevel} to find $(u, v)\in \pair(\frak{p}')$ with $u\sqsupseteq \sigma_m(i)$ and $v\sqsupseteq \xi(i)$. We then consider the functions $\sigma'$ and $\xi'$, which are identical to $\sigma_m$ and $\xi$ except we set $\sigma'(i) = u$ and $\xi'(i) = v$, and then apply Proposition~\ref{Prop:MovingPairsUpAgePreserve}. The rest of the proof is similar to the $d> 1$ case.
\vspace{2 mm}

\textbf{Case 3:} $m\in \cd(\Delta)$. Let $j< d$ be such that $t_j\in \cdnd(\Theta)$, and set $t_j^\sf{p} = \frak{p}\in \MP$, $t_j^\sf{u} = h< \sfU$. Write $\c{A} = \age_\Theta(m)$, so also $\age_\bK(\sigma_m) = \age_\Delta(\xi_m) = \c{A}$. We also write $\chi\colon (d-1)\to k$ for the function with $(t_{(d{\setminus}\{j\})(i)})^\frown \chi(i)\in \Theta(m+1)$ for each $i< d-1$. As $(\c{A}, j, \chi)$ is a controlled coding triple, item $2$ of Definition~\ref{Def:ControlledCodingTriple} gives us $\age_\Theta(\xi_m(j)) = \min(P_h)$, implying that $(\sigma_m(j), \xi_m(j))\in \pair(\min(P_h))$. So find $r< \omega$ so that $c^\bK(r)\sqsupseteq \sigma_m(j)$ and $c^\Delta(r)\sqsupseteq \xi_m(j)$. Define $\sigma'\colon (d-1)\to \ct^\bK$ by setting $\sigma'(i) = \Left(\sigma_m\circ (d{\setminus}\{j\})(i), r)^\frown\chi(i)$. Note that $\age_\bK(\sigma') = \c{A}{\cdot}\rm{Add}_{j, \chi}$, so item $1$ of Definition~\ref{Def:ControlledCodingTriple} yields $\age_\bK(\sigma') = \c{A}{\cdot}(d{\setminus}\{j\})$. 

We now run Proposition~\ref{Prop:MovingPairsUpAgePreserve} on $(\sigma', \xi_m\circ (d{\setminus}\{j\}))$ using some suitably large $N> r$. Doing this, we obtain functions $\sigma_{m+1}\colon (d-1)\to \ct^\bK$ and $\xi_{m+1}\colon (d-1)\to \Delta(N)$ with $\sigma_{m+1}\sqsupseteq \sigma'$, $\xi_{m+1}\sqsupseteq \xi_m\circ (d{\setminus}\{j\})$, $(\sigma_{m+1}, \xi_{m+1})\in \pair$, and $\age_\bK(\sigma') = \age_\bK(\sigma_{m+1}) = \age_\Delta(\xi_{m+1})$. We note that since $\sigma_{m+1}(i)\sqsupseteq \sigma'(i)$ for each $i < d-1$, we must have $\xi_{m+1}(i)(r) = \chi(i)$. We now set $\phi(t_j) = c^\Delta(r)$ (so $\wt{\phi}(m) = \ell^\Delta(r)$), and for $b\in d\setminus \{j\}$, we set 
\begin{align*}
\phi(t_b) = \begin{cases}
    \pi_{\wt{\phi}(m)}\circ \sp(\xi_{m+1}, j)(b) \quad \text{if }j< d-1,\\
    \pi_{\wt{\phi}(m)}\circ \xi_{m+1}(b) \quad \text{if }j = d-1.
\end{cases} 
\end{align*}
With this definition of $\phi|_{\Theta(m)}$, $\phi^+|_{\Theta(m+1)}$ is well defined. Lastly, to show that $\age_\Delta(\phi[\Theta(m)]) =  \c{A}$, write  $\c{B} = \age_\Delta(\phi[\Theta(m)])$. Then $\c{B}\subseteq \c{A}$, $\c{B}{\cdot}(d{\setminus}\{j\}) = \c{A}{\cdot}(d{\setminus}\{j\})$, and $\c{B}{\cdot}\rm{Add}_{j, \chi} = \c{A}{\cdot}\rm{Add}_{j, \chi}$. As $(\c{A}, j, \chi)$ is a controlled coding triple, item $3$ of Definition~\ref{Def:ControlledCodingTriple} gives us that $\c{A} = \c{B}$. Lastly, $\age_\Theta(m+1) = \age_\Delta(\phi^+[\Theta(m+1)])$ by the properties of controlled coding triples.

\subsection{Degenerate unaries}
\label{Subsec:Degenerates}

Up until now, we have been assuming that every $i< \sfU$ is a non-degenerate unary predicate (Definition~\ref{Def:NonDegenerate}). We now briefly discuss the modifications that need to be made when there are degenerate unary predicates; luckily, these modifications are all very straightforward.

\textbf{We now assume} that for some $\sfU_{deg}\subseteq \sfU$, every $i\in \sfU_{deg}$ is degenerate. We have for each $i\in \sfU_{deg}$ that $\{a\in \b{K}: U^\bK(a) = i\}$ is an infinite set of isolated points. Letting $\b{K}_{nd} = \{a\in \b{K}: U^\bK(a)\not\in \sfU_{deg}\} \subseteq \b{K}$ denote the induced substructure on the points with non-degenerate unary predicate, we have $\aut(\b{K}) \cong \aut(\b{K}_{nd})\times (S_\infty)^{|\sfU_{deg}|}$. If $\bA\leq \bK$ is enumerated, then for every $a< A$ with $U^\bA(a) \in \sfU_{deg}$, we have $c^\bA(a)^{\sf{seq}} = 0^a$. 

We now collect all of the new conventions and definition modifications needed to define diaries and prove Theorem~\ref{Thm:EmbDiary} in this slightly more general setting.
\begin{enumerate}
    \item 
    Degenerate unary predicates are not free; however, we reserve the notation $\sfU_{non}$ for the set of non-free \emph{and non-degenerate} unaries. Hence $\sfU = \sfU_{fr}\sqcup \sfU_{non}\sqcup \sfU_{deg}$, and similarly we write $\MP = \MP_{fr}\sqcup \MP_{non}\sqcup \MP_{deg}$. If $i\in \sfU_{deg}$, then $P_i = \{\c{K}{\cdot}\tilde{\iota}_i\}$ and $|\MP_i| = 1$.
    \item 
    In Definition~\ref{Def:Diary}, the levels $m< \rm{ht}(\Delta)$ with $|\is_\Delta(s)| = 1$ for every $s\in \Delta(m)$ now can be either age-change levels or a ``degenerate coding" level. Writing $\rm{Dcd}(\Delta)$ for the set of degenerate coding levels, we have $\rm{ht}(\Delta) = \cd(\Delta)\sqcup \rm{DCd}(\Delta)\cup \sp(\Delta)\cup \ac(\Delta)$.
    
    If $m\in \sp(\Delta)\cup \cd(\Delta)$ and $t\in \Delta(m)$ is the splitting or coding node, we demand $t^\sf{u}\not\in \sfU_{deg}$. Hence if $s\in \Delta(m)$ has $s^\sf{u}\in \sfU_{deg}$, then $s^\sf{seq} = 0^m$. 
    
    If $m\in \rm{DCd}(\Delta)$, then we designate exactly one $t\in \Delta(m)$ with $t^\sf{u}\in \sfU_{deg}$ to be a coding node. We can choose whether or not $t$ is terminal. If $t$ is terminal, then writing $S = \Delta(m)\setminus \{t\}$, we set $\Delta(m+1) = \{s^\frown 0: s\in S\}$ and $\age_\Delta(m+1) = \age_\Delta(S)$. If $t$ is non-terminal, we set $\Delta(m+1) = \{s^\frown 0: s\in \Delta(m)\}$ and $\age_\Delta(m+1) = \age_\Delta(m)$.  Hence $\cdnd(\Delta)$ now contains every terminal node of $\Delta$ (in particular, all non-degenerate coding nodes) as well as possibly some non-terminal degenerate coding nodes.

    \item 
    In the proof of Proposition~\ref{Prop:Induced_Subdiary}, we explicitly define $\cdnd(\Delta\|_X) = \im(\pi_{\Delta, X})$.

    \item 
    The construction of the previous subsection now has a case $4$, namely $m\in \rm{DCd}(\Theta)$. If $t_j\in \Theta(m)$ is the coding node, find a suitably large $N$ with $\Left(\xi_m(j), N-1)\in \cdnd(\Delta)$. Apply Proposition~\ref{Prop:MovingPairsUpAgePreserve} on $(\sigma_m, \xi_m)$ with level $N$ to obtain $(\sigma_{m+1}, \xi_{m+1})$, and for each $i< d$, set $\phi(t_i) = \pi_{N-1}\circ \xi_{m+1}(i)$. Define $\phi^+|_{\Theta(m+1)}$ as we must, depending on if $t_j$ is terminal.
\end{enumerate}

\section{Upper bounds}
\label{Sec:Upper_bounds}

This section uses the ``coding tree Milliken theorem," Theorem~3.5 from \cite{Zuc22}, to show that the lower bounds we produced in Sections~\ref{Sec:Diaries} and \ref{Section:Proof_Of_Thm_EmbDiary} are sharp.

\textbf{For this section}, we take $\bK$ to be an enumerated structure. Furthermore, we choose this enumeration to be \emph{left dense} (see \cite{Zuc22} for the definition). For us, this amounts to saying that for every $t\in \ct^\bK$, there is $n\geq \rm{ht}(t)$ with $\Left(t, n) = c^\bK(n)$. Every \fr free amalgamation class admits an enumerated left-dense \fr limit. 

\begin{theorem*}[Theorem 3.5 from \cite{Zuc22}]
    For any enumerated $\b{A}\in \c{K}$, $r< \omega$, and coloring $\chi\colon\aemb(\ct^\bA, \ct^\bK)\to r$, there is $\phi\in \aemb(\ct^\bK, \ct^\bK)$ so that $\chi\circ \phi$ is constant.
\end{theorem*}

\begin{rem}
    Coding trees of enumerated structures are defined slightly differently in \cite{Zuc22} (see Section~\ref{Sec:Diaries} for a discussion); however, the proof of the theorem for the version we use here is almost identical.
\end{rem}

\subsection{Shapes of embeddings into enumerated structures}
\label{SubSec:Embedding_Shapes}

\begin{notation}
    Let $\cL' = \cL\cup \{\leq_\ell\}$. We can view an $\cL$-structure $\bA$ with $A\subseteq \omega$ as an $\cL'$-structure by letting $\leq_\ell$ be the usual order. If $\Theta$ is a diary, we can view $\str(\Theta)$ as an $\cL'$-structure where $\leq_\ell$ is the order of relative levels. We will refer to embeddings between these instances of $\cL'$-structures as  \emph{ordered} embeddings; any reference to embeddings without the word ``ordered" refers to as $\cL$-structures. If $\bA$ and $\bB$ are $\cL'$-structures as above, we write $\oemb(\bA, \bB)$ for the $\cL'$-embeddings (``O" for ``ordered") and $\emb(\bA, \bB)$ for the $\cL$-embeddings. \qed
\end{notation}

We recall some facts from \cite{Zuc22} that we will need going forward. 

\begin{fact}
\label{Fact:Shape_Facts}
    \begin{enumerate}
        \item 
        By Theorem~2.10 of \cite{Zuc22}, $\aemb(\ct^\bY, \ct^\bK)\neq \emptyset$ for any enumerated structure $\bY\leq \bK$. If $\bE, \bY\leq \bK$ are enumerated structures and $\phi\in \aemb(\ct^\bE, \ct^\bY)$, note that $\wt{\phi}\in \oemb(\bE, \bY)$.
        \item 
        Conversely, fix enumerated structures $\bE, \bY\leq \bK$. Given $S\subseteq Y$, we call $S$ \emph{$\bY$-closed} if $\crit^\bY(c^\bY[S]) = S$, and the \emph{closure of $S$ in $\bY$}, denoted $\ol{S}^\bY$, is the smallest $\bY$-closed subset of $Y$ containing $S$. If $\sigma\in \oemb(\bE, \bY)$, then $\sigma = \wt{\phi}$ for some $\phi\in \aemb(\ct^\bE, \ct^\bY)$ iff $\im(\sigma) = \ol{\im(\sigma)}^\bY$; this is more-or-less Proposition~4.2 and Definition~5.1 from \cite{Zuc22}. Furthermore, \emph{because in this paper we define $\ct^\bE = \im(c^\bE){\downarrow}$}, such a $\phi$ is unique. \qed
    \end{enumerate}
\end{fact} 

Because of this fact, we can work interchangeably with either aged embeddings of $\ct^\bE$ into $\ct^\bY$ or ordered embeddings of $\bE$ into $\bY$ with $\bY$-closed image.

\begin{defin}
    \label{Def:Embedding_Shape}
    Fix an enumerated structure $\bY\leq \bK$, a structure $\bA\leq \bY$, and $f\in \emb(\bA, \bY)$. The \emph{$\bY$-shape} of $f$, denoted $\shp_\bY(f)$, is the unique pair $(\bE, g)$ such that:
    \begin{enumerate}
        \item 
        $\bE\leq \bY$ is an enumerated structure and $g\in \emb(\bA, \bE)$.
        \item 
        There is (a unique) $\alpha\in \oemb(\bE, \bY)$ with $\alpha\circ g = f$ and $\im(\alpha) = \ol{\im(f)}^{\bY}$. 
    \end{enumerate}
    In particular, $\bE\cong \bY|_{\ol{\im(f)}^{\bY}}$. \qed
\end{defin}

\begin{fact}
\label{Fact:Upper_Bounds}
    Reasoning as in the proof of Theorem~4.6 from \cite{Zuc22}, we have for any $\bA\in \cK$ and for any $\eta\in \emb(\bK, \bK)$ that $\rm{BRD}(\bA, \cK)\leq |\{\shp_\bK(\eta\circ f): f\in \emb(\bA, \bK)\}|$. \qed
\end{fact}

The main theorem of this section is an application of Fact~\ref{Fact:Upper_Bounds}. Recall, given $\bA\in \cK$, the set $D_\bA$ and the coloring $\shp_{\Delta, \bA}\colon \emb(\bA, \str(\Delta))\to D_\bA$ from Definition~\ref{Def:Diary_Shape}.

\begin{theorem}
\label{Thm:Upper_Bound}
    Fix a diary $\Delta$ with $\str(\Delta)\cong \bK$. There is $\eta\in \emb(\str(\Delta), \bK)$ such that for any $\bA\in \cK$ and $f\in \emb(\bA, \str(\Delta))$, $\shp_\bK(\eta\circ f)$ depends only on $\shp_{\Delta}(f)$. Thus $\rm{BRD}(\bA, \cK)\leq |D_\bA|$.
\end{theorem}

We prove Theorem~\ref{Thm:Upper_Bound} in the next subsection. Combined with Theorem~\ref{Thm:EmbDiary}, we have $\rm{BRD}(\bA, \cK) = |D_\bA|$, proving Theorem~\ref{thm:main_intro}. Note that Theorems~\ref{Thm:Expansion_Codes_Diary} and \ref{Thm:Upper_Bound} together imply that whenever $\Delta$ is a diary with $\str(\Delta)\cong \bK$, then $\str^*(\Delta)$ (Definition~\ref{Def:Diary_Exp}) is a big Ramsey structure. The following proves Theorem~\ref{thm:main_intro_2}.

\begin{theorem}
    \label{Thm:Strong_BRS}
    Whenever $\Delta$ is a diary with $\str(\Delta)\cong \bK$, then $\str^*(\Delta)$ is a strong big Ramsey structure.
\end{theorem}

\begin{proof}
    Write $\bS^* = \str^*(\Delta)$. Fix $\bA^*\in \age(\bS^*)$ and a coloring $\gamma\colon \emb(\bA^*, \bS^*)\to 2$. Writing $\bA = \bA^*|_\cL$, we obtain a coloring $\ol{\gamma}\colon \emb(\bA, \str(\Delta))\to 2\cup \bS^*(\bA){\setminus}\{\bA^*\}$ by setting 
    $$\ol{\gamma}(f) = \begin{cases}
        \bS^*{\cdot}f \quad &\text{if } \bS^*{\cdot}f\neq \bA^*\\
        \gamma(f) \quad &\text{if } \bS^*{\cdot}f = \bA^*.
    \end{cases}$$ 
    By Theorem~\ref{Thm:Upper_Bound}, there is $\eta\in \emb(\str(\Delta), \str(\Delta))$ such that $|\im(\ol{\gamma}\circ \eta)|\leq |\bS^*(\bA)|$, i.e.\ so that we can avoid one of the colors of $\ol{\gamma}$. Write $\Theta = \Delta\|_{\im(\eta)}$. As $\str(\Theta)\cong \bK$, Theorem~\ref{Thm:EmbDiary} implies that the avoided color is one of the two colors from the $\bS^*{\cdot}f = \bA^*$ case. Write $\phi = \phi_{\Delta, \im(\eta)}$, and using Theorem~\ref{Thm:EmbDiary}, fix some $\psi\in \demb(\Delta, \Theta)$. Then $\phi\circ \psi|_{\str(\Delta)}:= \sigma\in \emb(\bS^*, \bS^*)$ satisfies that $\gamma\circ \sigma$ is constant.
\end{proof}

\begin{rem}
    Fix a \fr class $\cM$ with $\flim(\cM) = \bM$. If $\bM^*$ is an expansion of $\bM$, call $\bM^*$ \emph{recurrent} if for every $\eta\in \emb(\bM, \bM)$, there is $\theta\in \emb(\bM, \bM)$ with $\bM^*{\cdot}(\eta\circ \theta) = \bM^*$. Hence Theorem~\ref{Thm:EmbDiary} asserts that $\str^*(\Delta)$ is a recurrent expansion of $\str(\Delta)$. The argument in Theorem~\ref{Thm:Strong_BRS} shows generally that if $\bM^*$ is a recurrent big Ramsey structure for $\cM$, then it is strong. In a partial converse, one can show that if $\bM^*$ is a strong big Ramsey structure in a finite relational language, then $\bM^*$ is recurrent. 
\end{rem}

\subsection{Proof of Theorem~\ref{Thm:Upper_Bound}}

\textbf{We fix} a diary $\Delta$ which codes $\bK$. \textbf{To ease notation}, we assume $\sfU_{deg} = \emptyset$. The proof is almost identical when $\sfU_{deg} \neq \emptyset$.

Fix a function $\sf{f}\colon \MP\to \sfU\times (k{\setminus}\{0\})$ such that if $\frak{p}\in \MP$, $\sf{u}(\frak{p}) = i< \sfU$, and $\sf{f}(\frak{p}) = (j, q)$, then if $i\in \sfU_{fr}$, we have $\cK{\cdot}\gamma_{j, i, q} = \cK{\cdot}\tilde{\iota}_i$ (Definition~\ref{Def:Free}), and if $i\in \sfU_{non}$, then $\cK{\cdot}\gamma_{j, i, q} = \max(\frak{p}')$ (Fact~\ref{Fact:NonFree_Unaries}). We can arrange for $i\in \sfU_{fr}$ that $|\sf{f}[\MP_i]| = 1$, so simply write $\sf{f}(i)$. 

Write $\econ = \bigsqcup_\rho \econ(\rho)$, where the union is taken over all path sorts. By Fact~\ref{Fact:Bound_Essential} and Definition~\ref{Def:Path_Sorts}, there are only finite many path sorts $\rho$ with $\econ(\rho)\neq \emptyset$, so in particular, $\econ$ is finite. For each $(\cA, \cB)\in \econ(\rho)$, fix a gluing $\gamma_{\cA, \cB} = (\bX_{\cA, \cB}, \rho, \eta_{\cA, \cB})$ such that the following both hold.
\begin{enumerate}
    \item 
    For any gluing $\gamma_\cA$ with $\cK{\cdot}\gamma_\cA = \cA$, the gluing $\delta = \gamma_\cA\sqcup \gamma_{\cA, \cB}$ satisfies $\cK{\cdot}\delta = \cB$. 
    \item 
    $X_{\cA, \cB}$ has minimum possible cardinality.
\end{enumerate}
To see that a gluing as in item $1$ exists, one can take any $\gamma_\cB$ with $\cK{\cdot}\gamma_\cB = \cB$. In the case when $\dom(\rho) = 1$, $\rho(0) = \frak{p}\in \MP_i$ for $i\in \sfU_{non}$, and $\cA = \max(\frak{p})$, then writing $\sf{f}(\frak{p})= (j, q)$, we arrange that $\gamma_{\cA, \cB} = \gamma_{j, i, q}$.  

We will use $\sf{f}$ and the $\gamma_{\cA, \cB}$ to systematically build for every diary $\Theta$ an enumerated structure $\bY^\Theta\leq \bK$ and $\sigma^\Theta\in \oemb(\str(\Theta), \bY^\Theta)$. When $\Theta = \Delta$, we omit the $\Delta$-superscript, and we will show that for any $\bA\in \cK$ and $f\in \emb(\bA, \str(\Delta))$,  $\shp_{\bY}(\sigma\circ f)$ depends only on $\shp_{\Delta}(f)$, which will suffice to prove Theorem~\ref{Thm:Upper_Bound} (pick any $\phi\in \aemb(\ct^{\bY}, \ct^\bK)$ and set $\eta = \wt{\phi}\circ \sigma$). 

To build $\bY^\Theta$, we first build an $\cL'$-structure $\bZ^\Theta\supseteq \str(\Theta)$ such that $\leq_\ell$ has order type $|Z_\Theta|$. We will let $\bY^\Theta$ be the enumerated $\cL$-structure which is $\cL'$-isomorphic to $\bZ^\Theta$ and $\sigma^\Theta\colon \str(\Theta)\to \bY^\Theta$ be induced by the inclusion $\str(\Theta)\subseteq \bZ^\Theta$. To build $\bZ^\Theta$, we start with $\str(\Theta)$ and for each $m< \rm{ht}(\Theta)$, we attach a new finite $\cL'$-structure $\bZ^\Theta_m$ to $\str(\Theta)$. So suppose $m< \rm{ht}(\Theta)$ and we have constructed an $\cL'$-structure $\str(\Theta)\cup \bigcup_{i< m}\bZ^\Theta_i$ embeddable into $\bK$ as an $\cL$-structure and so that $\leq_\ell$ is either a finite order or an $\omega$-order. If $m\in \cd(\Theta)$, we set $\bZ^\Theta_m = \emptyset$. The other two cases are:
\vspace{2 mm}

\textbf{Case 1:} $m\in \sp(\Theta)$. Let $Z^\Theta_m = \{z^\Theta_m\}$ for some new point $z^\Theta_m$. Suppose $t\in \Theta(m)$ is the splitting node, with $t^\sf{u} = i$ and $t^\sf{p} = \frak{p}$. Write $\sf{f}(\frak{p}) = (j, q)$. We set $U^{\bZ^\Theta_m}(z^\Theta_m) = j$ and for each $y\in \cdnd(\Theta)$ with $y\sqsupseteq s^\frown 1$, we set $R^{\bZ^\Theta}(y, z^\Theta_m) = q$. All other binary relations  with $z^\Theta_m$ are zero. Note by the definition of $\sf{f}$, case $2$ of Definition~\ref{Def:SplittingClass}, and Proposition~\ref{Prop:Diary_Correct_Age} that $\str(\Theta)\cup\bigcup_{-1\leq i\leq m} \bZ^\Theta_i\leq \bK$. We put $z^\Theta_m$ as $\leq_\ell$-small as possible while being $\leq_\ell$-above all members of $\cdnd(\Theta)\cap \Theta({<}m)$ and $\bigcup_{i< m} Z^\Theta_i$.
\vspace{2 mm}

\textbf{Case 2:} $m\in \ac(\Theta)$. Suppose the age change is essential on $S = \{s_0\lex\cdots\lex s_{d-1}\}\subseteq \Theta(m)$, and write $\rho = \sort(S)$. Set $\age_\Theta(S) = \cA$ and $\age_\Theta(\{s_i^\frown 0: i< d\}) = \cB$ so that $(\cA, \cB)\in \econ(\rho)$. Let $Z^\Theta_m = X_{\cA, \cB}\times \{z^\Theta_m\}$ for some new point $z^\Theta_m$, and let $\bZ^\Theta_m$ be isomorphic to $\bX_{\cA, \cB}$ in the obvious way, with the $\leq_\ell$ order induced from the usual $\leq$-order on $X_{\cA, \cB}\subseteq \omega$. For each $x\in X_{\cA, \cB}$ and $y\in \cdnd(\Theta)$ with $s_i\sqsubseteq y$ for some $i< d$, we set $R^{\bZ^\Theta}(y, (x, z^\Theta_n)) = \eta_{\cA, \cB}(i, x)$. There are no other new non-zero binary relations. By our choice of $\gamma_{\cA, \cB}$ and Proposition~\ref{Prop:Diary_Correct_Age}, we have $\str(\Theta)\cup \bigcup_{i\leq m}\bZ^\Theta_m\leq \bK$. We put $Z^\Theta_m$ as an $\leq_\ell$-consecutive interval as $\leq_\ell$-small as possible while being $\leq_\ell$-above all members of $\cdnd(\Theta)\cap \Theta({<}m)$ and $\bigcup_{i< m} Z^\Theta_i$.
\vspace{2 mm}

To finish defining $\bZ^\Theta$, we attach one last finite $\cL'$-structure $\bB^\Theta$. Given $i\in \sfU_{fr}$, write $\sfP^\Theta_i = \{s^\sf{p}: s\in \Theta_i\}$, and write $\sfQ^\Theta_i = \sfP^\Theta_i\setminus \{\min_{\leq_{\MP}}(\sfP^\Theta_i)\}$. We let 
\begin{align*}
\sfR^\Theta_i = \begin{cases}
    \sfP^\Theta_i\quad &\text{if } \exists m< \rm{ht}(\Theta)\,\, \exists x\in Z^\Theta_m\,\, U^{\bZ^\Theta_m}(x) = i,\\
    \sfQ^\Theta_i\quad &\text{else }
\end{cases} 
\end{align*}
We set $B^\Theta = (\bigcup_{j\in \sfU_{fr}} \sfR^\Theta_j)\times \{b^\Theta\}$ for some new point $b^\Theta$. In $\bB^\Theta$, there are no non-zero binary relations. If $\frak{p}\in \sfR^\Theta_i$ and $\sf{f}(i) = (j, q)$, we set $U^{\bB^\Theta}((\frak{p}, b^\Theta)) = j$. If $y\in \cdnd(\Theta)\cap \Theta_\frak{p}$, we set $R^{\bZ^\Theta}(y, (\frak{p}, b^\Theta)) = q$. There are no other new non-zero binary relations; by the definition of $\sf{f}$, the resulting structure embeds into $\bK$. On $\bB^\Theta$, the $\leq_\ell$-order is induced from $\leq_\MP$, and in $\bZ^\Theta$, we put $B^\Theta$ $\leq_\ell$-below everything else.

This finishes the construction of $\bZ^\Theta$, so also of $\bY^\Theta$ and $\sigma^\Theta\colon \str(\Theta)\to \bY^\Theta$. We mildly abuse notation and also write $\sigma^\Theta\colon \bZ^\Theta\to \bY^\Theta$ for the $\cL'$-isomorphism. We will primarily make use of $\bY^\Delta$, and simply write $\bY, \bZ, \bZ_m, z_m, b, \sigma$, etc. when referring to $\Delta$. 

The intuition behind the construction of $\bY$ is to as much as possible build $\Delta$ into a coding tree of an enumerated structure. With this idea in mind, let us investigate $\bY$ more closely.  We start by defining $h\colon \omega\to \omega$ by setting
$$h(m) = \left|B\cup \left(\cdnd(\Delta)\cap \Delta({<}m)\right)\cup \bigcup_{i< m} Z_i\right|.$$
We also write $\xi:= c^\bY\circ \sigma\colon Z\to \ct^\bY$. We now collect facts about $\bY$, $h$, and $\xi$ we will need going forward. 
\begin{fact}
\label{Fact:Y_Facts}
\begin{enumerate}
    \item 
    If $s\in \cdnd(\Delta)$, then $\ell(\xi(s)) = h(\ell(s))$. If $x\in Z_n$ is $\leq_\ell$-least, then $\ell(\xi(x)) = h(n)$. 
    \item 
    If $n< \omega$ and $x\in Z_n$, then $(\xi(x)|_{h(n)})^\sf{seq} = 0^{h(n)}$. 
    \item 
    Fix $i< \sfU$ and $s\in \cdnd(\Delta_i)$.
    \begin{itemize}
        \item 
        If $i\in \sfU_{fr}$ and $\sfR_i = \sfP_i$, then for every $s\in \cdnd(\Delta_i)$, we have $\xi(s)^\sf{seq}\not\sqsupseteq 0^{h(0)}$. 
        \item 
        If $i\in \sfU_{non}$ and $s^\sf{seq}\neq 0^\ell(s)$, then if $m< \ell(s)$ is largest with $\pi_m(s)^\sf{seq} = 0^m$, then $m\in \sigma[Z_n]$ for some $n< \omega$.
    \end{itemize}
    
    \item 
    Fix $s, t\in \cdnd(\Delta)$.
    \begin{itemize}
        \item 
        If $s, t\in \cdnd(\Delta)$ and $s^\sf{p} = t^\sf{p}$, then $\ell(\xi(s)\wedge \xi(t)) = h(\ell(s\wedge t))$. 
        \item
        If $s^\sf{u} = t^\sf{u}\in \sfU_{fr}$ and $s^\sf{p}\neq t^\sf{p}$, then $\ell(\xi(s)\wedge \xi(t)) = \sigma((\min_{\leq_\MP}(s^\sf{p}, t^\sf{p}), b)) < h(0)$. 
        \item
        If $s^\sf{u} = t^\sf{u}\in \sfU_{non}$ and $s^\sf{p}\neq t^\sf{p}$, then if $m< \omega$ is largest with $\age_\Delta(s) = \age_\Delta(t) = \cK{\cdot}\tilde{\iota_i}$, then $\ell(\xi(s)\wedge \xi(t)) = h(m)$. 
     \end{itemize}
    \item 
    For any $d< \omega$, function $e\colon d\to \cdnd(\Delta)$, and $m\leq \min\{\ell(e(i)): i< d\}$ we have $\age_\Delta(\pi_m\circ e) = \age_{\bY}(\pi_{h(m)}\circ \xi\circ e)$. \qed
\end{enumerate}
\end{fact}

Now fix $f\in \emb(\bA, \str(\Delta))$, and write $\shp_{\Delta}(f) = (\Theta, g)\in D_\bA$. We will show that $\shp_{\bY}(\sigma\circ f) = (\bY^\Theta, \sigma^\Theta\circ g)$, which will prove Theorem~\ref{Thm:Upper_Bound}. First, we use $f$ to build $\beta \in \oemb(\bZ^\Theta, \bZ)$ as follows.   Start by setting $\beta((\frak{p}, b^\Theta)) = (\frak{p}, b)$ for each $(\frak{p}, b^\Theta)\in B^\Theta$; we note that $\sfR^\Theta_i \subseteq \sfR_i$ for each $i\in \sfU_{fr}$. Write $\phi_f = \phi_{\Delta, \im(f)}$. For $m< \omega$, since $\phi_f\in \demb(\Theta, \Delta)$,  we have that $\bZ^\Theta_m$ and $\bZ_{\wt{\phi}_f(m)}$ are $\cL'$-isomorphic, and we let $\beta|_{\bZ^\Theta_m}$ be the unique $\cL'$-isomorphism. Given $x\in \cdnd(\Theta)$, we set $\beta(x) = \phi_f(x)$. One needs to check that $\beta$ respect the binary relations of $\bZ^\Theta$, but this follows from the construction of $\bZ^\Theta$ and $\bZ$ and since $\phi_f\in \demb(\Theta, \Delta)$. This concludes the definition of $\beta$, and we let $\alpha\in \oemb(\bY^\Theta, \bY)$ be induced from $\beta$ in the obvious way. Note that $\alpha\circ \sigma^\Theta = \sigma\circ (\phi_f|_{\cdnd(\Theta)})$ and that $\im(\alpha) = \im(\sigma\circ \beta)$. 

We now show that $\im(\alpha) = \ol{\im(\sigma\circ f)}^{\bY}$. As $\im(\sigma\circ f)$ is finite, we can compute its $\bY$-closure using the ``top-down" procedure discussed in \cite{Zuc22} after Definition~4.4, which goes as follows: Given an arbitrary $S\subseteq Y$, set $S_{\max(S)} = \{\max(S)\}$. If $m< \max(S)$ and $S_{m+1}$ has been defined, we set $S_m = S_{m+1}$ or $S_m = S_{m+1}\cup \{m\}$. The latter happens iff $m\in S$ or $m\in \crit^\bY(c^\bY[S_{m+1}])$. Upon reaching $m = 0$, we have $S_0 = \ol{S}^{\bY}$. 

Now write $S = \im(\sigma\circ f)$, and consider the above procedure. We prove by reverse induction on $m\leq \max(S)$ that $S_m = \im(\alpha)\setminus m$. So fix $m\leq \max(S)$ and assume this holds for all $n$ with $m< n\leq \max(S)$. 
\begin{itemize}
    \item 
    If $m = \sigma(t)$ for some $t\in \cdnd(\Delta)$, then if $t\in \im(f)$, we have $m\in S$ and $S_m = S_{m+1}\cup \{m\}$. If $t\not\in \im(f)$, it follows from our inductive hypothesis and Fact~\ref{Fact:Y_Facts}(2) that any $x\in \pi_{m+1}\circ c^\bY[S_{m+1}\setminus S]$ satisfies $x^\sf{seq} = 0^{\ell(x)}$. It then follows from Fact~\ref{Fact:Y_Facts} $(3)$, $(4)$, and $(5)$ that $m\not\in \crit^\bY(c^\bY[S_{m+1}])$.
    \item 
    If $m \in \sigma[Z_n]$ for some $n< \omega$, there are two cases to consider. 
    \begin{itemize}
        \item 
        If $n\in \sp(\Delta)$, then $m = h(n)$, and $m\in \im(\alpha)$ iff there is $i\in \sp(\Theta)$ with $\wt{\phi}_f(i) = n$. 
        
        If $m\in \im(\alpha)$, we can find $s, t\in \cdnd(\Theta)$ with $s^\sf{p} = t^\sf{p}$ and $\ell(s)\wedge \ell(t) = i$. Then by Fact~\ref{Fact:Y_Facts}(4), we have $\ell(\xi\circ \phi_f(s)\wedge \xi\circ \phi_f(t)) = h(n)$, and since $\xi\circ (\phi_f|_{\cdnd(\Theta)}) = c^\bY\circ \alpha\circ \sigma^\Theta$, this shows that $h(n) = m\in \crit^\bY(c^\bY[S_{m+1}])$. 
        
        If $m\not\in \im(\alpha)$, then $n\not\in \im(\wt{\phi}_f)$. In particular, by Fact~\ref{Fact:Y_Facts}(4), there are no incomparable $s, t\in c^\bY[S]$ with $\ell(s)\wedge \ell(t) = m$, and by Fact~\ref{Fact:Y_Facts}(5), $\age_\bY(\pi_{m+1}\circ c^\bY[S_{m+1}]) = \age_\bY(\pi_{m}\circ c^\bY[S_{m+1}])$. Hence $m\not\in \crit^\bY(c^\bY[S_{m+1}])$.

        \item 
        If $n\in \ac(\Delta)$, then we consider all $m\in \sigma[Z_n]$ simultaneously. Note that $\min(\sigma[Z_n]) = h(n)\leq m$.  We have $m\in \im(\alpha)$ iff there is $i\in \ac(\Theta)$ with $\wt{\phi}_f(i) = n$. 

        If $m\in \im(\alpha)$, i.e.\ if $\sigma[Z_n]\subseteq \im(\alpha)$, then as $i\in\ac(\Theta)$, there must be an age change between $\pi_{h(n)}\circ c^\bY[S_{m+1}\cap S]$ and $\pi_{h(n+1)}\circ c^\bY[S_{m+1}\cap S]$ by Fact~\ref{Fact:Y_Facts}(5). Hence $\crit^\bY(c^\bY[S_{m+1}])\cap \sigma[Z_n]\neq\emptyset$, and this intersection must be all of $\sigma[Z_n]$ since, if level $i$ of $\Theta$ features an essential age change from $\cA$ to $\cB$, we chose $\bX_{\cA, \cB}$ with $|X_{\cA, \cB}|$ as small as possible.

        If $m\not\in \im(\alpha)$, i.e.\ if $\sigma[Z_n]\cap \im(\alpha) = \emptyset$, then by Fact~\ref{Fact:Y_Facts} $(2)$, $(4)$, and $(5)$, we see that $\crit^\bY(c^\bY[S_{m+1}])\cap \sigma[Z_n] = \emptyset$ and $S\cap \sigma[Z_n] = \emptyset$, so $S_{m+1} = S_{h(n)}$. 
    \end{itemize}
    \item 
    If $m\in \sigma[B]$, then suppose $i\in \sfU_{fr}$ is such that $m = \sigma((\frak{p}, b))$ for $\frak{p}\in \sfR_i$. If  $m\in \im(\alpha)$, then if $\frak{p}\in \sfQ_i^\Theta$, then if $s\in \cdnd(\Theta_\frak{p})$ and $t\in \cdnd(\Theta_\frak{q})$ with $\frak{q}\in \MP_i$ and $\frak{q}>_{\MP} \frak{p}$, we have $\ell(\xi\circ \phi_f(s)\wedge \xi\circ \phi_f(s)) = m$. If $\frak{p}\in \sfR_i^\Theta\setminus \sfQ_i^\Theta$, then there is $n< \rm{ht}(\Theta)$ and $x\in Z^\Theta_n$ with $U^{\bZ^\Theta_m}(x) = i$. If $s\in \cdnd(\Theta_\frak{p})$, then $\ell(\xi\circ \phi_f(s)\wedge \sigma\circ \beta(x)) = m$. In both cases, we have $m\in \crit^\bY(c^\bY[S_{m+1}])$.

    For the converse, we note that at levels in $\sigma[B]$, age changes are not an issue, and by Fact~\ref{Fact:Y_Facts} $(2)$, $(3)$, and $(4)$, splitting can only occur due to the two situations outlined above. Hence if $m\not\in \im(\alpha)$, we cannot have $m\in \crit^\bY(c^\bY[S_{m+1}])$.
\end{itemize}
This concludes the proof of Theorem~\ref{Thm:Upper_Bound}.

\section{Conclusion and future directions}
\label{Sec:FutureWork}

This paper concludes the project of understanding big Ramsey degrees for finitely-constrained binary free amalgamation classes. Naturally, this leads to many open questions and future research directions, some of which we shall state below.

\subsection{Finiteness of the language}
For the \fr limit of the class of all finite complete graphs with edges colored by countably many colors, it is possible to construct unavoidable colorings using arbitrarily many colors. This shows that the assumption of a finite language is necessary in Theorem~\ref{thm:main_intro}.

On the other hand, it is possible to prove finite big Ramsey degrees for the homogeneous unconstrained hypergraph which has hyperedges of each finite arity~\cite{braunfeld2023big}. It suggests that $\omega$-categoricity might be a relevant property in the study of big Ramsey degrees; see also Subsection 7.3 of \cite{Zuc19}.

\subsection{Infinite sets of constraints} Our result assumes that the sets of constraints are finite, as infinite sets in general lead to infinite posets of ages and, consequently, our strategy yields infinite embedding shapes in such cases. We believe that this is an inherent property of big Ramsey degrees, not just an artifact of our proof strategies, and state the following conjecture.

\begin{conj}
Let $\mathcal{L}$ be a finite language, let $\mathcal{F}$  be an infinite collection of  finite irreducible 
$\mathcal{L}$-structures such that no member of $\mathcal{F}$ embeds to any other member of $\mathcal{F}$ and let $\mathbf{K}$
be the \fr limit of $\mathrm{Forb}(\mathcal{F})$. Then there is $\mathbf{A} \in \mathrm{Forb}(\mathcal{F})$ such that the big Ramsey degree of $\mathbf{A} \in \c{K}$ is infinite. Furthermore, the size of such an $\b{A}\in \c{K}$ depends only on $\c{L}$.
\end{conj}

Sauer in~\cite{SauerGraphs} considers free amalgamation classes of directed graphs and gives examples of infinite collections $\c{F}$ of finite tournaments such that in the class $\mathrm{Forb}(\c{F})$, the singleton graph still has finite big Ramsey degree. We remark that the statement of Theorem 1.1 from~\cite{SauerGraphs} is not correct, as it should refer to maximal paths through $P(1)$ rather than just antichains; however, the arguments in \cite{SauerGraphs} correctly show that if $\max(\vec{P}_0)$ is finite, then the singleton graph has big Ramsey degree at most $|\max(\vec{P}_0)|$ in $\mathrm{Forb}(\c{F})$. Thus for $\c{L}$ a binary language and $\c{F}$ an infinite set of finite irreducible structures, a natural guess is that some $\b{A}\in\mathrm{Forb}(\c{F})$ of size $2$ has infinite big Ramsey degree.

\subsection{Languages of higher arity} A natural generalization of the main result is to drop the requirement that all relations are of arity at most two. Unlike the binary case, there are not yet general upper bound results (and trying to obtain them is a very active area of research in which many of the authors are presently involved), which are a pre-requisite for obtaining characterizations of exact big Ramsey degrees.

Nevertheless, there exists an upper bound theorem for the generic 3-uniform hypergraph~\cite{Bal_3unifGraphs20} and also, more generally, unrestricted structures in languages with finitely many relations of every arity~\cite{braunfeld2023big}. There are two major differences between the 3-uniform hypergraph and  binary structures:
Firstly, instead of a single tree of types, a product of two trees naturally appears. The first tree is an analogue of the tree of $1$-types, while the second is a binary tree which can be thought of as a tree of ``higher order" types (see also~\cite{Bal_3unifGraphs19} for the intuition behind the second tree). This makes the analysis of embedding shapes and construction of diagonal substructures more complicated, because one has to control the interplay between meets in both trees. Secondly, since the language is ternary, the tree of $1$-types branches more and more at each successive level.

\begin{prob}\label{prob:3unif}
Characterize the big Ramsey degrees of the generic $3$-uniform hypergraph. Does the generic $3$-uniform hypergraph admit a big Ramsey structure?
\end{prob}

Once Problem~\ref{prob:3unif} is solved, the following two problems are natural generalizations.

\begin{prob}
\label{Prob:LargeArities}
Given a relational language $\c{L}$ with finitely many relational symbols of each arity, characterize the big Ramsey degrees for the class $\c{K}$ of all finite $\c{L}$-structures. Letting $\b{K} = \flim(\c{K})$, does $\b{K}$ admit a big Ramsey structure?
\end{prob}

An upper bound for these kinds of structures appeared in~\cite{braunfeld2023big}. We believe that a characterization is possible with the present tools and methodology and we expect that the result will naturally generalize the situation for the 3-uniform hypergraph.

\begin{prob}\label{prob:higher_arities}
Let $\mathcal{L}$ be a finite relational language, let $\mathcal{F}$  be a finite collection of  finite irreducible 
$\mathcal{L}$-structures and let $\c{K} = \mathrm{Forb}(\mathcal{F})$. Characterize the big Ramsey degrees of $\c{K}$. Letting $\b{K} = \flim(\c{K})$, does $\b{K}$ admit a big Ramsey structure?
\end{prob}

As has already been mentioned, a key ingredient for a solution to this problem is missing, namely a general upper bound. A sufficient condition (a stronger form of amalgamation) has been announced in~\cite{TypeAmalgEurocomb} where it is also conjectured that it is also necessary for finiteness of big Ramsey degrees.

\subsection{Strong amalgamation classes} Another obvious condition in the main theorem which can be relaxed is the requirement that our classes have free amalgamation. Instead, one could ask for \emph{relational} classes which have strong amalgamation (every amalgamation class can be expanded to a strong amalgamation class by adding functions which represent closures).

For certain collections of strong amalgamation classes with finitely many relations each of arity at most two, exact big Ramsey degrees have been successfully  characterized. 
Devlin \cite{Devlin} characterizes the exact big Ramsey degrees for finite linear orders, and Laflamme, Sauer, and Vuksanovic \cite{LSV} characterize the
exact big Ramsey degrees for {\em unconstrained} binary strong amalgamation classes, including finite graphs, finite digraphs, finite tournaments, and similar classes of structures with finitely many binary relations. 
Unconstrained classes have the property that there are no age changes in their coding trees, and 
hence, their exact big Ramsey degrees are characterized simply by weakly diagonal structures. Similar characterizations has been found for the rationals with an equivalence relation with finitely many dense equivalence classes 
and the circular directed graphs (see  \cite{LNVTS} and \cite{Barbosa}).

Recently, Coulson, Dobrinen, and Patel  in \cite{CDP1, CDP2}  formulated an amalgamation property called SDAP,  a strengthening of  strong amalgamation,  and proved that its slightly stronger version, SDAP$^+$,  implies that the exact big Ramsey degrees are characterized by weakly diagonal structures. 
The framework in   \cite{CDP1, CDP2}  encompasses  the results  in \cite{Devlin},  \cite{LNVTS}, and  \cite{LSV} (minus $\mathbf{S}(2)$) 
and provides exact big Ramsey degrees for new classes including 
  unconstrained strong amalgamation classes with an additional linear order (for instance, ordered graphs and ordered tournaments), 
generic $k$-partite graphs (with or without an additional linear order), and finitely many nested  convex equivalence relations. The methods in \cite{CDP1, CDP2} also show that structures satisfying SDAP$^+$ in finite languages with relations of {\em any} arity are indivisible.  This includes hypergraphs omitting some finite set of $3$-irreducible substructures 
and their ordered versions.
We point out that 
the classes  Forb$(\mathcal{F})$ for which the main results of this paper hold do not satisfy SDAP whenever $\mathcal{F}$  contains an irreducible substructure of size three or greater. More generally, SDAP implies that the natural generalisations of the posets $P(\rho)$ are singletons for every sort $\rho$.

For some strong amalgamation classes 
in a binary language
not satisfying SDAP,
upper bound results already exist. In particular, using parameter space methods developed by Hubi\v cka~\cite{Bal_ForbCycle,Hubicka2020CS}, upper bounds have been obtained for the generic poset, for metric spaces in a finite language and for certain free superpositions of these and other classes. Exact big Ramsey degrees for the generic poset have been characterized by the authors of this paper~\cite{Balko2023}, see also~\cite{Bal_PO} (we remark that~\cite{Balko2023} also contains a short self-contained description of the big Ramsey degrees of the triangle-free graph). It is quite interesting that in this case the upper bound from~\cite{Hubicka2020CS} was not flexible enough and a stronger upper bound theorem had to be proved as well.

A natural next step in this direction would be to characterize the exact big Ramsey degrees for metric spaces of finite diameter with integer distances.

\begin{prob}\label{prob:Nmetric}
Let $n\geq 3$ and let $\c{M}_n$ denote the class of all finite metric spaces with distances $\{1,\ldots,n\}$. Characterize the big Ramsey degrees of $\c{M}_n$. Letting $\b{M}_n = \flim(\c{M}_n)$, does $\b{M}_n$ admit a big Ramsey structure?
\end{prob}

A subset of the authors has already done some work in this direction. In particular, besides splitting and coding, the age-change events come in three different flavors. For single types, each type has a \emph{diameter}, i.e.\ the largest distance between two vertices of the given type. At the beginning, this diameter is equal to $n$, but if at some point a neighbor of the type in distance $a$ is discovered, the diameter of the type cannot be larger than $2a$. Additionally, each pair of types has a \emph{lower bound} and \emph{upper bound} on the distances between realizations of these two types. For example, if a vertex with distance $a$ to one type and $b$ to the other is discovered, then we know that the minimum distance between realizations of these two types is at least $\lvert a-b\rvert$ and the maximum is at most $a+b$. We conjecture that the big Ramsey degrees will be characterized by precisely describing the interplay between these five types of interesting events.

\subsection{Enumerating big Ramsey degrees and asymptotics}
\label{SubSec:Asymptotics}

After characterizing the big Ramsey degrees of $\bA\in \cK$ by some combinatorial objects as is done in this paper, it remains to actually count these objects, which becomes an intriguing problem enumerative combinatorics. One notable result is that in the class of finite linear orders, the big Ramsey degree of the $k$-element linear order is precisely $\tan^{(2k-1)}(0)$. Similar results for  the class of finite linear orders with an equivalence relation with $n$ many  equivalence classes  were obtained in \cite{LNVTS}
and explicit formulas for all circular directed graphs were obtained in \cite{Barbosa}, see also~\cite{Larson}.

It seems likely that for most other cases, such nice explicit formulas do not exist. A more tractable problem would be to characterize the asymptotic growth of the big Ramsey degrees for certain distinguished sequences of structures from a given \fr class. For instance, the end of Example~\ref{Exa:3FreeDiaries} suggests the problem of understanding the growth of the function which sends the number $n$ to the big Ramsey degree of the $n$-element anti-clique in $\c{G}_3$; the same problem makes sense for any finitely-constrained binary free amalgamation class.

In fact, calculating the asymptotic growth of big Ramsey degrees might serve as a measure of how ``difficult" it is to show that a given \fr class has finite big Ramsey degrees. For instance, we conjecture that among binary relational classes, those with SDAP will have slower big Ramsey degree growth than classes with non-trivial forbidden substructures. And among classes with non-trivial forbidden substructures, we conjecture that the sizes of the constraints influence the growth. For example, the asymptotic growth of big Ramsey degrees for the class of finite $K_4$-free graphs seems to be much larger than that of the class of finite posets or that of the class of $K_3$-free graphs. This may suggest which sorts of Ramsey theorems are needed to show that a given class has finite big Ramsey degrees. There are three main types of Ramsey theorems which have successfully been applied to big Ramsey degrees, each having different strengths and weaknesses: Milliken's tree theorem, which generalizes the ordinary infinite Ramsey theorem and is most suitable for unconstrained structures; various coding tree theorems (see~\cite{CDP1, CDP2, Dob0, Dob, Zuc22}), which this paper makes use of (see also \cite{DobIfCoLog}); 
and the Carlson--Simpson theorem about parameter spaces, which has recently been applied in proving finite big Ramsey degrees for the generic poset and giving a new proof of finite big Ramsey degrees for the class of finite $K_3$-free graphs \cite{Hubicka2020CS}. This naturally raises interest in knowing what the limitations of each proof method are. For instance, coding tree methods can handle any finitely-constrained binary free amalgamation class, but cannot handle the class of finite posets. On the other hand, the parameter space method gives simple proofs for the class of finite $K_3$-free graphs, the class of finite posets, and some other classes defined by (not necessarily irreducible) constraints of size at most $3$, but it seems unable to handle even the class of $K_4$-free graphs. Perhaps the parameter space method necessarily constrains the asymptotic growth of big Ramsey degrees for those classes on which it is successful.

\subsection{Infinite-dimensional Ramsey theorems}
Ramsey theorems involving colorings of  infinite objects 
are considered {\em infinite-dimensional}.
A  subset $\mathcal{X}$ of the Baire space $[\omega]^{\omega}$ is  called {\em Ramsey} if  for each infinite set $X\subseteq \omega$, there is an infinite subset $Y\subseteq X$ such that $[Y]^{\omega}$ is either contained in $\mathcal{X}$ or disjoint from $\mathcal{X}$.
While the  Axiom of Choice can be used to construct an $\cX\subseteq [\omega]^\omega$ which is not Ramsey, there are positive results upon restricting to suitably definable  $\cX$
 In particular, the  Galvin--Prikry theorem \cite{GP} shows that 
 every Borel $\cX\subseteq [\omega]^\omega$ is Ramsey, and Ellentuck \cite{Ellentuck} strengthens this to characterize those $\cX\subseteq [\omega]^\omega$ which are \emph{completely} Ramsey in terms of a topology refining the usual metric topology on Baire space. 
 
 In \cite{KPT}, Kechris, Pestov, and Todorcevic suggested the following direction of research.

\begin{prob}
Which homogeneous structures admit generalizations of the Gal\-vin--Prikry and/or Ellentuck theorems?
\end{prob}

 They also ask for dynamical properties associated to such infinite-dimensional Ramsey theorems. 
 Big Ramsey degrees provide constraints on the types of theorems which are possible:
 Given an enumerated \fr limit $\mathbf{K}$,  
 subcopies of $\mathbf{K}$ may be given different colors depending on their diaries.
 Thus, for an exact analogue of the Galvin--Prikry theorem, one must restrict to subcopies of $\mathbf{K}$ with the same diagonal diary.
 The first theorem of this sort appeared 
 in \cite{DobRado}, 
 where Dobrinen proved that upon fixing a particular coding tree for the Rado graph, all subcopies of the Rado graph with similar induced coding subtrees satisfy a version of the Galvin--Prikry Theorem.
 This was extended and sharpened 
 in \cite{DobSDAP}, where Dobrinen proved Galvin--Prikry theorems for the classes of structures in \cite{CDP1, CDP2}.
 In a recent preprint building on the present work, Dobrinen and Zucker \cite{DZ2023} show that every strong big Ramsey structure for any finitely-constrained binary free amalgamation class satisfies a sharp version of the Galvin--Prikry theorem.

\section*{Acknowledgement}
The authors are indebted to the referee, whose insistence that we shorten the paper has led to a vastly improved presentation.

M.\ B.\ and J.\ H.\ were supported by the Center for Foundations of Modern Computer Science (Charles University project UNCE/SCI/004). 
N.\ D.\ was supported by National Science Foundation grant DMS-1901753. 
J.\ H.\ and M.\ K.\ were supported by the project 21-10775S of  the  Czech  Science Foundation (GA\v CR). This article is part of a project that has received funding from the European Research Council (ERC) under the European Union's Horizon 2020 research and innovation programme (grant agreement No 810115).
M.\ K.\ was supported by the Charles University Grant Agency (GA UK), project 378119.
L.\ V.\ was supported by Beatriu de Pin\'os BP2018, funded by the AGAUR (Government of Catalonia) and by the Horizon 2020 programme No 801370.
A.\ Z.\ was supported by National Science Foundation grant DMS-2054302 and NSERC grants RGPIN-2023-03269 and DGECR-2023-00412. 

\bibliography{andy_masterpiece.bib}
\end{document}